\documentclass[12pt]{amsart}
\usepackage[english]{babel}
\usepackage{amsmath}
\usepackage{amsthm}
\usepackage{amssymb}
\usepackage{mathrsfs}
\usepackage{enumerate}
\usepackage[notcite, final, notref]{showkeys}
\usepackage{amsfonts}
\usepackage{dsfont}
\usepackage{float}
\usepackage{tikz}
\usepackage[siunitx]{circuitikz}
\usetikzlibrary{arrows,calc,chains,shapes,dsp}
\def\C{\mathbb C}
\def\R{{\mathbb R}}
\def\T{\hfill{\triangledown\triangledown\triangledown}}
\newtheorem{Pa}{Paper}[section]
\newtheorem{Tm}[Pa]{{\bf Theorem}}
\newtheorem{La}[Pa]{{\bf Lemma}}

\newtheorem{Cy}[Pa]{{\bf Corollary}}
\newtheorem{Rk}[Pa]{{\bf Remark}}
\newtheorem{Pn}[Pa]{{\bf Proposition}}

\newtheorem{Ex}[Pa]{{\bf Example}}
\newtheorem{Dn}[Pa]{{\bf Definition}}
\usepackage{xcolor}

\title[Quantitatively Hyper-Positive Real Functions III]
{Quantitatively Hyper-Positive Real Rational Functions III}

\author[D. Alpay]{Daniel Alpay}
\address{(DA)
Faculty of Mathematics, Physics, and Computation\\
Schmidt College of Science and Technology\\
Chapman University\\
One University Drive
Orange, California 92866\\
USA}
\email{alpay@chapman.edu}
\thanks{Daniel Alpay thanks the Foster G. and Mary McGraw Professorship in
Mathematical Sciences, which supported this research.}

\author[I. Lewkowicz]{Izchak Lewkowicz}
\address{(IL) School of Electrical and Computer Engineering\\
Ben-Gurion University of the Negev\\ P.O.B. 653\\ Beer-Sheva, 84105\\
Israel}
\email{izchak@bgu.ac.il}

\begin{document}
\bibliographystyle{plain}

\begin{abstract}
Hyper-Positive Real, matrix-valued, rational functions are associated with absolute stability (the Lurie
problem). Here, quantitative subsets of Hyper-positive functions, related through nested inclusions, are
introduced. Structurally, this family of functions turns out to be matrix-convex and closed under
inversion. A state-space characterization of these functions through a corresponding Kalman-Yakubovich-Popov
Lemma, is given. Technically, the classical Linear Matrix Inclusions, associated with passive systems, are
here substituted by Quadratic Matrix Inclusions.
\end{abstract}
\maketitle


\date{today}
\tableofcontents

\section{Introduction and Main Results}
\setcounter{equation}{0}

Passivity is a fundamental property in the study of dynamical systems.
Thus, it has been extensively addressed in various frameworks.
Although passivity is essentially a {\em physical} property, it recently turned
out that at least in the framework of finite-dimensional, linear,
time-invariant systems, passivity can be characterized by {\em structural}
properties of a {\em whole family}, see \cite{Lewk2021a}, \cite{Lewk2020c} and \cite{Lewk2024a}.
Specifically, the set is matrix-convex, closed under inversion (or product among elements) and
in addition, maximal in some rigorous sense.
\smallskip

In this work we proceed on the path started in \cite{AlpayLew2021} and \cite{AlpayLew2024}
and further elaborated on the above observation while focusing on {\em strictly
dissipative}\begin{footnote}{Roughly speaking, where damping is
everywhere.}\end{footnote} systems, a proper subset of the passive ones. 

{\sl
\begin{quote}
It is emphasized that restricting the discussion from passive to dissipative systems, in the
physical sense, from mathematical structure point-of-view means that we parametrize
sub-families where, in a rigorous sense, one is ``more stable" than the other and in
particular, maximality is compromized.
\end{quote}
}

A word about notation:~ ${\C}_L$, $\overline{\C}_L$ and ${\C}_R$, $\overline{\C}_R~$ will be the open (closed)
 Left and Right halves of the complex plane, respectively.\\
For (constant) matrices, we denote by ($\mathbf{H}_q$), $\overline{\mathbf H}_q$ the set of $q\times q$
(non-singular) Hermitian (a.k.a self-adjoint). We shall say that $\overline{\mathbf H}_q$ is the closure of
the open set $\mathbf{H}_q~$. Skew-Hermitian matrices are denoted by $~i\overline{\mathbf H}_q~$. It is common
to consider $\overline{\mathbf H}_q$ and $i\overline{\mathbf H}_q$ as the matricial extensions of $\R$ and
$i\R$, respectively.
\smallskip

In a similar way, ($\overline{\mathbf P}_q$) and $\mathbf{P}_q$ will be the (closed) open subsets of positive
(semi)-definite matrices\begin{footnote}{$H\in\mathbf{P}_q~~ \Longleftrightarrow~~x^*Hx>0,~\forall~
\not=x\in\C^q$.}\end{footnote}. Sometimes, when dimensions are clear from the context we shall omit the
subscript and simply write $\mathbf{H}$ or $\mathbf{P}$.
\smallskip

For a single matrix, we sometimes adopt the short-hand notation, ($P\succcurlyeq 0$) $P\succ 0$ to indicate
($P\in\overline{\mathbf P}$) $P\in\mathbf{P}$. So far now for notation.
\smallskip

We here address matrix-valued real rational functions of a complex variable $s$.

\begin{Rk}
{\rm
To ease the physical interpretation, throughout this work we focus on {\em real rational}\begin{footnote}
{Physically, they model {\em lumped} systems, while {\em irrational} functions allow for {\em distributed}
systems. See e.g. \cite{Wohl1969}, \cite{ze1}, \cite{ze2}.}\end{footnote} functions $F(s)$. Yet, many of the
results apply to functions which are not necessarily rational, or with complex coefficients.
$\T$
}
\end{Rk}

The set $\mathcal{P}$, of positive real functions, defined below, classically serves as a model for passive,
continuous-time, linear, time-invariant systems, see e.g. \cite[Theorem 2.7.1]{AnderVongpa1973},
\cite[Section 3.18]{Belev1968}, \cite[Corollary 2.39]{BroLozaMasEge2020},
\cite[Section 6.3]{Khalil2000}, \cite[Proposition 1]{Will1976} and \cite[Theorem 11]{YoulCastCarl1959}.

\begin{Dn}\label{Dn:Positive}
{\rm 
We shall call a $m\times m$-valued rational function $F(s)$, {\em Positive}, (Real) i.e.
\mbox{$F\in\mathcal{P}$,} if $\forall s\in\C_R~$,  $F(s)$ is analytic and\begin{footnote}{The symbol
$\triangledown\triangledown\triangledown$ indicates the end of a definition, an example or a remark.
The symbol $\square$ indicates the end of a proof.}\end{footnote},
\begin{equation}\label{eq:Def_P}
F(s)+(F(s))^*\in\overline{\mathbf P}_m~.
\end{equation}
$\T$
}
\end{Dn}

Here is the first structural result.

\begin{Tm}\label{Tm:Set_Of_P_Functions}

{\bf a.}~
\cite[Proposition 4.1.1]{CohenLew2007}.
The set of $m\times m$-valued positive functions is a cone of real rational functions
and a maximal convex set of functions, closed under inversion, which are analytic in $\C_R~$.
\smallskip

{\bf b.}~
\cite[Theorem 5.3]{Lewk2021a}. The set of $\mathcal{P}$ functions, of all dimensions, is a maximal
matrix-convex cone of matrix-valued real rational functions, closed under inversion, which are analytic in
$\C_R~$.
\smallskip

Conversely, a maximal matrix-convex cone of
matrix-valued rational functions (of various dimensions), closed under inversion,
which are analytic in $\C_R~$ and containing the zero degree
function \mbox{$F_o(s)\equiv I_m$}, is the set $\mathcal{P}$.
\end{Tm}

As already mentioned, this work focuses on the subset of {\em dissipative} systems.
To formalize it, we need the following.

\begin{Dn}\label{Dn:Hyper_Positive}
{\rm Let \mbox{${\color{blue}T}$} be a matricial parameter so that\begin{footnote}{Meaning
$(I_m-{\color{blue}T})\in\mathbf{P}_m$ and
${\color{blue}T}\in\overline{\mathbf P}_m~$.}\end{footnote}
\[
I_m\succ{\color{blue}T}\succcurlyeq 0.
\]
We shall call a $m\times m$-valued rational function $F(s)$,
$\begin{smallmatrix}{\color{blue}T}\end{smallmatrix}$-{\em Hyper-Positive}, i.e.
\mbox{$F\in\mathcal{HP}_{\color{blue}T}$,} if\begin{footnote}{To ease the reading, when next
to functions \mbox{${\color{blue}T}$,} is written in a smaller font.}\end{footnote}
$\forall s\in\C_R~$, $F(s)$ is analytic and,
\begin{equation}\label{eq:HP_Delta}
F(s)+(F(s))^*\succcurlyeq\begin{smallmatrix}{\color{blue}T}\end{smallmatrix}+(F(s))^*
\begin{smallmatrix}{\color{blue}T}\end{smallmatrix}F(s).
\end{equation}
$\T$
}
\end{Dn}

\begin{Rk}\label{Rk:Supplay_Rate}
{\rm
{\bf a.}~ 
For perspective consider the following state-space system 
\[
\dot{x}=f(x, u)
\quad\quad
y=g(x, u)
\quad{\rm where}\quad\begin{smallmatrix}x\in\R^n\\~\\ u, y\in\R^m.\end{smallmatrix}
\]
where $u(t)$ and $y(t)$ are {\em input} and {\em output}, respectively. Recall that passivity
of a system, is often described by some inner product, denoted by \mbox{$\langle u, y\rangle$.}
See e.g.  \cite{ze1,ze2}. Following J.C. Willems, this quantity is referred to as the ``supply
rate" of energy from the (vector-valued) input $u$ to the (vector-valued) output $y$, see e.g.
\cite[Definition 2, Subsection 7.1]{Will1972a}, \cite[Eq. (2)]{HillMoyl1976}. For non-linear
systems account, see e.g. \cite[Section 2.2]{van_der_Schaft2017}. And for a recent extension,
see \cite{MorBinAstPar2024}.
\smallskip

To describe the framework here, consider for {\em real} vector-valued $u(t)$ and $y(t)$, the
following inner-product
relation\begin{footnote}{As the vectors $u$ and $y$ are real, here $(~)^*$ simply means the
transpose.}\end{footnote},
\begin{equation}\label{eq:Passivity_Dissipativity}
\int_0^t(y(\tau))^*u(\tau)d\tau\geq{\scriptstyle\gamma}+\int_0^t(u(\tau))^*{\color{blue}T}
u(\tau)d\tau+\int_0^t(y(\tau))^*{\color{blue}T}y(\tau)d\tau\quad\quad t\geq 0,
\end{equation}
where ${\color{blue}T}$, \mbox{$I_m\succ{\color{blue}T}\succcurlyeq 0$} and
\mbox{${\scriptstyle\gamma}\in(-\infty,~0]$,} are parameters. The integral in Eq.
\eqref{eq:Passivity_Dissipativity} can be replaced by an integral on the whole real line by multiplying $u$ by
a set characteristic function. Then Plancherel's identity
(we denote by $u(s)$, $y(s)$, where $s$
is a complex variable, the Laplace transform of $u(\tau)$, $y(\tau)$, where $\tau$ is real),
along with the relation $y(s)=F(s)u(s)$, will lead to 
\[
\int_{\mathbb R}(u(s))^*\left(~(F(s))^*+F(s)-T-(F(s))^*TF(s)~\right)u(s)ds\geq{\scriptstyle\gamma}.
\]
For ${\scriptstyle\gamma}=0$ and since $u$ is arbitrary, we get  \eqref{eq:HP_Delta} on the
real line. But $F$ has a real positive part in $\C_R$ and so one can use the maximum modulus principle (via
using the Cayley transform to reduce to bounded functions) to extend this inequality to  all of $\C_R~$.
\smallskip

Systems satisfying Eq. \eqref{eq:Passivity_Dissipativity} shall be here called :~ {\em dissipative}, 
when \mbox{${\color{blue}T}\not=0$}, and if \mbox{${\color{blue}T}=0$} (i.e. Eq. \eqref{eq:Def_P})
they will be referred to as {\em passive}.
\smallskip

In parts of the literature, see e.g. \cite[Definition 2.1, Theorem 2.81]{BroLozaMasEge2020}
\cite[Definition 3]{HillMoyl1976}, systems satisfying Eq. \eqref{eq:Passivity_Dissipativity} are called
``Very Strictly (or Strongly) Passive". Sometimes, functions satusfying Eq. \eqref{eq:HP_Delta} are
called ``Strong Strictly Positive Real" (SSPR), see e.g. \cite[Definition 2.78, Remark
3.17]{BroLozaMasEge2020}. In \cite[Theorem 2.81]{BroLozaMasEge2020} it is shown that for linear,
time-invariant systems, having the Very Strictly Passive property is equivalent to being SSPR.\\
A somewhat similar result appeared in \cite[Lemma 2.3]{SunKharShim1994}, where the family under
focus was called ``Extended Strictly Positive Real" (ESPR).
\bigskip

{\bf b.}~ 
Strictly speaking, Eq. \eqref{eq:HP_Delta} describes {\em Right}~
$\mathcal{HP}_{\color{blue}T}$ functions. {\em Left}~ 
$\mathcal{HP}_{\color{blue}T}$ functions are given by
matrix-valued functions $F(s)$ which
$\forall s\in\C_R~$,  $F(s)$ are analytic and,
\[
F(s)+(F(s))^*\succcurlyeq\begin{smallmatrix}{\color{blue}T}\end{smallmatrix}+F(s)
\begin{smallmatrix}{\color{blue}T}\end{smallmatrix}(F(s))^*.
\]
The two classes may be different, already in the setting of constant functions;
say take $F(s)\equiv \begin{smallmatrix}\frac{1}{4}\end{smallmatrix}
\left(\begin{smallmatrix}1&&1\\~\\0&&3\end{smallmatrix}\right)$ and
${\color{blue}T}=\begin{smallmatrix}\frac{1}{5}\end{smallmatrix}
\left(\begin{smallmatrix}2&&0\\~\\0&&3\end{smallmatrix}\right)$. Then it is easy to verify
that\begin{footnote}{The red superscript in $F^{\color{red}*}$, is to emphasize the distinction between
Left and Right.}\end{footnote},
\[
F+F^*\succcurlyeq{\scriptstyle\color{blue}T}+F^{\color{red}*}
{\scriptstyle\color{blue}T}F\quad{\rm but}\quad F+
F^*\not\succcurlyeq{\scriptstyle\color{blue}T}+F{\scriptstyle\color{blue}T}F^{\color{red}*}.
\]
{\em Left} $\mathcal{HP}_{\color{blue}T}$ functions were used in \cite[Sections 4-6]{AlpayLew2024} and
in \cite{AlpayLew2025a}. 
Throughout this work we focus on {\em Right}
$\mathcal{HP}_{\color{blue}T}$ functions. 
\bigskip

{\bf c.}~
In the special case where in Eq. \eqref{eq:HP_Delta}
\mbox{${\color{blue}T}=\begin{smallmatrix}{\color{blue}\beta}\end{smallmatrix}I_m$,} where
\mbox{$\begin{smallmatrix}{\color{blue}\beta}\end{smallmatrix}\in[0,~1)$} we
shall say that \mbox{$F\in\mathcal{HP}_{\color{blue}\beta}$} if
$\forall s\in\C_R~$, $F(s)$ is analytic and,
\begin{equation}\label{eq:Basic_Def_HP_Beta}
F(s)+{F(s)}^*\succcurlyeq\begin{smallmatrix}{\color{blue}\beta}\end{smallmatrix}(I_m+{F(s)}^*F(s)).
\end{equation}
In fact, throughout \cite{AlpayLew2021} and in \cite[Sections 1-3]{AlpayLew2024} the set 
$\mathcal{HP}_{\color{blue}\beta}$ 
was used.
}
$\T$
\end{Rk}
\smallskip

Already at this stage, the introduction of Hyper-Positive functions, may be viewed as
a quantitative refinement of the set $\mathcal{P}$. Specifically, we have the following.

\begin{Pn}\label{Pn:HP_Delta_Order}
Definition \ref{Dn:Hyper_Positive} induces a partial order among family of functions, namely if
\[
I_m\succ{\color{cyan}T_2}\succcurlyeq{\color{blue}T_1}\succcurlyeq 0
\]
then
\begin{equation}\label{eq:Partial_Order}
\mathcal{HP}_{\color{cyan}T_2}\subset\mathcal{HP}_{\color{blue}T_1}.
\end{equation}
In addition there is a boundary condition,
\[
{\color{blue}T_1}=0_{m\times m}\quad\Longrightarrow\quad\mathcal{HP}_{{\color{blue}T_1}=0}=\mathcal{P}.
\]
\end{Pn}
\smallskip

In the context of the set $\mathcal{HP}_{\color{blue}\beta}$ from Eq. \eqref{eq:Basic_Def_HP_Beta}
Proposition \ref{Pn:HP_Delta_Order} takes the following form:
\smallskip

\begin{Cy}\label{Cy:HP_beta_Order}
In the case of scalar weights it holds that
\[
1>\begin{smallmatrix}{\color{cyan}{\beta}_2}\end{smallmatrix}\geq
\begin{smallmatrix}{\color{blue}{\beta}_1}\end{smallmatrix}\geq 0\quad\Longrightarrow\quad
\mathcal{HP}_{\color{cyan}{\beta}_2}\subset\mathcal{HP}_{\color{blue}{\beta}_1}\subset
\mathcal{HP}_{{\color{blue}\beta}=0}=\mathcal{P}.
\]
\end{Cy}

This partial order is studied in Subsection \ref{Subsec:Analogy}
and illustrated in Figure \ref{Fig:Sub-Unit_Disk} below.
\bigskip

The parametrization of sets of rational functions, $\mathcal{HP}_{\color{blue}T}$ or
$\mathcal{HP}_{\color{blue}\beta}$ enables us to ``zoom-into" the set $\mathcal{P}$. As a first result
we have the following refinement of Theorem \ref{Tm:Set_Of_P_Functions}.

\begin{Tm}\label{Tm:Set_Of_HP_Functions}
For a given (constant Hermitian) matrix $\begin{smallmatrix}{\color{blue}T}\end{smallmatrix}$, where
\mbox{$I_m\succ\begin{smallmatrix}{\color{blue}T}\end{smallmatrix}\succcurlyeq 0$,} let the set
$\mathcal{HP}_{\color{blue}T}$
of $m\times m$-valued rational functions, be as in Definition \ref{Dn:Hyper_Positive}.

\begin{itemize}
\item[(i)~~~]{}
Let $F(s)$ be in \mbox{$\mathcal{HP}_{\color{blue}T}$.} Whenever its inverse \mbox{${F(s)}^{-1}$} is
well defined, it belongs to the same \mbox{$\mathcal{HP}_{\color{blue}T}$.} In particular, when
\mbox{$I_m\succ{\color{blue}T}\succ 0$,} the inverse \mbox{${F(s)}^{-1}$}, always exists. 
\smallskip

\item[(ii)~~]{}
When \mbox{$F(s)\in\mathcal{HP}_{\color{blue}T}$}, it is equivalent to having the function
\mbox{$\begin{smallmatrix}(I_m-{\color{blue}T}^2)^{-\frac{1}{2}}\end{smallmatrix}
(F(s^*))^*\begin{smallmatrix}(I_m-{\color{blue}T}^2)^{\frac{1}{2}}\end{smallmatrix}$}
in \mbox{$\mathcal{HP}_{\color{blue}T}$,} as well.
\smallskip

\item[(iii)~]{}
For an arbitrary unitary matrix $U\in\C^{m\times m},$ denote
\mbox{$\hat{F}(s):=\begin{smallmatrix}U^*\end{smallmatrix}F(s)\begin{smallmatrix}U\end{smallmatrix}$.}\\
Whenever \mbox{$F(s)\in\mathcal{HP}_{\color{blue}T}$} it implies that 
$\hat{F}(s)\in\mathcal{HP}_{U^*{\color{blue}T}U}$.\\
In the special case where \mbox{$U{\color{blue}T}={\color{blue}T}U$,} also
\mbox{$\hat{F}(s)\in\mathcal{HP}_{\color{blue}T}~$.}
\smallskip

\item[(iv)~]{}
The set \mbox{$\mathcal{HP}_{\color{blue}T}$} is convex.
\smallskip

\item[(v)~~]{}
For given \mbox{$\begin{smallmatrix}{\color{blue}\beta}\end{smallmatrix}\in[0,~1)$},
the set $\mathcal{HP}_{\color{blue}\beta}$ is \mbox{$~m$-matrix-convex.}
\end{itemize}
\end{Tm}

\begin{Rk}
{\rm
{\bf a.}~ Theorem \ref{Tm:Set_Of_HP_Functions} is proved in Section \ref{ Proof_of_HP_T_structure_theorem}
below.
\smallskip

{\bf b.}~ 
{\em Maximality}, is the key to the gap between Theorems \ref{Tm:Set_Of_P_Functions} and
\ref{Tm:Set_Of_HP_Functions}.  More precisely, there are basic differences between the cases where the
parameter ${\color{blue}T}$ is zero or non-singular.
\smallskip

{\bf c.}~ 
Item (i) appeared in item (I) of \cite[Proposition 4.4]{AlpayLew2024}. Another proof is provided below. 
\smallskip

{\bf d.}~ 
Item (ii) appeared in \cite[Proposition 4.7]{AlpayLew2024}, with a typographical error, 
and thus corrected here.
\smallskip

{\bf e.}~ 
Of particular interest is a boundary case, where on $i\R$, Eqs. \eqref{eq:HP_Delta},
\eqref{eq:Basic_Def_HP_Beta} hold with equality. This is the focus of
\cite{AlpayLew2025b}.
$\T$
}
\end{Rk}
\smallskip

It turns out that the fact that the structure of the subset of $\mathcal{HP}_{\color{blue}T}$ functions
is richer than that of $\mathcal{P}$ functions, is carried over to state-space realization setup. To
explore it, one needs to first resort to the celebrated Kalman-Yakubovich-Popov Lemma characterizing
$\mathcal{P}$ rational functions through the respective state space realization. To this end, recall
that a $m\times m$-valued rational function $F(s)$, with no pole at infinity, admits a realization
\begin{equation}\label{eq:Realization}
F(s)=C(sI_n-A)^{-1}B+D,
\quad\quad\quad
R_F={\footnotesize
\left(
\begin{array}{c|c}A&B\\ \hline C&D\end{array}\right)}. 
\end{equation}
Now, if there exists $H\in\mathbf{P}_n$ so that,
\begin{equation}\label{eq:Original_KYP}
\left(\begin{smallmatrix}-H&&0\\~\\0&&I_m\end{smallmatrix}\right)R_F+{R_F}^*
\left(\begin{smallmatrix}-H&&0\\~\\0&&I_m\end{smallmatrix}\right)
\in\overline{\mathbf P}_{n+m},
\end{equation}
then $F(s)$ is in $\mathcal{P}$. See e.g. \cite[Chapter 5]{AnderVongpa1973}, \cite[Chapter 7]{Belev1968}, 
\cite[Subsection 2.7.2]{BGFB1994}, \cite[Chapter 3]{BroLozaMasEge2020},
\cite[Subsection 7.1.1]{DullPaga2000}, \cite[Theorem 3]{Will1972b} and \cite[Proposition 2]{Will1976}.
\smallskip

Next, we are to specialize this result to the framework of quantitatively Hyper-Positive
functions.

\begin{Tm}\label{Tm:Kyp_Hyper_Pos_W}
Let $F(s)$ be a $m\times m$-valued rational function, with no pole at infinity. Let
\[
F(s)=C(sI_n-A)^{-1}B+D,\quad and \,\, let \quad\quad
R_F={\footnotesize
\left(\begin{array}{c|c}A&B\\ \hline C&D\end{array}\right)}\quad 
\]
be a corresponding realization.
\smallskip

\begin{itemize}
\item[(i)~~]{}
If there exist matrices \mbox{$I_m\succ{\color{blue}T}\succcurlyeq 0$}
and $H\in\mathbf{P}_n$, satisfying
\begin{equation}\label{eq:HP_Delta_KYP}
\left(\begin{smallmatrix}-H&&0\\~\\0&&I_m\end{smallmatrix}\right)R_F+{R_F}^*
\left(\begin{smallmatrix}-H&&0\\~\\0&&I_m\end{smallmatrix}\right)
\succcurlyeq\underbrace{\left(\begin{smallmatrix}C~&&D\\~\\0_{m\times n}&&I_m\end{smallmatrix}
\right)^*}_{{\Gamma}^*}
\left(\begin{smallmatrix}{\color{blue}T}&&0\\~\\0&&{\color{blue}T}\end{smallmatrix}\right)
\underbrace{\left(\begin{smallmatrix}C~&&D\\~\\0_{m\times n}&&I_m\end{smallmatrix}\right)}_{\Gamma}
\end{equation}
then the function $F(s)$ is $\begin{smallmatrix}{\color{blue}T}\end{smallmatrix}$-Hyper-Positive.\\
If the above realization is minimal, the converse is true as well.
\smallskip

\item[(ii)~~]{}
Whenever Eq. \eqref{eq:HP_Delta_KYP} holds, with \mbox{$I_m\succ{\color{blue}T}\succ 0$,}
and the realization is minimal, the \mbox{$(n+m)\times(n+m)$}
``matrix" $R_F$ is non-singular, and denote \mbox{$R_{\hat{F}}:={R_F}^{-1}$.} 
\smallskip

\item[(iii)~]{}
Let $\hat{F}(s)$ be a \mbox{$m\times m$-valued} rational function whose realization array is
$R_{\hat{F}}$ from the previous item. Then, $R_{\hat{F}}$ satisfies Eq. \eqref{eq:HP_Delta_KYP}
with the same $\begin{smallmatrix}{\color{blue}T}\end{smallmatrix}$ and the same $H$. In
particular $\hat{F}(s)$ belongs to the same $\mathcal{HP}_{\color{blue}T}$.
\end{itemize}
\end{Tm}
\smallskip

\begin{Rk}\label{Rk:KYP_HP}
{\rm
{\bf a.}~ For proof and further details, see Section \ref{Sec:Realizations} below.
\smallskip

{\bf b.}~
As a boundary case, when one substitutes \mbox{${\color{blue}T}=0$}, Eq. \eqref{eq:Original_KYP} is obtained,
i.e. we get the state-space characterization of the set $\mathcal{P}$. For $\mathcal{P}$ functions, assuming
non-singularity of $R_F$, the result of item (iii) appeared in \cite[Theorem 6.5]{Lewk2021a}.
\smallskip

{\bf c.}~
In the special case of $\mathcal{HP}_{\color{blue}\beta}$ (see Eq. \eqref{eq:Basic_Def_HP_Beta})
\eqref{eq:HP_Delta_KYP} is simplified to
\begin{equation}\label{eq:M_HP_beta}
\left(\begin{smallmatrix}-H&&0\\~\\0&&I_m\end{smallmatrix}\right)R_F
+{R_F}^*\left(\begin{smallmatrix}-H&&0\\~\\0&&I_n\end{smallmatrix}\right)
\succcurlyeq\begin{smallmatrix}{\color{blue}\beta}\end{smallmatrix}
\underbrace{
\left(\begin{smallmatrix}C~&&D\\~\\0_{m\times n}&&I_m\end{smallmatrix}\right)^*
}_{{\Gamma}^*}
\underbrace{
\left(\begin{smallmatrix}C~&&D\\~\\0_{m\times n}&&I_m\end{smallmatrix}\right)
}_{\Gamma}.
\end{equation}
\smallskip

This case has already appeared (in a different form) in \cite[Theorem 10]{SakSuz1996}, and without
proof in \cite[Lemma 3.15]{BroLozaMasEge2020}.
\smallskip

{\bf d.}~
Upon comparing Eq. \eqref{eq:Original_KYP} with \eqref{eq:M_HP_beta}, note that to guarantee that
$F(s)$ is in $\mathcal{P}$, the matrix on the right-hand side, must be positive semi-definite (of
any rank, including zero). Now, for
\mbox{$\begin{smallmatrix}{\color{blue}\beta}\end{smallmatrix}\in(0,~1)$,} the
$\mathcal{HP}_{\color{blue}\beta}$ case is more demanding: 
The rank of the matrix on
the right-hand side is {\em at least} \mbox{$~m+{\rm rank}(C)$.}
\smallskip

{\bf e.}~
A particular boundary case, where Eqs. \eqref{eq:HP_Delta_KYP} \eqref{eq:M_HP_beta} hold with equality, 
is addressed in \cite{AlpayLew2025b}.
}
$\T$
\end{Rk}
\smallskip

To gain perspective, we recall the following.

\begin{Dn}\label{Dn:SP}
{\rm
A $m\times m$-valued function $F(s)$ is said to be Strictly Positive (Real), $F\in\mathcal{SP}$, if
there exists $\begin{smallmatrix}\epsilon\end{smallmatrix}>0$ so that
$F(s-\begin{smallmatrix}\epsilon\end{smallmatrix})$ is in $\mathcal{P}$. See e.g.
\cite[Definition 2.58]{BroLozaMasEge2020}.
}

$\T$
\end{Dn}

One can now relate the family of $\mathcal{SP}$ functions to its hyper-positive subset.

\begin{Pn}\label{Pn:SP_and_HP}
The following are equivalent
\begin{itemize}
\item[(i)~~~]{}
Let $F(s)$ be a $m\times m$-valued $\mathcal{HP}_{\color{blue}T}$ with {\em some} ${\color{blue}T}$,
\mbox{$I_m\succ{\color{blue}T}\succ 0$.}
\smallskip

\item[(ii)~~]{}
Eq. \eqref{eq:Original_KYP} holds and its right-hand side is positive definite.
\smallskip

\item[(iii)~]{}
Let $F(s)$ be a $m\times m$-valued $\mathcal{SP}$ function so that
\mbox{$D:=\lim\limits_{s~\rightarrow~\infty}F(s)$} exists and is so that $(D+D^*)\in\mathbf{P}_m$.
\end{itemize}
\end{Pn}

Proposition \ref{Pn:SP_and_HP} is {\em qualitative}, i.e. the parameter
$\begin{smallmatrix}{\color{blue}T}\end{smallmatrix}$ is not specified. As suggested by the title of
this work, and illustrated by Proposition \ref{Pn:HP_Delta_Order} and Theorems 
\ref{Tm:Set_Of_HP_Functions}, \ref{Tm:Kyp_Hyper_Pos_W}, our approach is {\em quantitative}. 
\smallskip

We next examine realization of {\em families} of $\mathcal{HP}_{\color{blue}T}$ functions.

\begin{Pn}\label{Pn:Sets_Of_Realizations}
Consider the framework of Theorem \ref{Tm:Kyp_Hyper_Pos_W}.

\begin{itemize}
\item[(i)~~]{}
Let the matricial parameters $I_m\succ{\color{blue}T}\succcurlyeq 0$ and $H\in\mathbf{P}_n$ be prescribed.
Then the set of all \mbox{$(n+m)\times(n+m)$} realization arrays $R_F$ satisfying Eq.
\eqref{eq:HP_Delta_KYP}, is convex.
\smallskip

\item[(ii)~]{}
Consider the set of all \mbox{$(n+m)\times(n+m)$} realization arrays $R_F$ satisfying Eq. \eqref{eq:M_HP_beta},
with $H=I_n$. Whenever the parameter is
\mbox{${\color{blue}T}=\begin{smallmatrix}{\color{blue}\beta}\end{smallmatrix}I_m$,}
\mbox{$\begin{smallmatrix}{\color{blue}\beta}\end{smallmatrix}\in[0,~1)$}
is prescribed, this set is \mbox{$n, m$-matrix-convex.}
\end{itemize}
\end{Pn}

\begin{Rk}
{\rm 
{\bf a.}~ In both parts of Proposition \ref{Pn:Sets_Of_Realizations}, minimality of the original realizations
is not assumed. Furthermore, even when the original realization arrays are all minimal, the resulting
realization is not necessarily minimal.
\smallskip

{\bf b.}~ In item (i) of Proposition \ref{Pn:Sets_Of_Realizations}, the case where ${\color{blue}T}=0$, i.e. of
positive functions, first appeared in \cite[Theorem 6.5]{Lewk2021a}.
$\T$
}
\end{Rk}

All results which to the best of our knowledge, have previously appeared (including ours)
are explicitly indicated.
\smallskip

The outline of this work are drawn by the table of contents, at the very opening.

\section{Background}
\label{Sec:Background}
\setcounter{equation}{0}

\subsection{Rational Functions in Quadratic Form}

In the sequel we find it convenient to re-write the family $\mathcal{HP}_{\color{blue}T}$ with
\mbox{$I_m\succ{\color{blue}T}\succcurlyeq 0$}, (originally defined in Eq. \eqref{eq:HP_Delta},
as functions $F(s)$ satisfying
\[
F(s)+(F(s))^*\succcurlyeq\begin{smallmatrix}{\color{blue}T}\end{smallmatrix}+
(F(s))^*\begin{smallmatrix}{\color{blue}T}\end{smallmatrix}F(s),
\]
$\forall s\in\C_R$),
in a quadratic form: Namely, all functions $F(s)$ satisfying,
\begin{equation}\label{eq:Quadratic_HP_Delta}
\mathcal{HP}_{\color{blue}T}=\left\{~\begin{smallmatrix}F(s)\end{smallmatrix}~:~
\left(\begin{smallmatrix}F(s)\\~\\I_m\end{smallmatrix}\right)^*
\left(\begin{smallmatrix}-{\color{blue}T}&&~~I_m\\~\\~~I_m&&-{\color{blue}T}
\end{smallmatrix}\right)
\left(\begin{smallmatrix}F(s)\\~\\I_m\end{smallmatrix}\right)
\in\overline{\mathbf P}_m\quad\forall s\in\C_R~\right\}.
\end{equation}
Substituting ${\color{blue}T}=0$, yields the quadratic
form of $\mathcal{P}$ functions from Eq. \eqref{eq:Def_P}, namely
\begin{equation}\label{eq:Quad_Def_P}
\mathcal{P}=\left\{~\begin{smallmatrix}F(s)\end{smallmatrix}~:~
\left(\begin{smallmatrix}F(s)\\~\\I_m\end{smallmatrix}\right)^*
\left(\begin{smallmatrix}0&&~~I_m\\~\\~~I_m&&~0\end{smallmatrix}\right)
\left(\begin{smallmatrix}F(s)\\~\\I_m\end{smallmatrix}\right)
\in\overline{\mathbf P}_m\quad\forall s\in\C_R~\right\}.
\end{equation}

\begin{Rk}\label{Rk:Quadratic_HP_T}
{\rm
Eq. \eqref{eq:Quadratic_HP_Delta} can be casted in the framework of supply-rate, addressed
in item {\bf a.} of Remark \ref{Rk:Supplay_Rate}. Indeed, employing again the relation $y=Fu$,
one can now write,
\[
\begin{matrix}
u^*
\left(\begin{smallmatrix}F(s)\\~\\I_m\end{smallmatrix}\right)^*
\left(\begin{smallmatrix}-{\color{blue}T}&&~~I_m\\~\\~~I_m&&-{\color{blue}T}
\end{smallmatrix}\right)
\left(\begin{smallmatrix}F(s)\\~\\I_m\end{smallmatrix}\right)
u
&=&
\left(\begin{smallmatrix}y\\~\\u\end{smallmatrix}\right)^*
\left(\begin{smallmatrix}-{\color{blue}T}&&~~I_m\\~\\~~I_m&&-{\color{blue}T}
\end{smallmatrix}\right)
\left(\begin{smallmatrix}y\\~\\u\end{smallmatrix}\right)
&~\\~\\~&=&
\overbrace{u^*y+y^*u}^{2{\rm Re} \langle u, y\rangle}-(u^*{\color{blue}T}u+y^*{\color{blue}T}y)&\geq 0.
\end{matrix}
\]
}
$\T$
\end{Rk}

An advantage of the above quadratic formulation, is that it enables us to technically unify the treatment
of rational functions and state-space realization.
\bigskip

To proceed, we need to resort to the classical Cayley transform.

\begin{Dn}\label{Dn:Cayley_Transform}
{\rm
{\bf a.}~ We denote by $\mathcal{C}(A)$ the Cayley transform of a matrix
$A\in\C^{n\times n}$, \mbox{$-1\not\in{\rm spec}(A)$},
\[
\mathcal{C}\left(A\right):=\left(I_n-A\right)\left(I_n+A\right)^{-1}
=-I_n+2\left(I_n+A\right)^{-1}.
\]
{\bf b.}~ Next we formally apply the Cayley transform to \mbox{$m\times m$-valued} rational functions 
\begin{equation}\label{eq:Cayley_F}
\mathcal{C}\left(F(s)\right):=(I_m-F(s))(I_m+F(s))^{-1}\quad\quad
{\rm det}(I_m+F(s))\not\equiv 0.
\end{equation}
$\T$
}
\end{Dn}

Recall that the Cayley transform is involutive in the sense that,
whenever well defined,
\[
\mathcal{C}\left(\mathcal{C}\left(A\right)\right)=A.
\]
Recall also now that $F\in\mathcal{P}$ if and only if
\[
G(s)=\mathcal{C}\left(F(s)\right)
\]
(see Eq. \eqref{eq:Cayley_F}),
belongs to $\mathcal{B}$, the family of Bounded (Real) functions, given by,
\begin{equation}\label{eq:Quad_Def_B}
\mathcal{B}=\left\{~\begin{smallmatrix}G(s)\end{smallmatrix}~:~
\left(\begin{smallmatrix}G(s)\\~\\I_m\end{smallmatrix}\right)^*
\left(\begin{smallmatrix}-I_m&&0\\~\\0&&I_m\end{smallmatrix}\right)
\left(\begin{smallmatrix}G(s)\\~\\I_m\end{smallmatrix}\right)\in\overline{\mathbf P}_m
\quad\forall s\in\C_R~\right\}.
\end{equation}
This function family will be used in the sequel.

\subsection{An analogy with $\mathbb{C}$}
\label{Subsec:Analogy}

We first examine open disks, in $\mathbb{C}$, of the form
\[
\mathbb{D}({\scriptstyle{\rm Center}},~{\scriptstyle{\rm Radius}})
=\{c\in\C~:~{\scriptstyle{\rm Radius}}>|c-{\scriptstyle{\rm Center}}|~\}
\quad\quad
\begin{smallmatrix}{\rm Center}&\in\C\\~\\
{\rm Radius}&>0.
\end{smallmatrix}
\]
Recall that the Cayley transform (Definition \ref{Dn:Cayley_Transform}) forms a bijection between
\mbox{$\mathbb{D}(0,~1)$,} the open unit disk, 
and $\mathbb{C}_R$, the
open right half of the complex plane, i.e.
\[
\mathbb{D}(0,~1)=\mathcal{C}\left(\C_R\right).
\]
To refine the analysis, a scalar parameter, \mbox{${\scriptstyle\beta}\in[0,~1)$} is introduced, and
we examine disks of the form~~ \mbox{$\mathbb{D}_{\rm Center}({\scriptstyle\beta}):=
\mathbb{D}\left(0+i0,~{\scriptstyle\frac{\sqrt{1-\beta}}{\sqrt{1+\beta}}}\right)$.}~
Clearly, for \mbox{${\scriptstyle\beta}=0$,} the unit disk is recovered. 
\smallskip

Under the Cayley transform these disks are mapped to~~
\mbox{$\mathbb{D}_{\rm Inv}({\scriptstyle\beta}):=
\mathbb{D}\left({\scriptstyle\frac{1}{\beta}}+i0,~
{\scriptstyle\frac{\sqrt{1-{\beta}^2}}{\beta}}\right)$.}\\
Expectedly, $\mathbb{C}_R$ is recovered for \mbox{${\scriptstyle\beta}=0$.~} It turns out that under
inversion\begin{footnote}{$(\mathbb{D}({\rm Center},~{\rm Radius}))^{-1}:=\{\frac{1}{c}~:~
c\in\mathbb{D}({\rm Center},~{\rm Radius})~\}$, whenever defined.}\end{footnote}
, each disk is mapped {\em onto} itself (and hence the subscript ``Inv"). This can be shown
in two ways (a) directly, as in \cite[Lemma 3.2 (iii)]{Lewk2024a}, or (b) by combining the three
following facts that (i) \mbox{$\mathbb{D}_{\rm Center}({\scriptstyle\beta})=-
\mathbb{D}_{\rm Center}({\scriptstyle\beta})$,} (ii) whenever the Cayley transform is well defined.
\mbox{$\mathcal{C}(-X)=\left(\mathcal{C}(X)\right)^{-1}$} and (iii) Eq.
\eqref{eq:Classical_Stein_Lypunov}, below.\\
In Figure  \ref{Fig:Sub-Unit_Disk}, these disks, along with their image under the Cayley transform,
are illustrated, where the color is preserved.

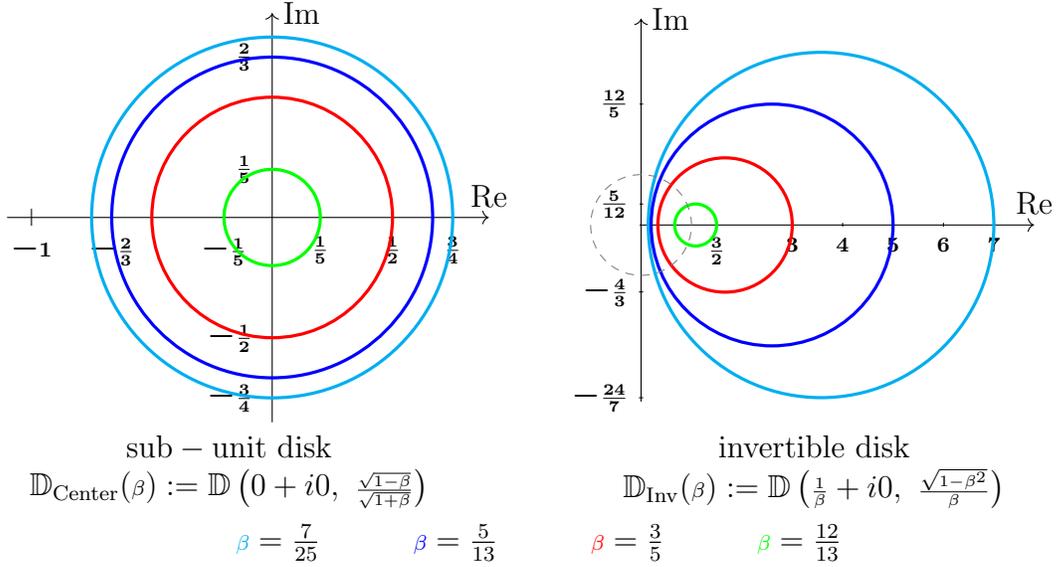
\begin{figure}[H]
\begin{minipage}[b]{0.38\linewidth}
\begin{tikzpicture}[scale=3.2,cap=round]
\tikzstyle{axes}=[]
\tikzstyle{important line}=[very thick]
\tikzstyle{information text}=[rounded corners,fill=red!10,inner sep=1ex]
\begin{scope}[style=axes]
\
\draw[->] (-1.10,0) -- (0.9,0) node[above] {Re};
\draw[->] (0,-0.85) -- (0,0.85) node[right] {Im};

\foreach \x/\xtext in{
-1/{\mbox{\boldmath$-{\scriptstyle 1}$}},
-0.666666/{\mbox{\boldmath$-{\scriptstyle\frac{2}{3}}$}},
-0.2/{\mbox{\boldmath$-{\scriptstyle\frac{1}{5}}$}}, 
0.2/{\mbox{\boldmath${\scriptstyle\frac{1}{5}}$}}, 
0.5/{\mbox{\boldmath${\scriptstyle\frac{1}{2}}$}},
0.75/{\mbox{\boldmath${\scriptstyle\frac{3}{4}}$}}
}
\draw[xshift=\x cm] (0pt,1pt) -- (0pt,-1pt) node[below,fill=white]
{$\xtext$}; 

\foreach \y/\ytext in{
-0.75/{\mbox{\boldmath$-{\scriptstyle\frac{3}{4}}$}},
-0.5/{\mbox{\boldmath$-{\scriptstyle\frac{1}{2}}$}},
0.2/{\mbox{\boldmath${\scriptstyle\frac{1}{5}}$}}, 
0.666666/{\mbox{\boldmath${\scriptstyle\frac{2}{3}}$}}
	}
\draw[yshift=\y cm] (1pt,0pt) -- (-1pt,0pt) node[left,fill=white]
{$\ytext$};
\end{scope}
\draw[arrows=->,style=important line, cyan] (0,0) circle (0.75);
\draw[arrows=->,style=important line, blue] (0,0) circle (0.666666666);
\draw[arrows=->,style=important line, red] (0,0) circle (0.5);
\draw[arrows=->,style=important line, green] (0,0) circle (0.2);
\end{tikzpicture}
\begin{center}
$\begin{matrix}
{\rm sub-unit~disk}
\\
\mathbb{D}_{\rm Center}({\scriptstyle\beta}):=
\mathbb{D}\left(0+i0,~{\scriptstyle\frac{\sqrt{1-\beta}}{\sqrt{1+\beta}}}\right)
\end{matrix}$
\end{center}
\end{minipage}
\quad\quad\quad
\begin{minipage}[b]{0.42\linewidth}
\begin{tikzpicture}[scale=0.67,cap=round]
\tikzstyle{axes}=[]
\tikzstyle{important line}=[very thick]
\tikzstyle{information text}=[rounded corners,fill=red!10,inner sep=1ex]
\begin{scope}[style=axes]
\
\draw[->] (-0.3,0) -- (7.8,0) node[above] {Re};
\draw[->] (0,-3.5) -- (0,4.1) node[right] {Im};

\foreach \x/\xtext in {
1.5/{\mbox{\boldmath${\scriptstyle\frac{3}{2}}$}}, 
3/{\mbox{\boldmath${\scriptstyle 3}$}},
4/{\mbox{\boldmath${\scriptstyle 4}$}},
5/{\mbox{\boldmath${\scriptstyle 5}$}},
6/{\mbox{\boldmath${\scriptstyle 6}$}},
7/{\mbox{\boldmath${\scriptstyle 7}$}}
}
\draw[xshift=\x cm] (0pt,1pt) -- (0pt,-1pt) node[below,fill=white]
{$\xtext$}; 

\foreach \y/\ytext in {
-3.43/{\mbox{\boldmath$-{\scriptstyle\frac{24}{7}}$}},
-1.33/{\mbox{\boldmath$-{\scriptstyle\frac{4}{3}}$}},
0.41666/{\mbox{\boldmath${\scriptstyle\frac{5}{12}}$}},
2.4/{\mbox{\boldmath${\scriptstyle\frac{12}{5}}$}}
}
\draw[yshift=\y cm] (1pt,0pt) -- (-1pt,0pt) node[left,fill=white]
{$\ytext$};
 \end{scope}
\draw[arrows=->,style=important line, cyan] (25/7,0) circle (24/7);
\draw[arrows=->,style=important line, blue] (2.6,0) circle (2.4);
\draw[arrows=->,style=important line, red] (5/3,0) circle (4/3);
\draw[arrows=->,style=important line, green] (13/12,0) circle (5/12);
\draw[dashed, gray] (0,0) circle (1);
\end{tikzpicture}
\begin{center}
$\begin{matrix}
{\rm invertible~disk}
\\
\mathbb{D}_{\rm Inv}({\scriptstyle\beta}):=
\mathbb{D}\left({\scriptstyle\frac{1}{\beta}}+i0,~
{\scriptstyle\frac{\sqrt{1-{\beta}^2}}{\beta}}\right)
\end{matrix}$
\end{center}
\end{minipage}
\smallskip

\begin{center}
${\scriptstyle\color{cyan}\beta}=\frac{7}{25}\quad\quad\quad
{\scriptstyle\color{blue}\beta}=\frac{5}{13}\quad\quad\quad
{\scriptstyle\color{red}\beta}=\frac{3}{5}\quad\quad\quad
{\scriptstyle\color{green}\beta}=\frac{12}{13}
$
\end{center}
\caption{Sub-Unit Disks and their Image under the Cayley Transform}
\label{Fig:Sub-Unit_Disk}
\end{figure}
\smallskip

As already pointed out, the disks in the two parts of in Figure \ref{Fig:Sub-Unit_Disk}
(including the boundary case, \mbox{${\scriptstyle\beta}=0$}), are
related through the Cayley transform, i.e.
\begin{equation}\label{eq:Cayley_Disks}
\mathbb{D}_{\rm Inv}({\scriptstyle\beta})=\mathcal{C}\left(
\mathbb{D}_{\rm Center}({\scriptstyle\beta})\right)
\end{equation}
We now extend Eq. \eqref{eq:Cayley_Disks} to the framework of $n\times n$ matrices. First, using the
Cayley transform for a pair of  matrices $A$ and $\hat{A}$ (see Definition \ref{Dn:Cayley_Transform}) 
denote,
\begin{equation}\label{eq:Cayley_Ttansform}
\hat{A}=\mathcal{C}(A).
\end{equation}
In this framework, the unit disk is substituted by the Stein inclusions. Specifically,
for a non-singular $n\times n$ Hermitian parameter $H$, with the above $A$ and $\hat{A}$,
one has the following equivalence between the Stein and the Lyapunov inclusions,
\begin{equation}\label{eq:Classical_Stein_Lypunov}
H-\hat{A}^*H\hat{A}\succcurlyeq 0,\quad\quad\quad HA+A^*H\succcurlyeq 0.
\end{equation}
In particular, whenever $H$ is positive definite, the spectrum of $A$ is in the closed right-half
plane and the spectrum of $\hat{A}$ is in the closed unit disk. For more details see e.g.
\cite{Ando2001}, \cite{Tau1964}.\\
For completeness we point out that having in the Lyapunov equation,
\mbox{$HA+A^*H\succ 0$,} \mbox{$-H\succ 0$,} it implies that
\mbox{${\rm spec}(A)\subset\C_L$,,} and then $A$ is said to be ``Hurwitz stable", see e.g.
\cite{Belev1968}, \cite{BroLozaMasEge2020}, \cite{DullPaga2000},
\cite{Gant1971} and \cite{Zhou1996}.
\smallskip

Now, the above scalar parameter ${\scriptstyle\beta}$, is extended to ${\color{blue}T}$,
\mbox{$I_n\succ{\color{blue}T}\succcurlyeq 0$.} Then, an equivalence between the following,
\begin{equation}\label{eq:Hyper_Stein_Lypunov}
H-\hat{A}^*H\hat{A}\succcurlyeq{\color{blue}T}+\hat{A}^*{\color{blue}T}\hat{A},\quad\quad\quad
HA+A^*H\succcurlyeq{\color{blue}T}+A^*{\color{blue}T}A,
\end{equation}
(quantitative) Hyper-Stein and Hyper-Lyapunov inclusions is obtained. As before, for ${\color{blue}T}=0$
the classical case from Eq. \eqref{eq:Classical_Stein_Lypunov} is recovered. These relations still
satisfy the Cayley transform in Eq. \eqref{eq:Cayley_Ttansform}. For details see \cite{Lewk2024a}. This
may be viewed as the matricial form of Eq. \eqref{eq:Cayley_Disks}.
\smallskip

One can now extend the above, from the framework of constant matrices in Eq. \eqref{eq:Hyper_Stein_Lypunov},
to matrix-valued functions, by taking $H=I_n$ and consider\footnote{To facilitate distinction, next to
{\em functions}, the parameter ${\color{blue}T}$ appears in a smaller font.},
\begin{equation}\label{eq:Def_F_and_Hat_F}
\begin{matrix}
I_m-{G(s)}^*G(s)\succcurlyeq\begin{smallmatrix}{\color{blue}T}\end{smallmatrix}
+{G(s)}^*\begin{smallmatrix}{\color{blue}T}\end{smallmatrix}G(s)\\~\\
F(s)+{F(s)}^*\succcurlyeq
\begin{smallmatrix}{\color{blue}T}\end{smallmatrix}
+{F(s)}^*\begin{smallmatrix}{\color{blue}T}\end{smallmatrix}F(s)
\end{matrix}
\quad\quad\quad s\in\mathbb{C}_R~.
\end{equation}
Note that the functions $F(s)$ and $G(s)$ in Eq. \eqref{eq:Def_F_and_Hat_F},
are related through the Cayley transform, 
i.e.
\begin{equation}\label{eq:Cayley_HB_T}
G(s)=\mathcal{C}\left(F(s)\right),
\end{equation}
as defined in Eq. \eqref{eq:Cayley_F}. Now, the quadratic forms corresponding to 
Eq. \eqref{eq:Def_F_and_Hat_F} are: The lower equation is given by
Eq. \eqref{eq:Quadratic_HP_Delta}; the upper equation is given by
\begin{equation}\label{eq:Quad_Def_HB_T}
\mathcal{HB}_{\color{blue}T}=\left\{~\begin{smallmatrix}G(s)\end{smallmatrix}~:~
\left(\begin{smallmatrix}G(s)\\~\\I_m\end{smallmatrix}\right)^*
\left(\begin{smallmatrix}-(I_m+{\color{blue}T})&&0\\~\\0&&I_m+{\color{blue}T}\end{smallmatrix}\right)
\left(\begin{smallmatrix}G(s)\\~\\I_m\end{smallmatrix}\right)\in\overline{\mathbf P}_m
\quad\quad\forall s\in\C_R~\right\}.
\end{equation}
Namely, a quantitatively Hyper-Bounded version of Eq. \eqref{eq:Quad_Def_B}.
\bigskip

We next briefly return to the disks and illustrate how with a given
\mbox{$\mathbb{D}({\scriptstyle{\rm Center}},~{\scriptstyle{\rm Radius}})$,} contained in some
\mbox{$\mathbb{D}_{\rm Inv}({\scriptstyle\beta})$,} ${\scriptstyle\beta}\in(0.~1)$, one can define
{\em three different} one-to-one maps to disks, within $\mathbb{D}_{\rm Center}({\scriptstyle\beta})$,
with the same ${\scriptstyle\beta}$. Take for example, 
\mbox{${\color{red}\mathbb{D}_0}:=\mathbb{D}(1,~1)$.} The smallest
$\mathbb{D}_{\rm Inv}({\scriptstyle\beta})$ containing $\mathbb{D}(1,~1)$, is with
\mbox{${\scriptstyle\beta}={\scriptstyle\frac{3}{5}}$.}
Then
take \mbox{${\color{blue}\mathbb{D}_1}=\mathbb{D}(-\frac{1}{4},~\frac{1}{4})$,}
\mbox{${\color{blue}\mathbb{D}_2}=\mathbb{D}(\frac{1}{8},~\frac{3}{4})$,}
\mbox{${\color{blue}\mathbb{D}_3}=\mathbb{D}(-\frac{1}{8},~\frac{3}{4})$.}
It turns out that 
the smallest $\mathbb{D}_{\rm Center}({\scriptstyle\beta})$ containing 
each of these three blue disks, is the same, with
\mbox{${\scriptstyle\beta}={\scriptstyle\frac{3}{5}}$.} Furthermore,
it is easy to verify that 
\[
{\color{blue}\mathbb{D}_1}=\mathcal{C}({\color{red}\mathbb{D}_0})\quad\quad\quad{\color{blue}\mathbb{D}_2}
={\scriptstyle\frac{\beta}{1+\beta}}\cdot{\color{red}\mathbb{D}_0}-{\scriptstyle\frac{1}{1+\beta}}
\quad\quad\quad{\color{blue}\mathbb{D}_3}={\scriptstyle\frac{1}{1+\beta}}-{\scriptstyle\frac{\beta}{1+\beta}
}\cdot{\color{red}\mathbb{D}_0}~.
\]
In principle, with an abuse of notation, one may view Figure \ref{Fig:Two_Maps} as representing these four
disks (${\color{blue}\mathbb{D}_1}$ upper-left, ${\color{blue}\mathbb{D}_2}$ lower-left). In addition, in
Figure \ref{Fig:Two_Maps}, in dashed curves, there is a red
\mbox{$\mathbb{D}_{\rm Inv}({\scriptstyle\beta})$,} and three blue
\mbox{$\mathbb{D}_{\rm Center}({\scriptstyle\beta})$,} all with
\mbox{${\scriptstyle\beta}={\scriptstyle\frac{3}{5}}$.}
\smallskip

We next extend the above disks illustration, to the framework of $m\times m$-valued rational functions:
From Eqs. \eqref{eq:Def_F_and_Hat_F}, \eqref{eq:Cayley_HB_T}, \eqref{eq:Quad_Def_HB_T} we know that whenever
$F(s)$ is a given function within $\mathcal{HP}_{\color{blue} T}$, for some
\mbox{$I_m\succ{\color{blue}T}\succcurlyeq 0$,} \mbox{$G_1:=\mathcal{C}(F)$} is a
$\mathcal{HB}_{\color{blue} T}$ function, with the same ${\color{blue}T}$. It turns out that, when
\mbox{$I_m\succ{\color{blue}T}\succ 0$,} (i.e. ${\color{blue}T}$ is {\em non-singular}), with the same
$F$ one can write down two additional one-to-one maps to functions within the same
$\mathcal{HB}_{\color{blue} T}$, say $G_2$, $G_3$. Indeed, take 
\begin{equation}\label{eq:Map_HP_HB}
\begin{matrix}
F(s)&=&\begin{smallmatrix}{\color{blue}T}^{-1}+(I_m+{\color{blue}T}^{-1})^{\frac{1}{2}}\end{smallmatrix}
G_2(s)\begin{smallmatrix}(I+{\color{blue}T}^{-1})^{\frac{1}{2}}\end{smallmatrix}
\\~\\
F(s)&=&\begin{smallmatrix}{\color{blue}T}^{-1}-(I_m+{\color{blue}T}^{-1})^{\frac{1}{2}}\end{smallmatrix}
G_3(s)\begin{smallmatrix}(I+{\color{blue}T}^{-1})^{\frac{1}{2}}\end{smallmatrix}
\\~\\
G_2(s)&=&\begin{smallmatrix}(I_m+{\color{blue}T}^{-1})^{-\frac{1}{2}}\end{smallmatrix}(F(s)-
\begin{smallmatrix}{\color{blue}T}^{-1}\end{smallmatrix})
\begin{smallmatrix}(I_m+{\color{blue}T}^{-1})^{-\frac{1}{2}}\end{smallmatrix}
\\~\\
G_3(s)&=&\begin{smallmatrix}(I_m+{\color{blue}T}^{-1})^{-\frac{1}{2}}\end{smallmatrix}(
\begin{smallmatrix}{\color{blue}T}^{-1}\end{smallmatrix}-F(s))
\begin{smallmatrix}(I_m+{\color{blue}T}^{-1})^{-\frac{1}{2}}\end{smallmatrix}.
\end{matrix}
\end{equation}
Exploiting the symmetry $G^*G=(-G)^*(-G)$, we have $G_2=-G_3$.
\smallskip

Note that, in contrast to the Cayley transform, the map in Eq. \eqref{eq:Map_HP_HB}, is affine.
In \cite{AlpayLew2025a} and \cite{AlpayLew2025b}, the functional version of this affine map
is exploited to explore the structure of $\mathcal{HP}_{\color{blue}T}$ functions.\\
The functions $F(s)$, $G_1(s)=\mathcal{C}(F)$, $G_2(s)$ and $G_3(s)$, are illustrated for 
${\color{blue}T}={\scriptstyle\beta}I_m$, in Figure \ref{Fig:Two_Maps}.

\begin{figure}[H]
\centering
\begin{minipage}[b]{0.38\linewidth}
 \begin{tikzpicture}[scale=2.3,cap=round]
    \tikzstyle{axes}=[]
    \tikzstyle{important line}=[very thick]
    \tikzstyle{information text}=[rounded corners,fill=red!10,inner sep=1ex]
    \begin{scope}[style=axes]
  \
      \draw[->] (-0.7,0) -- (0.7,0) node[above] {Re};
      \draw[->] (0,-0.6) -- (0,0.7) node[right] {Im};
      \foreach \x/\xtext in{
  -0.5/{\mbox{\boldmath$-{\scriptstyle\frac{1}{2}}$}},
      -0.25/{\mbox{\boldmath$-{\scriptstyle\frac{1}{4}}$}}, 
        0.25/{\mbox{\boldmath${\scriptstyle\frac{1}{4}}$}},
        0.5/{\mbox{\boldmath${\scriptstyle\frac{1}{2}}$}}
}
\draw[xshift=\x cm] (0pt,1pt) -- (0pt,-1pt) node[below,fill=white]
              {$\xtext$}; 

      \foreach \y/\ytext in{
        -0.5/{\mbox{\boldmath$-{\scriptstyle\frac{1}{2}}$}},
        -0.25/{\mbox{\boldmath$-{\scriptstyle\frac{1}{4}}$}},
         0.25/{\mbox{\boldmath${\scriptstyle\frac{1}{4}}$}},
        0.5/{\mbox{\boldmath${\scriptstyle\frac{1}{2}}$}}
}
       \draw[yshift=\y cm] (1pt,0pt) -- (-1pt,0pt) node[left,fill=white]
              {$\ytext$};
    \end{scope}
   \draw[arrows=->,style=important line, blue] (-0.25,0) circle (0.25);
\draw[color=blue, thick] [->] (-0.251,0.25) -- (-0.249,0.25){};
   \draw[arrows=->,style=dashed, blue] (0,0) circle (0.5);
 \end{tikzpicture}
$\begin{matrix}
g_1(s)=&\mathcal{C}({\color{red}f(s)})=&\frac{-1}{s+2}
\\~\\
\end{matrix}$
\end{minipage}
\quad\quad\quad
\begin{minipage}[b]{0.48\linewidth}
\begin{tikzpicture}[scale=0.85,cap=round]
    \tikzstyle{axes}=[]
    \tikzstyle{important line}=[very thick]
    \tikzstyle{information text}=[rounded corners,fill=red!10,inner sep=1ex]
    \begin{scope}[style=axes]
  \
      \draw[->] (-0.1,0) -- (3.8,0) node[above] {Re};
      \draw[->] (0,-1.6) -- (0,1.7) node[right] {Im};

      \foreach \x/\xtext in {
       0.333333/{\mbox{\boldmath$\scriptstyle\frac{1}{3}$}}, 
       1.0/{\mbox{\boldmath${\scriptstyle 1}$}}, 
       2.0/{\mbox{\boldmath${\scriptstyle 2}$}}, 
        3/{\mbox{\boldmath${\scriptstyle 3}$}}
}
\draw[xshift=\x cm] (0pt,1pt) -- (0pt,-1pt) node[below,fill=white]
              {$\xtext$}; 

      \foreach \y/\ytext in {
   -1.333333/{\mbox{\boldmath${\scriptstyle-\frac{4}{3}}$}},
1.00/{\mbox{\boldmath${\scriptstyle 1}$}}
}
       \draw[yshift=\y cm] (1pt,0pt) -- (-1pt,0pt) node[left,fill=white]
              {$\ytext$};
    \end{scope}
   \draw[arrows=->,style=important line, red] (2,0) circle (1.0);
\draw[color=red, thick] [->] (2.05,-1.0) -- (1.95,-1.0){};
   \draw[arrows=->,style=dashed, red] (5/3,0) circle (4/3);
\end{tikzpicture}
$
{\color{red}f(s)=\begin{smallmatrix}\frac{s+3}{s+1}\end{smallmatrix}}~~{\rm in}~~
\mathcal{HP}_{\color{blue}\beta}\quad
\begin{smallmatrix}{\color{blue}\beta}\end{smallmatrix}
=
\begin{smallmatrix}\frac{3}{5}\end{smallmatrix}
$
\end{minipage}

\begin{minipage}[b]{0.38\linewidth}
 \begin{tikzpicture}[scale=2.3,cap=round]
    \tikzstyle{axes}=[]
    \tikzstyle{important line}=[very thick]
    \tikzstyle{information text}=[rounded corners,fill=red!10,inner sep=1ex]
    \begin{scope}[style=axes]
  \
      \draw[->] (-0.7,0) -- (0.7,0) node[above] {Re};
      \draw[->] (0,-0.6) -- (0,0.7) node[right] {Im};

      \foreach \x/\xtext in{
  -0.5/{\mbox{\boldmath${\scriptstyle-\frac{1}{2}}$}},
        0.25/{\mbox{\boldmath${\scriptstyle\frac{1}{4}}$}},
        0.5/{\mbox{\boldmath${\scriptstyle\frac{1}{2}}$}}
}
\draw[xshift=\x cm] (0pt,1pt) -- (0pt,-1pt) node[below,fill=white]
              {$\xtext$}; 

      \foreach \y/\ytext in{
        -0.375/{\mbox{\boldmath${\scriptstyle-\frac{3}{8}}$}},
        0.5/{\mbox{\boldmath${\scriptstyle\frac{1}{2}}$}}
}
       \draw[yshift=\y cm] (1pt,0pt) -- (-1pt,0pt) node[left,fill=white]
              {$\ytext$};
    \end{scope}
   \draw[arrows=->,style=important line, blue] (0.125,0) circle (0.375);
\draw[color=blue, thick] [->] (0.126,-0.375) -- (0.124,-0.375){};
   \draw[arrows=->,style=dashed, blue] (0,0) circle (0.5);
 \end{tikzpicture}
\begin{center}
$
\begin{matrix}
g_2(s)=
\frac{1}{1+{\color{blue}\beta}}(\begin{smallmatrix}{\color{blue}\beta}\end{smallmatrix}{\color{red}f(s)}-1)
=\frac{2-s}{4(s+1)}
\end{matrix}
$
\end{center}
\end{minipage}
\quad\quad\quad
\begin{minipage}[b]{0.48\linewidth}
\begin{tikzpicture}[scale=2.3,cap=round]
    \tikzstyle{axes}=[]
    \tikzstyle{important line}=[very thick]
    \tikzstyle{information text}=[rounded corners,fill=red!10,inner sep=1ex]
    \begin{scope}[style=axes]
  \
      \draw[->] (-0.7,0) -- (0.7,0) node[above] {Re};
      \draw[->] (0,-0.6) -- (0,0.7) node[right] {Im};
      \foreach \x/\xtext in {
       -0.5/{\mbox{\boldmath${\scriptstyle-\frac{1}{2}}$}}, 
       0.25/{\mbox{\boldmath${\scriptstyle\frac{1}{4}}$}},
       0.5/{\mbox{\boldmath${\scriptstyle\frac{1}{2}}$}}
}
\draw[xshift=\x cm] (0pt,1pt) -- (0pt,-1pt) node[below,fill=white]
              {$\xtext$}; 

      \foreach \y/\ytext in {
   -0.375/{\mbox{\boldmath${\scriptstyle-\frac{3}{8}}$}},
   0.5/{\mbox{\boldmath${\scriptstyle\frac{1}{2}}$}}
}
       \draw[yshift=\y cm] (1pt,0pt) -- (-1pt,0pt) node[left,fill=white]
              {$\ytext$};
    \end{scope}
   \draw[arrows=->,style=important line, blue] (-0.125,0) circle (0.375);
\draw[color=blue, thick] [->] (-0.390,0.264) -- (-0.389,0.265){};
   \draw[arrows=->,style=dashed, blue] (0,0) circle (0.5);
\end{tikzpicture}
\begin{center}
$
g_3(s)=\frac{1}{1+{\color{blue}\beta}}(1-\begin{smallmatrix}{\color{blue}\beta}\end{smallmatrix}
{\color{red}f(s)})=\frac{s-2}{4(s+1)}
$
\end{center}
\end{minipage}
   \caption{Three $\mathcal{HB}_{\color{blue}\beta}$ associated with the same
$\mathcal{HP}_{\color{blue}\beta}$ function.}
   \label{Fig:Two_Maps}
\end{figure}
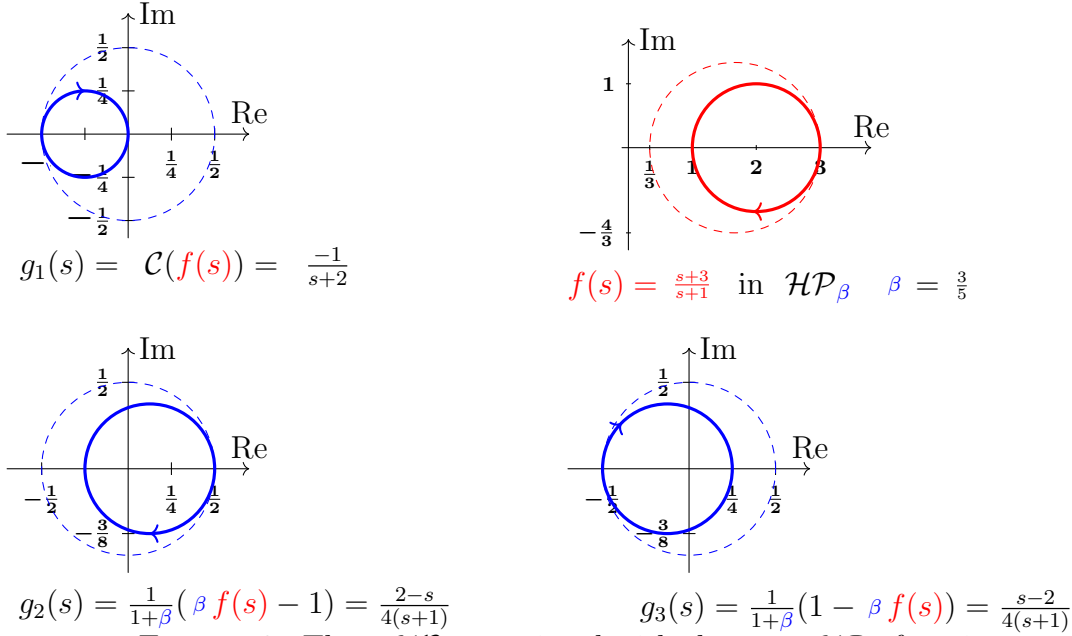
Within the family of $\mathcal{HP}_{\color{blue}T}$ functions
described in Eq. \eqref{eq:HP_Delta} $\mathcal{HP}_{\color{blue}T}$ of a special interest are
the ~{\em canonical}~ ones defined as
\begin{center}
$
F(s)+(F(s))^*-\left(\begin{smallmatrix}{\color{blue}T}\end{smallmatrix}+
(F(s))^*\begin{smallmatrix}{\color{blue}T}\end{smallmatrix}F(s)\right)~~
\left\{\begin{matrix}\succcurlyeq 0&&\forall s\in\C_R\\~\\ =0&&\forall s\in{i}\R.
\end{matrix}\right.
$
\end{center}

This is illustrated in Figure \ref{Figure:Reza_HP}. The subset of {\em canonical}
$\mathcal{HP}_{\color{blue}T}$ functions is in the focus of \cite{AlpayLew2025b}.\\

\begin{figure}[H]
\centering
\begin{minipage}{0.50\linewidth}
{\rm Consider three $\mathcal{HP}_{\beta}$ functions with 
\mbox{$\begin{smallmatrix}\beta\end{smallmatrix}=\frac{3}{5}$}\\
\mbox{${\color{red}f_1}(s)=\frac{3}{5}+\frac{8a^2}{5(s+a)^2}$}\quad
\mbox{${\color{blue}f_2}(s)=\frac{41}{15}-\frac{8b^2}{5(s+b)^2}$,}\\
and a {\em canonical} \mbox{$f_3(s)=\frac{3s+\frac{1}{3}c}{s+c}~$.}
In particular,

$\begin{matrix}{\color{red}f_1}(s)_{|_{s=\pm{i}c}}=f_3(s)_{|_{s=\pm{i}\frac{c}{3}}}=&\frac{3}{5}
\mp{i}\frac{4}{5}\\~\\{\color{blue}f_2}(s)_{|_{s=\pm{i}b}}=f_3(s)_{|_{s=\mp{i}3c}}=&\frac{41}{15}
\mp{i}\frac{4}{5}~.\end{matrix}$
\bigskip

The value of the positive parameters $a$, $b$ or $c$, do not affect the Nyquist plots to the right.
}
\end{minipage}\quad\quad\begin{minipage}{0.44\linewidth}
\includegraphics[width=0.99\textwidth]{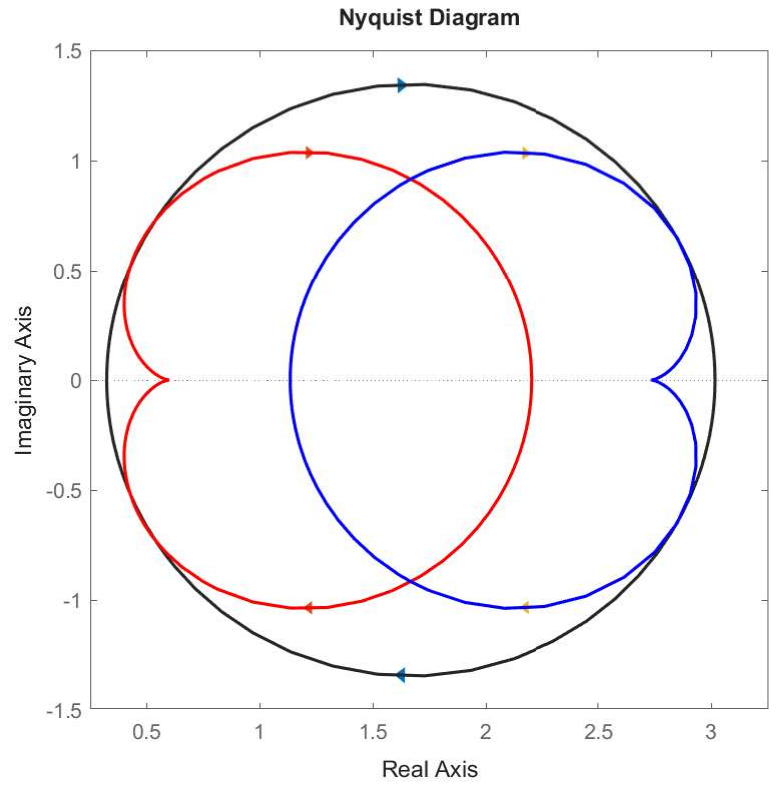}
		\end{minipage}
\caption{The Nyquist plots of:~${\color{red}f_1}(s)=\frac{3}{5}+\frac{8a^2}{5(s+a)^2}$~
${\color{blue}f_2}(s)=\frac{41}{15}-\frac{8b^2}{5(s+b)^2}$ and
$f_3(s)=\frac{3s+\frac{1}{3}c}{s+c}$~.}
\label{Figure:Reza_HP}
\end{figure}

We are now ready to start proving the results presented in the Introduction.
\\

{\bf Proof of Proposition \ref{Pn:HP_Delta_Order}:~}\\
\noindent
Indeed, if in Eq. \eqref{eq:Quadratic_HP_Delta} we denote \mbox{$M_{\color{blue}T}:=
\left(\begin{smallmatrix}-{\color{blue}T}&&~~I_m\\~\\~~I_m&&-{\color{blue}T} \end{smallmatrix}\right)$}
then, for $\mathcal{HP}_{\color{blue}T_1}$ and $\mathcal{HP}_{\color{cyan}T_2}$ one has that
\[
\overbrace{\left(\begin{smallmatrix}-{\color{blue}T}_1&&~~I_m\\~\\~~I_m&&
-{\color{blue}T}_1\end{smallmatrix}\right)}^{M_{{\color{blue}T_1}}}
-
\overbrace{\left(\begin{smallmatrix}-{\color{cyan}T_2}&&~~I_m\\~\\~~I_m&&
-{\color{cyan}T_2}\end{smallmatrix}\right)}^{M_{\color{cyan}T_2}}
=
\left(\begin{smallmatrix}{\color{cyan}T_2}-{\color{blue}T_1}&&0\\~\\0&&
{\color{cyan}T_2}-{\color{blue}T_1}\end{smallmatrix}\right)
\in\overline{\mathbf P}_{2m}
\]
and thus, Eq. \eqref{eq:Partial_Order} holds.
\smallskip

The fact that the set $\mathcal{P}$ is on some boundary is immediate upon comparing Eqs.
\eqref{eq:HP_Delta}, \eqref{eq:Quadratic_HP_Delta}, with Eqs. \eqref{eq:Def_P},
\eqref{eq:Quad_Def_P}, respectively.
\qed

\subsection{A Useful Lemma}

A unifying algebraic framework used throughout this work is of sets of $q\times q$ matrices satisfying
Matricial Quadratic Form Inequalities of the form
\begin{equation}\label{eq:E}
\begin{matrix}
\mathbf{E}=&\left\{ E\in{\C}^{q\times q}~:~\left(\begin{smallmatrix}E\\~\\I_q\end{smallmatrix}\right)^*
M\left(\begin{smallmatrix}E\\~\\I_q\end{smallmatrix}\right)\in\mathbf{P}_q~\right\}\\~\\
\overline{\mathbf E}=&\left\{ E\in{\C}^{q\times q}~:~
\left(\begin{smallmatrix}E\\~\\I_q\end{smallmatrix}\right)^*
M
\left(\begin{smallmatrix}E\\~\\I_q\end{smallmatrix}\right)
\in\overline{\mathbf P}_q~\right\}
\end{matrix}
\quad\quad\quad
\begin{smallmatrix}
M:=\left(\begin{smallmatrix}X&&V\\~\\V&&Y\end{smallmatrix}\right)
\\~\\~\\
X,~V,~Y\in\overline{\mathbf H}_q
\\~\\~\\
{\rm inertia}(H)=(q,~0,q).
\end{smallmatrix}
\end{equation}
In this case, $M$ is said to have ``non-singular (or regular) balanced\begin{footnote}{In particular, if
$V=0$, it implies that $-X,~Y\in\mathbf{P}_q$, and if $X=Y=0$, then $V\in\mathbf{H}_q$.}\end{footnote}
inertia", see e.g. \cite{LoewyPierce1995} and \cite{PierRodm1988}.\quad
As before, $\overline{\mathbf E}$ denotes the closure of the open set $\mathbf{E}$.
\smallskip

Sometimes it is more convenient to explicitly write Eq. \eqref{eq:E} as
\begin{equation}\label{eq:Explicit_Qudratic_Form}
\begin{matrix}
\mathbf{E}=&\left\{ E\in{\C}^{q\times q}~:~VE+E^*V+E^*XE+Y\succ 0 \right\}
\\~\\
\overline{\mathbf E}=&\left\{ E\in{\C}^{q\times q}~:~VE+E^*V+E^*XE+Y\succcurlyeq 0 \right\}
\end{matrix}
\quad\quad
X,~V,~Y\in\overline{\mathbf H}_q~.
\end{equation}
See e.g. \cite[Eq. (9.11)]{LanRod1995}.

\begin{Rk}\label{Rk:Riccati}
{\rm
A word of caution, to consolidate Eq. \eqref{eq:Explicit_Qudratic_Form} with the framework of the renowned
Algebraic Riccati Equation, e.g. \cite[Section 6.2]{DullPaga2000}, \cite{LanRod1995}, one needs to confine
the discussion to
\[
\mathbf{E}\bigcap \overline{\mathbf H}_q~,
\]
(and in addition substitutes $\succcurlyeq$ by equality). However, such an assumption, is too
restrictive to our purposes.
$\T$
}
\end{Rk}

A closer scrutiny of Eqs. \eqref{eq:E}, \eqref{eq:Explicit_Qudratic_Form} reveals that
this seemingly simple formulation, implies further structural properties of the sets
$\mathbf{E}$ and $\overline{\mathbf E}$. To this end, we first need to recall the following.

\begin{Dn}\label{Dn:n-MatrixConvex}
{\rm
{\bf a.}~ A family $\mathbf{A}_n\subset\C^{n\times n}$ is said to be~
$~n$-{\em matrix-convex}~ if for all natural $k$:\\
For all $A_1$, $\ldots$, $A_k$ in $\mathbf{A}_n$, for all 
${\scriptstyle\Upsilon}_1$,
$\ldots$,
${\scriptstyle\Upsilon}_k\in\C^{n\times n}$
one has that
\begin{equation}\label{eq:Quasi_Matrix_Convex_Def}
\sum\limits_{j=1}^k{{\scriptstyle\Upsilon}_j}^*
{\scriptstyle\Upsilon}_j=I_n~,
\quad~~\Longrightarrow~~\quad
\sum\limits_{j=1}^k{{\scriptstyle\Upsilon}_j}^*
A_j{\scriptstyle\Upsilon}_j\in\mathbf{A}_n~.
\end{equation}
Alternatively, for all natural $k$,
\begin{equation}\label{eq:AlternativeDefMatrix-Convex}
\left(\begin{smallmatrix}{\Upsilon}_1\\ \vdots\\~\\{\Upsilon}_k\end{smallmatrix}\right)^*
\left(\begin{smallmatrix}{\Upsilon}_1\\ \vdots\\~\\{\Upsilon}_k\end{smallmatrix}\right)
=I_n
\quad~~
\Longrightarrow
\quad~~
\left(\begin{smallmatrix}{\Upsilon}_1\\ \vdots\\~\\{\Upsilon}_k\end{smallmatrix}\right)^*
\left(\begin{smallmatrix}A_1&&~&&~\\ ~&&\ddots&&~\\~\\~&&~&&A_k\end{smallmatrix}\right)
\left(\begin{smallmatrix}{\Upsilon}_1\\ \vdots\\~\\{\Upsilon}_k\end{smallmatrix}\right)
\in\mathbf{A}_n~.
\end{equation}
\bigskip

{\bf b.}~ A family \mbox{\scalebox{1.2}{$\mathbf{A}$}$=\bigcup\limits_{n=1}^{\infty}\mathbf{A}_n$}
is said to be~ {\em matrix-convex}~ if for all natural $k$:\\
For all $A_1$, $\ldots$, $A_k$ of dimensions \mbox{$n_1\times n_1$}, $\ldots$, 
\mbox{$n_k\times n_k$,} respectively, all in $\mathbf{A}$,\\
for all ${\scriptstyle\Upsilon}_1\in\C^{n_1\times\nu}$, $\ldots$,
${\scriptstyle\Upsilon}_k\in\C^{n_k\times\nu}$, with
$\nu\in[1,~\max(n_1,~\ldots~,~n_k)]$, so that
\begin{equation}\label{eq:MatrixConvexDef1}
\sum\limits_{j=1}^k{{\scriptstyle\Upsilon}_j}^*
{\scriptstyle\Upsilon}_j=I_{\nu}
\quad\quad\forall\nu\in[1,~\max(n_1,~\ldots~,~n_k)],
\end{equation}
one has that
\begin{equation}\label{eq:Matrix_Convex_Def_2}
\sum\limits_{j=1}^k{{\scriptstyle\Upsilon}_j}^*
A_j{\scriptstyle\Upsilon}_j
\end{equation}

\hfill{belongs to $\mathbf{A}$ as well.}
}
$\T$
\end{Dn}

\begin{La}\label{La:Structural_Properties}
Let the sets $\mathbf{E}$ and $\overline{\mathbf E}$ be as in Eqs. \eqref{eq:E},
\eqref{eq:Explicit_Qudratic_Form}. Then the following is true.
\begin{itemize}
\item[(i)~~~]{}
The set $\mathbf{E}$ ($\overline{\mathbf E}$) is not empty, open (closed).
\vskip 0.1cm

\item[(ii)~~]{} 
The set $\mathbf{E}$ ($\overline{\mathbf E}$) is convex, if and only if \mbox{$-X\in\overline{\mathbf P}_q~$.}\\
In particular, with the same $V$ and $Y$, let the sets $\mathbf{E}$, $\hat{\mathbf E}$ be associated with
\mbox{$0\succcurlyeq X\succcurlyeq\hat{X}$,} respectively. Then, \mbox{$\mathbf{E}\subset\hat{\mathbf E}$.}
\vskip 0.1cm

\item[(iii)~]{}
The set $\mathbf{E}$ ($\overline{\mathbf E}$) is closed under inversion,
if and only if, $X=Y$.\\
This induces a partial order: With the same $V$, let $\mathbf{E}_0$, $\mathbf{E}_1$ $\mathbf{E}_2$ be
associated with \mbox{$0=X_0=Y_0\succcurlyeq X_1=Y_1\succcurlyeq X_2=Y_2$,} respectively. Then, 
\mbox{$\mathbf{E}_2\subset\mathbf{E}_1\subset\mathbf{E}_0~$,} and in particular, the set $\mathbf{E}_0$
is maximal.
\vskip 0.1cm

\item[(iv)~]{} 
The set $\mathbf{E}$ ($\overline{\mathbf E}$) is a cone, if and only if,
either $X=0$, $Y\succcurlyeq 0$ or $X\succcurlyeq 0$
and $Y\succcurlyeq VX^{\dagger}V$, where 
$X^{\dagger}$ is the generalized inverse\begin{footnote}{If
${\rm rank}(X)=r$, $r\in[1,~q]$, then
$X=U^*\left(\begin{smallmatrix}\hat{X}&0\\0&0\end{smallmatrix}\right)U$
with
$U^*U=I_q=UU^*$ and $\hat{X}\in\mathbf{P}_r$ then
$X^{\dagger}=U^*\left(\begin{smallmatrix}\hat{X}^{-1}&0\\0~&0\end{smallmatrix}\right)U$.
}
\end{footnote}
of $X$.
\vskip 0.1cm

\noindent
\item[(v)~~]{} 
The set $\mathbf{E}$ ($\overline{\mathbf E}$) is closed under sign change of its elements, i.e.
$\pm E$ belong to it, if and only if, \mbox{$Y\succcurlyeq 0$} and $V=0$.
\vskip 0.1cm

\noindent
In particular, consider the following partial order: With $V=0$, let $\mathbf{E}_0$, $\mathbf{E}_1$,
$\mathbf{E}_2$ be associated with \mbox{$Y_j=H-T_j$} and \mbox{$X_j=-(H+T_j)$,} where 
\mbox{$H\succ T_2\succcurlyeq T_1 \succcurlyeq T_0=0$,} and $H$ is prescribed. Then,
\mbox{$\mathbf{E}_2\subset\mathbf{E}_1\subset\mathbf{E}_0~$,} and the set $\mathbf{E}_0$ is maximal.
\vskip 0.1cm

\item[(vi)~]{}
The set $\mathbf{E}$ ($\overline{\mathbf E}$) is closed under product among its elements (i.e. whenever
\mbox{$E_0, E_1\in\mathbf{E}$,} also the product $E_0E_1$ belongs to it), if and only if, 
\mbox{$(X+Y)\succcurlyeq 0$.}
\vskip 0.1cm

\noindent
In particular, consider the following partial order: With $V=0$, let $\mathbf{E}_0$, $\mathbf{E}_1$,
$\mathbf{E}_2$ be associated with
\mbox{$(X_2+Y_2)\succcurlyeq(X_1+Y_1)\succcurlyeq(X_0+Y_0)=0$,}
respectively. Then, 
\mbox{$\mathbf{E}_2\subset\mathbf{E}_1\subset\mathbf{E}_0~$,} and the set $\mathbf{E}_0$ is maximal.
\vskip 0.1cm

\item[(vii)]{} 
Let the set $\mathbf{E}$ ($\overline{\mathbf E}$) be as in Eq. \eqref{eq:E}, with
\begin{center}
$
M=\left(\begin{smallmatrix}xI_q&&vI_q\\~\\vI_q&&yI_q\end{smallmatrix}\right)\quad\quad
\begin{smallmatrix}0\geq x\\~\\v, y\in\R\\~\\v^2>wy\end{smallmatrix}
$
\end{center}
where $x, v, y$ are scalar parameters. Then this set is $q$-matrix-convex.
\end{itemize}
\end{La}

\noindent
{\bf Proof :}~ (i)~ 
By definition, the matrix $M$ in Eq. \eqref{eq:E} has regular balanced inertia. Thus
the Sylvester's Inertia Theorem (see e.g. \cite[Theorem 4.5.8]{HJ1}) implies that
there exists a non-singular $L^{2q\times 2q}$ so that
\[
L^*ML=\left(\begin{smallmatrix}-I_q&&0\\~\\~0&&I_q\end{smallmatrix}\right).
\]
Note also that this $L$ can always be chosen so that there are non-zero $E,~K\in\C^{q\times q}$
so that
\[
\left(\begin{smallmatrix}E\\~\\I_q\end{smallmatrix}\right)
:=
L\left(\begin{smallmatrix}K\\~\\{\scriptstyle\sqrt{2}}K\end{smallmatrix}\right).
\]
Then, one has that,
\[
\left(\begin{smallmatrix}E\\~\\I_q\end{smallmatrix}\right)^*
M
\left(\begin{smallmatrix}E\\~\\I_q\end{smallmatrix}\right)
=K^*K\in\overline{\mathbf P}_q~,
\]
so $\mathbf{E}$ is not empty.\quad This set is open (closed),
in the sense that within
$\overline{\mathbf H}_q$, the set $\overline{\mathbf P}_q$ is the
closure of the open set $\mathbf{P}_q~$.
\bigskip

(ii)~
To simplify the presentation, whenever a property holds for both, the open and the closed sets,
we shall show it only for either $\mathbf{E}$ or $\overline{\mathbf E}$.\\
Convexity.\quad
Let \mbox{$E_0, E_1$} be in $\overline{\mathbf E}$ i.e.  from Eq. \eqref{eq:Explicit_Qudratic_Form},
\begin{equation}\label{eq:Pair_of_Inequalities}
\begin{matrix}
Q_0:=VE_0+{E_0}^*V+{E_0}^*XE_0+Y\succcurlyeq 0\\~\\
Q_1:=VE_1+{E_1}^*V+{E_1}^*XE_1+Y\succcurlyeq 0,
\end{matrix}
\end{equation}
Next, \mbox{$\forall~{\scriptstyle\alpha}\in[0, 1]$}, denote,
\mbox{$~E_{\alpha}:={\scriptstyle\alpha}E_1+(1-{\scriptstyle\alpha})E_0$}.
Convexity means that \mbox{$E_{\alpha}\in\overline{\mathbf E}$,}
\mbox{$\forall~{\scriptstyle\alpha}\in[0, 1]$.}
This in turn means that the quantity $\Delta_{\alpha}$ below is positive semi-definite:
\[
\begin{matrix}
{\Delta}_{\alpha}&:=&VE_{\alpha}+{E_{\alpha}}^*V+{E_{\alpha}}^*XE_{\alpha}+Y
\\~\\~&=&
{\scriptstyle\alpha}\left({\scriptstyle\alpha}E_1^*YE_1+VE_1+{E_1}^*V+Y+E_1^*YE_1-E_1^*YE_1\right)
\quad\quad\\~\\~&~&+
(1-{\scriptstyle\alpha})\left((1-{\scriptstyle\alpha})E_0^*YE_0+VE_0+{E_0}^*V+Y+E_0^*YE_0-E_0^*YE_0\right)
\\~\\~&~&+
{\scriptstyle\alpha}(1-{\scriptstyle\alpha})\left(E_1^*XE_0+E_0^*XE_1\right)\quad\quad\quad
\\~\\~&=&
\underbrace{
{\scriptstyle\alpha}Q_1+(1-{\scriptstyle\alpha})Q_0}_{\succcurlyeq 0}
+
{\scriptstyle\alpha}(1-{\scriptstyle\alpha})(E_0-E_1)^*(-X)(E_0-E_1).
\end{matrix}
\]
Thus ${\Delta}_{\alpha}\succcurlyeq 0$, if and only if, $~-X\succcurlyeq 0$.
\smallskip

Now, assume that with the same $V$, $Y$, one has that \mbox{$0\succcurlyeq X\succcurlyeq\hat{X}$.}
If we denote,
\[
{\Delta}_{\alpha}=VE_{\alpha}+{E_{\alpha}}^*V+{E_{\alpha}}^*XE_{\alpha}+Y
\quad{\rm 
and
}\quad
\hat{\Delta}_{\alpha}=VE_{\alpha}+{E_{\alpha}}^*V+{E_{\alpha}}^*\hat{X}E_{\alpha}+Y,
\]
then
\[
{\Delta}_{\alpha}
-\hat{\Delta}_{\alpha}
=
{E_{\alpha}}^*(X-\hat{X})E_{\alpha}\succcurlyeq 0.
\]
Thus, this part of the claim is established.
\bigskip

(iii)~ 
Closure under inversion.\quad One can re-write the relation in Eq. \eqref{eq:E} as
\[
\left(\begin{smallmatrix}E\\~\\I_q\end{smallmatrix}\right)^*
M
\left(\begin{smallmatrix}E\\~\\I_q\end{smallmatrix}\right)
=Q\succ 0.
\]
Assuming that $E$ is non-singular, we can multiply the relation by $(E^{-1})^*$ from the left
and $E^{-1}$ from the right to obtain,
\[
\left(\begin{smallmatrix}E^{-1}\\~\\I_q\end{smallmatrix}\right)^*
\left(\begin{smallmatrix}0&&I_q\\~\\I_q&&0\end{smallmatrix}\right)
M
\left(\begin{smallmatrix}0&&I_q\\~\\I_q&&0\end{smallmatrix}\right)
\left(\begin{smallmatrix}E^{-1}\\~\\I_q\end{smallmatrix}\right)
=
\left(\begin{smallmatrix}I_q\\~\\ E^{-1}\end{smallmatrix}\right)^*
M
\left(\begin{smallmatrix}I_q\\~\\ E^{-1}\end{smallmatrix}\right)
=
(E^{-1})^*
Q
E^{-1}
\succ 0.
\]
To guarantee that this holds for all $E$, one needs to verify that
\[
\underbrace{
\left(\begin{smallmatrix}X&&V\\~\\V&&Y\end{smallmatrix}\right)
}_M
=
\left(\begin{smallmatrix}0&&I_q\\~\\I_q&&0\end{smallmatrix}\right)
\underbrace{
\left(\begin{smallmatrix}X&&V\\~\\V&&Y\end{smallmatrix}\right)
}_M
\left(\begin{smallmatrix}0&&I_q\\~\\I_q&&0\end{smallmatrix}\right)
=
\left(\begin{smallmatrix}Y&&V\\~\\V&&X\end{smallmatrix}\right),
\]
and this holds if and only if, $X=Y$. Hence, the first part of the claim is established.
\smallskip

To establish the partial order, here with the same $V$, and with
\mbox{$0=X_0=Y_0\succcurlyeq X_1=Y_1\succcurlyeq X_2=Y_2$,}
one can re-write Eq. \eqref{eq:Explicit_Qudratic_Form} as
\begin{equation}\label{eq:E_with_X=Y}
\overline{\mathbf E}_j=\left\{ E\in{\C}^{q\times q}~:~VE+E^*V\succcurlyeq
-(E^*X_jE+X_j)~\right\}
\quad{\rm with}\quad
j=0,~1,~2. 
\end{equation}
As the right-hand side takes the form of,
\mbox{$\left(\begin{smallmatrix}E\\~\\I_q\end{smallmatrix}\right)^*
\left(\begin{smallmatrix}-X_j&~0\\0&-X_j\end{smallmatrix}\right)
\left(\begin{smallmatrix}E\\~\\I_q\end{smallmatrix}\right)\succcurlyeq 0$,}
the partial order part, is established as well.
\bigskip

(iv)~ Cone.\quad From Eq. \eqref{eq:Explicit_Qudratic_Form} it is clear that if 
(in norm) $E$ is small, one must have $Y\succcurlyeq 0$, and if 
$E$ is large, one must have $X\succcurlyeq 0$. More precisely,
if $X\succcurlyeq 0$ with ${\rm rank}(X)=r$, $r\in[1,~q]$, one can write,
$X=U^*\left(\begin{smallmatrix}\hat{X}&0\\0&0\end{smallmatrix}\right)U$
with
$U^*U=I_q=UU^*$ and $\hat{X}\in\mathbf{P}_r$ and then
$X^{\dagger}=U^*\left(\begin{smallmatrix}\hat{X}^{-1}&0\\0~&0\end{smallmatrix}\right)U$.
In this case, one can rewrite Eq,, one can rewrite Eq.  \eqref{eq:Explicit_Qudratic_Form} as,
\[
Q:=VE+E^*V+E^*XE+Y=(XE+V)^*X^{\dagger}(XE+V)+Y-
V^*X^{\dagger}V.
\]
Thus, $Q\succcurlyeq 0$, whenever $Y\succcurlyeq V^*X^{\dagger}V$, so this item is established.
\bigskip

(v)~ 
Closure under change of sign.\quad
Eq. \eqref{eq:Explicit_Qudratic_Form} be re-written as
\begin{equation}\label{eq:Alternative_Explicit_Form}
\overline{\mathbf E}=\left\{ E\in{\C}^{q\times q}~:~VE+E^*V\succcurlyeq
-(E^*XE+Y)~\right\}.
\end{equation}
Now, the left-hand side of the inequality sign clearly depends on the sign of $E$, while the right-hand
side is independent of the sign of $E$. This implies that $V=0$, and \mbox{$Y\succcurlyeq 0$,} so the
first part is established (if in addition \mbox{$X\succcurlyeq 0$,} then the set $\overline{\mathbf E}$
is comprized of all $\C^{q\times q}$ matrices).
\smallskip

To establish the partial order, note that here for

\begin{center}
$H\succ T_2\succcurlyeq T_1\succcurlyeq T_0=0,$
\end{center}

where $H$ is prescribed. Then, one has that,
\begin{center}
$\mathbf{E_j}=\{ E~:~H-E^*HE\succ T_j+E^*T_jE~\}\quad\quad j=0,~1,~2.$
\end{center}

This, implies that
\mbox{$\mathbf{E}_2\subset\mathbf{E}_1\subset\mathbf{E}_0$}
and since for $\mathbf{E}_0$, the right-hand side vanishes, it is maximal.
\bigskip

(vi)~ 
Closure under product of elements.\quad By assumption Eq. \eqref{eq:Pair_of_Inequalities}
implies that  both
\[
\begin{matrix}
Q_2:=VE_0E_1+(E_0E_1)^*V+(E_0E_1)^*XE_0E_1+Y\succcurlyeq 0
\\~\\
Q_3:=VE_1E_0+(E_1E_0)^*V+(E_1E_0)^*XE_1E_0+Y\succcurlyeq 0.
\end{matrix}
\]
This already means that one must have $V=0$, and thus,
\[
Q_2=(E_0E_1)^*XE_0E_1+Y\succcurlyeq 0
\quad{\rm and}\quad
Q_3=(E_1E_0)^*XE_1E_0+Y\succcurlyeq 0.
\]
Note now that
\begin{equation}\label{eq:Q1_Q_3}
\begin{matrix}
Q_2=(E_0E_1)^*XE_0E_1+Y=\underbrace{Q_1}_{\succcurlyeq 0}+\underbrace{{E_1}^*Q_0E_1}_{\succcurlyeq 0}
-{E_1}^*(X+Y)E_1
\\
Q_3=(E_1E_0)^*XE_1E_0+Y=\overbrace{Q_0}^{\succcurlyeq 0}+\overbrace{{E_0}^*Q_1E_0}^{\succcurlyeq 0}
-{E_0}^*(X+Y)E_0
\end{matrix}
\end{equation}
Hence indeed the condition is \mbox{$(X+Y)\succcurlyeq 0$.}
\smallskip

The partial order follows from Eq. \eqref{eq:Q1_Q_3}.
\bigskip

(vii)~ $q$-Matrix-Convex.\quad Take in Eq. \eqref{eq:E},
$M=\left[\begin{smallmatrix}xI_q&&vI_q\\~\\vI_q&&yI_q\end{smallmatrix}\right]$ with scalar
parameters $v, x, y\in\R$ and $v^2>xy$. For some natural $k$, let \mbox{$E_1,~\ldots~,~E_k$}
be in $\mathbf{E}$, i.e,
\begin{equation}\label{eq:Rashidiyeh}
\left[\begin{smallmatrix}E_j\\~\\I_q\end{smallmatrix}\right]^*\left[\begin{smallmatrix}xI_q&&vI_q
\\~\\vI_q&&yI_q\end{smallmatrix}\right]\left[\begin{smallmatrix}E_j\\~\\I_q\end{smallmatrix}\right]
={\scriptstyle Q_j}\in\mathbf{P}_q\quad\quad\begin{smallmatrix}j=1~,~\ldots~k\\~\\
v,~x,~y\in\R\\~\\v^2>xy.\end{smallmatrix}
\end{equation}
To show matrix-convexity, for \mbox{$j=1,~\ldots~,~k$} let 
$\begin{smallmatrix}{\Upsilon}_j\end{smallmatrix}\in\C^{q\times q}$
be so that
\mbox{$\sum\limits_{j=1}^k{\begin{smallmatrix}{\Upsilon}_j\end{smallmatrix}}^*
\begin{smallmatrix}{\Upsilon}_j\end{smallmatrix}=I_q$}.
Now take a matrix-convex combination of
\mbox{$E_1~,~\ldots~,~E_k$},

\begin{center}
$\begin{matrix}
A&:=\left(
\left[\begin{smallmatrix}{{\Upsilon}_1}^*E_1{{\Upsilon}_1}
\\~\\I_{q_1}\end{smallmatrix}\right]
+~\cdots~+
\left[\begin{smallmatrix}{{\Upsilon}_k}^*E_k{{\Upsilon}_k}
\\~\\I_{q_k}\end{smallmatrix}\right]
\right)^*
\left[\begin{smallmatrix}aI_q&&vI_q\\~\\vI_q&&bI_q\end{smallmatrix}\right]
\left(
\left[\begin{smallmatrix}{{\Upsilon}_1}^*E_1{{\Upsilon}_1}
\\~\\I_{q_1}\end{smallmatrix}\right]
+~\cdots~+
\left[\begin{smallmatrix}{{\Upsilon}_k}^*E_k{{\Upsilon}_k}
\\~\\I_{q_k}\end{smallmatrix}\right]
\right)
\\~\\~&=
\left(
\left[\begin{smallmatrix}{{\Upsilon}_1}^*E_1{{\Upsilon}_1}
\\~\\{{\Upsilon}_1}^*{{\Upsilon}_1}\end{smallmatrix}\right]
+~\cdots~+
\left[\begin{smallmatrix}{{\Upsilon}_k}^*E_k{{\Upsilon}_k}
\\~\\{{\Upsilon}_k}^*{{\Upsilon}_k}\end{smallmatrix}\right]
\right)^*
\left[\begin{smallmatrix}aI_q&&vI_q\\~\\vI_q&&bI_q\end{smallmatrix}\right]
\left(
\left[\begin{smallmatrix}{{\Upsilon}_1}^*E_1{{\Upsilon}_1}
\\~\\{{\Upsilon}_1}^*{{\Upsilon}_1}\end{smallmatrix}\right]
+~\cdots~+
\left[\begin{smallmatrix}{{\Upsilon}_k}^*E_k{{\Upsilon}_k}
\\~\\{{\Upsilon}_k}^*{{\Upsilon}_k}\end{smallmatrix}\right]
\right).
\end{matrix}
$
\end{center}
Next, denoting
$
\begin{smallmatrix}{\Theta}_j\end{smallmatrix}:=\left[\begin{smallmatrix}{\Upsilon}_j&&0\\~\\0&&{\Upsilon}_j
\end{smallmatrix}\right],
$
with the above
$M=\left[\begin{smallmatrix}xI_q&&vI_q\\~\\vI_q&&yI_q\end{smallmatrix}\right]$,
this can be written as
\[
A=
\left(
\begin{smallmatrix}{\Theta}_1\end{smallmatrix}
\left[\begin{smallmatrix}E_1\\~\\I_{q_1}\end{smallmatrix}\right]
\begin{smallmatrix}{\Upsilon}_1\end{smallmatrix}
+~\cdots~+
\begin{smallmatrix}{\Theta}_k\end{smallmatrix}
\left[\begin{smallmatrix}E_k\\~\\I_{q_k}\end{smallmatrix}\right]
\begin{smallmatrix}{\Upsilon}_1\end{smallmatrix}
\right)^*
M
\left(
\begin{smallmatrix}{\Theta}_1\end{smallmatrix}
\left[\begin{smallmatrix}E_1\\~\\I_{q_1}\end{smallmatrix}\right]
\begin{smallmatrix}{\Upsilon}_1\end{smallmatrix}
+~\cdots~+
\begin{smallmatrix}{\Theta}_k\end{smallmatrix}
\left[\begin{smallmatrix}E_k\\~\\I_{q_k}\end{smallmatrix}\right]
\begin{smallmatrix}{\Upsilon}_1\end{smallmatrix}
\right).
\]
Using ${\delta}_{j,l}$ the $\delta$ of Kronecker, note that
\[
\begin{smallmatrix}{\Theta}_j\end{smallmatrix}^*M\begin{smallmatrix}{\Theta}_l\end{smallmatrix}
={\delta}_{j,l}\begin{smallmatrix}{\Theta}_j\end{smallmatrix}^*M
\begin{smallmatrix}{\Theta}_j\end{smallmatrix}\quad\quad j, l=1,~\ldots~,~k,
\]
and thus one has that,
\vskip 0.1cm

$A=
\overbrace{
\left(
\begin{smallmatrix}{\Theta}_1\end{smallmatrix}
\left[\begin{smallmatrix}E_1\\~\\I_{q_1}\end{smallmatrix}\right]
\begin{smallmatrix}{\Upsilon}_1\end{smallmatrix}
+~\cdots~+
\begin{smallmatrix}{\Theta}_k\end{smallmatrix}
\left[\begin{smallmatrix}E_k\\~\\I_{q_k}\end{smallmatrix}\right]
\begin{smallmatrix}{\Upsilon}_1\end{smallmatrix}
\right)^*
M
\begin{smallmatrix}{\Theta}_1\end{smallmatrix}
\left[\begin{smallmatrix}E_1\\~\\I_{q_1}\end{smallmatrix}\right]
\begin{smallmatrix}{\Upsilon}_1\end{smallmatrix}
}^{
\left(\left[\begin{smallmatrix}E_1\\~\\I_{q_1}\end{smallmatrix}\right]
{\Upsilon}_1\right)^*M
\left(\left[\begin{smallmatrix}E_1\\~\\I_{q_1}\end{smallmatrix}\right]
{\Upsilon}_1\right)}
$
\vskip 0.3cm

\begin{center}
$+
\overbrace{
\left(
\begin{smallmatrix}{\Theta}_1\end{smallmatrix}
\left[\begin{smallmatrix}E_1\\~\\I_{q_1}\end{smallmatrix}\right]
\begin{smallmatrix}{\Upsilon}_1\end{smallmatrix}
+~\cdots~+
\begin{smallmatrix}{\Theta}_k\end{smallmatrix}
\left[\begin{smallmatrix}E_k\\~\\I_{q_k}\end{smallmatrix}\right]
\begin{smallmatrix}{\Upsilon}_1\end{smallmatrix}
\right)^*
M
\begin{smallmatrix}{\Theta}_2\end{smallmatrix}
\left[\begin{smallmatrix}E_2\\~\\I_{q_2}\end{smallmatrix}\right]
\begin{smallmatrix}{\Upsilon}_2\end{smallmatrix}
}^{
\left(\left[\begin{smallmatrix}E_2\\~\\I_{q_2}\end{smallmatrix}\right]
{\Upsilon}_2\right)^*M
\left(\left[\begin{smallmatrix}E_2\\~\\I_{q_2}\end{smallmatrix}\right]
{\Upsilon}_2\right)}
$
\vskip 0.2cm

$
\vdots
$
\end{center}

\hfill{
$+
\overbrace{
\left(
\begin{smallmatrix}{\Theta}_1\end{smallmatrix}
\left[\begin{smallmatrix}E_1\\~\\I_{q_1}\end{smallmatrix}\right]
\begin{smallmatrix}{\Upsilon}_1\end{smallmatrix}
+~\cdots~+
\begin{smallmatrix}{\Theta}_k\end{smallmatrix}
\left[\begin{smallmatrix}E_k\\~\\I_{q_k}\end{smallmatrix}\right]
\begin{smallmatrix}{\Upsilon}_1\end{smallmatrix}
\right)^*
M
\begin{smallmatrix}{\Theta}_k\end{smallmatrix}
\left[\begin{smallmatrix}E_k\\~\\I_{q_k}\end{smallmatrix}\right]
\begin{smallmatrix}{\Upsilon}_k\end{smallmatrix}
}^{
\left(\left[\begin{smallmatrix}E_k\\~\\I_{q_k}\end{smallmatrix}\right]
{\Upsilon}_k\right)^*M
\left(\left[\begin{smallmatrix}E_k\\~\\I_{q_k}\end{smallmatrix}\right]
{\Upsilon}_k\right)},
$}

which by using Eq. \eqref{eq:Rashidiyeh} can be compactly written as
\[
A=
\begin{smallmatrix}{\Upsilon}_1\end{smallmatrix}^*
Q_1
\begin{smallmatrix}{\Upsilon}_1\end{smallmatrix}
+
\begin{smallmatrix}{\Upsilon}_2\end{smallmatrix}^*
Q_2
\begin{smallmatrix}{\Upsilon}_2\end{smallmatrix}
+~\cdots~+
\begin{smallmatrix}{\Upsilon}_k\end{smallmatrix}^*
Q_k
\begin{smallmatrix}{\Upsilon}_k\end{smallmatrix}\in\mathbf{P}_q~,
\]
so indeed \mbox{$({\scriptstyle{\Upsilon}_1}^*E_1{\scriptstyle{\Upsilon}_1}
+~\cdots~+{\scriptstyle{\Upsilon}_k}^*E_k{\scriptstyle{\Upsilon}_k})
\in\mathbf{E}$} and this set is $q$-matrix-convex.
Thus the proof is complete.
\qed
\smallskip

As already pointed out, the underlying theme of this work is exploring strong ties between physical
properties and mathematical structure. Below, an immediate consequence of the above Lemma
\ref{La:Structural_Properties}, illustrates how {\em stability}, in the continuous-time framework (item
(i)) and in discrete-time (item (ii)), can be inferred from the mathematical structure of the set. Moreover,
for dissipativity, maximality of the structure, should be compromized.

\begin{Cy}
Let the sets $\mathbf{E}$ and $\overline{\mathbf E}$ be as in Eqs. \eqref{eq:E},
\eqref{eq:Explicit_Qudratic_Form}. Then the following is true.
\begin{itemize}
\item[(i)~]{} $\mathbf{E}$ is a cone, closed under inversion, if and only if,
Eq. \eqref{eq:Explicit_Qudratic_Form} take the form
\[
\mathbf{E}=\{E~:~VE+E^*V\succ 0\quad{\rm with}\quad V\in\mathbf{H}_q~\}.
\]
In particular, then $\mathbf{E}$ is a maximal open set closed under inversion.
\smallskip

\noindent
In in addition $I_q\in\mathbf{E}$ then,
\[
\mathbf{E}=\{E~:~VE+E^*V\succ 0\quad{\rm with}\quad V\in\mathbf{P}_q~\}.
\]
\smallskip

\item[(ii)]{} $\mathbf{E}$ is a maximal set of matrices closed under product among its 
elements, if and only if, Eq. \eqref{eq:Explicit_Qudratic_Form} take the form
\[
\mathbf{E}=\{E~:~X-E^*XE\succ 0\quad{\rm with}\quad X\in\mathbf{H}_q~\}.
\]
If in addition it is closed under sign change of its elements, then
\[
\mathbf{E}=\{E~:~X-E^*XE\succ 0\quad{\rm with}\quad X\in\mathbf{P}_q~\}.
\]
\end{itemize}
\end{Cy}

\section{Proof of Theorem \ref{Tm:Set_Of_HP_Functions}}
\label{ Proof_of_HP_T_structure_theorem}

{\bf Proof of item (i) of Theorem \ref{Tm:Set_Of_HP_Functions}}
\smallskip

For the first part, use Eq. \eqref{eq:Quadratic_HP_Delta} and substitute in item (iii) of Lemma
\ref{La:Structural_Properties}, $q=m$, $E=F(s)$, \mbox{$X=Y=-{\color{blue}T}$} and \mbox{$V=-I_m~$.}\\
Now, when \mbox{$I_m\succ{\color{blue}T}\succ 0$,} one can write
\mbox{$F(s)+(F(s))^*\succ\begin{smallmatrix}{\color{blue}T}\end{smallmatrix}$,}
$~\forall s\in\C_R~$, so invertibility is guaranteed.
\smallskip

Recall that for $I_m\succ{\color{blue}T}\succcurlyeq 0$, having $G\in\mathcal{HB}_{\color{blue}T}$, see
Eq. \eqref{eq:Quad_Def_HB_T}, is equivalent to \mbox{$F(s)=\mathcal{C}\left(G(s)\right)$} being a
$\mathcal{HP}_{\color{blue}T}$ function, see Eq. \eqref{eq:Def_F_and_Hat_F}. Now,
$G\in\mathcal{HB}_{\color{blue}T}$ can be equivalently written as,
\begin{equation}\label{eq:Hyper_Bounded_Def_new}
\left(\begin{smallmatrix}(I_m-{\color{blue}T})\end{smallmatrix}
-{G(s)}^*\begin{smallmatrix}(I_m+{\color{blue}T}\end{smallmatrix})G(s)\right)
\in\overline{\mathbf P}_m\quad\quad\quad\forall s\in\C_R~.
\end{equation}
In this case, we shall say that $G(s)$ is 
$\begin{smallmatrix}{\color{blue}T}\end{smallmatrix}$-{\em Hyper-Bounded} function, denoted by
$\mathcal{HB}_{\color{blue}T}$.
\smallskip

Next, note that multiplying Eq. \eqref{eq:Hyper_Bounded_Def_new} by 
$(I_m-{\color{blue}T})$ from both sides, yields
(for simplicity we temporarily omit the dependence on $s$),
\begin{equation}\label{eq:Hyper_Bounded_Def_new_alternative1}
(I_m-\underbrace{(\begin{smallmatrix}(I_m+{\color{blue}T})\end{smallmatrix}^{\frac{1}{2}}G
\begin{smallmatrix}(I_m-{\color{blue}T})\end{smallmatrix}^{-\frac{1}{2}})^*}_{{\tilde{G}}^*}
\underbrace{\begin{smallmatrix}(I_m+{\color{blue}T})\end{smallmatrix}^{\frac{1}{2}}G
\begin{smallmatrix}(I_m-{\color{blue}T})\end{smallmatrix}^{-\frac{1}{2}}}_{\tilde{G}})
\in\overline{\mathbf P}_m\quad\quad\quad\forall s\in\C_R~.
\end{equation}
Namely, $\tilde{G}(s)$ is in $\mathcal{B}$. Clearly this is equivalent to having
also $(\tilde{G}(s^*))^*$ in $\mathcal{B}$,
\[
I_m-\underbrace{\begin{smallmatrix}(I_m+{\color{blue}T})\end{smallmatrix}^{\frac{1}{2}}G
\begin{smallmatrix}(I_m-{\color{blue}T})\end{smallmatrix}^{-\frac{1}{2}}}_{\tilde{G}}
\underbrace{(\begin{smallmatrix}(I_m+{\color{blue}T})\end{smallmatrix}^{\frac{1}{2}}G
\begin{smallmatrix}(I_m-{\color{blue}T})\end{smallmatrix}^{-\frac{1}{2}})^*}_{{\tilde{G}}^*}
\in\overline{\mathbf P}_m\quad\quad\quad\forall s\in\C_R~.
\]
One can multiply this relation by
$\begin{smallmatrix}(I_m+{\color{blue}T})\end{smallmatrix}^{-\frac{1}{2}}$ from both sides
to obtain
\[
\begin{smallmatrix}(I_m+{\color{blue}T})\end{smallmatrix}^{-1}
-G\begin{smallmatrix}(I_m-{\color{blue}T})\end{smallmatrix}^{-1}G^*\in\overline{\mathbf P}_m
\quad\quad\quad\forall s\in\C_R~.
\]
Applying again the Cayley transform $F=\mathcal{C}(G)$ yields,
\[
\begin{smallmatrix}(I_m+{\color{blue}T})\end{smallmatrix}^{-1}-\underbrace{(I_m+F)^{-1}(I_m-F)}_G
\begin{smallmatrix}(I_m-{\color{blue}T})\end{smallmatrix}^{-1}\underbrace{(I_m-F^*)(I_m+F^*)^{-1}}_{G^*}
\in\overline{\mathbf P}_m\quad\quad\quad\forall s\in\C_R~.
\]
Now, multiplying by $(I_m+F)$ and $(I_m+F^*)$ from the left and the right respectively results in
\[
(I_m+F)
\begin{smallmatrix}(I_m+{\color{blue}T})\end{smallmatrix}^{-1}
(I_m+F^*)
-
(I_m-F)
\begin{smallmatrix}(I_m-{\color{blue}T})\end{smallmatrix}^{-1}
(I_m-F^*)
\in\overline{\mathbf P}_m
\quad\quad\quad\forall s\in\C_R~.
\]
Using the fact that
\[
(I_m-{\color{blue}T})^{-1}
-
(I_m+{\color{blue}T})^{-1}
=
2{\color{blue}T}
(I_m-{\color{blue}T}^2)^{-1}
\quad{\rm and}\quad
(I_m-{\color{blue}T})^{-1}
+
(I_m+{\color{blue}T})^{-1}
=
2(I_m-{\color{blue}T}^2)^{-1},
\]
one can now conclude that
\[
\begin{smallmatrix}(I_m-{\color{blue}T}^2)^{\frac{1}{2}}\end{smallmatrix}
F(s)
\begin{smallmatrix}(I_m-{\color{blue}T}^2)^{-\frac{1}{2}}\end{smallmatrix}
+
\begin{smallmatrix}(I_m-{\color{blue}T}^2)^{-\frac{1}{2}}\end{smallmatrix}
(F(s^*))^*
\begin{smallmatrix}(I_m-{\color{blue}T}^2)^{\frac{1}{2}}\end{smallmatrix}
\succcurlyeq
\begin{smallmatrix}{\color{blue}T}\end{smallmatrix}
+
F(s)
\begin{smallmatrix}{\color{blue}T}\end{smallmatrix}
(F(s^*))^*
\quad\quad\forall s\in\C_R~,
\]
so this part is established.

\qed

{\bf Proof of item (iii) of Theorem \ref{Tm:Set_Of_HP_Functions}}
\smallskip

Multiply Eq. \eqref{eq:HP_Delta} by $U^*$ and $U$ from the left and right, respectively.
\bigskip

{\bf Proof of item (iv) of Theorem \ref{Tm:Set_Of_HP_Functions}}
\smallskip

We shall show it in two ways:\\
Using Eq. \eqref{eq:Quadratic_HP_Delta}, substitute in item (ii) of Lemma \ref{La:Structural_Properties}, 
$q=m$, $E=F(s)$, \mbox{$X=Y=-{\color{blue}T}$.} Since, by assumption \mbox{$-X\succcurlyeq 0$,} the claim
is established.
\smallskip

{\bf Proof of item (v) of Theorem \ref{Tm:Set_Of_HP_Functions}}
\smallskip

Using Eq. \eqref{eq:Quadratic_HP_Delta}, substitute in item (vii) of Lemma \ref{La:Structural_Properties}, 
$q=m$, $E=F(s)$, \mbox{$V=I_m$} and \mbox{$X=Y=-\begin{smallmatrix}{\color{blue}\beta}\end{smallmatrix}I_m$,}
with \mbox{$\begin{smallmatrix}{\color{blue}\beta}\end{smallmatrix}\in[0,~1)$,}
so the claim is established.
\qed
\smallskip

For completeness we point out the following.

\begin{Rk}
{\rm
In numerous works, see e.g.  \cite{BridForb2016}, \cite{Gupta1996} and
\cite{XieGahiAbrouBuhrLaro2020} a related family of functions (sometimes called ``b-bounded" or
``conic sector bound") is addressed. $\Phi(s)$ is $m\times m$-valued rational function analytic
in $\C_R$ and in addition, there exists $b\in[0,~1]$ so that
\begin{center}
$
\frac{1}{2}(1-b^2)\left(\Phi(s)+{\Phi(s)}^*\right)
\succcurlyeq
b\left({\Phi(s)}^*\Phi(s)-I_m\right)
\quad\quad\forall s\in{i}\R.
$
\end{center}
Note that at the end points, two classical sets are obtained: For $b=0$ this is the set $\mathcal{P}$
and for $b=1$ this is $\mathcal{B}$.
$\T$
}
\end{Rk}

\section{Realization of Quantitatively Hyper-Positive Real Rational Functions}
\label{Sec:Realizations}
\setcounter{equation}{0}

\subsection{Proof of item (i) of Theorem \ref{Tm:Kyp_Hyper_Pos_W}}
\smallskip

Recall that for $m\times m$-valued rational functions $F(s)$, analytic in $\mathbb{C}_R$,
the classical Kalman-Yakubovich-Lemma can be casted in the following quadratic form.
For details, see \cite{Lewk2021b}.

\begin{La}\label{La:Quadratic_KYP}
Let $F(s)$ be an $m\times m$-valued rational function, analytic in $\mathbb{C}_R$ and let
$R_F$ be a \mbox{$(n+m)\times(n+m)$} corresponding realization.
\[
F(s)=C(sI_n-A)^{-1}B+D\quad\quad\quad
R_F=\left({\footnotesize\begin{array}{c|c}A&B\\ \hline C&D\end{array}}\right).
\]
Let also $M\in\mathbf{H}_{2m}$ be partitioned to $m\times m$ blocks, so that
\[
M=\left(\begin{smallmatrix}m_{11}&&m_{12}\\~\\m_{12}^*&&m_{22}\end{smallmatrix}\right).
\]
Then the following are equivalent.
\begin{itemize}
\item[(i)~~]{}
\mbox{$
\left(\begin{smallmatrix}F(s)\\~\\I_m\end{smallmatrix}\right)^*
\underbrace{\left(\begin{smallmatrix}m_{11}&&m_{12}\\~\\m_{12}^*&&m_{22}\end{smallmatrix}\right)}_M
\left(\begin{smallmatrix}F(s)\\~\\I_m\end{smallmatrix}\right)
\in\overline{\mathbf P}_m$}\quad
$\forall s\in\C_R$.
\bigskip

\item[(ii)~]{}{} There exists $H\in\mathbf{P}_n$
so that
\begin{equation}\label{eq:KYP_Quadratic_C_R}
\left(\begin{smallmatrix}A&&B\\~\\C&&D\\~\\I_n&&0\\~\\0&&I_m\end{smallmatrix}\right)^*
\underbrace{\left(\begin{smallmatrix}0&&0&&-H&&0\\~\\0&&m_{11}&&0&&m_{12}\\~\\-H&&0&&0&&0
\\~\\0&&{m_{12}}^*&&0&&m_{22}\end{smallmatrix}\right)}_V\left(\begin{smallmatrix}
A&&B\\~\\C&&D\\~\\I_n&&0\\~\\0&&I_m\end{smallmatrix}\right)\in\overline{\mathbf P}_{n+m}\, .
\end{equation}
\end{itemize}
\end{La}

In the classical cases one has the following:
Using Eq. \eqref{eq:Quad_Def_P}, for the Positive Real Lemma take $
M_{\mathcal{P}}=\left(\begin{smallmatrix}0~&&I_m\\~\\I_m&&0\end{smallmatrix}\right)$.\\
Similarly, for the Bounded Real Lemma, use Eq. \eqref{eq:Quad_Def_B} and take $M_{\mathcal{B}}=
\left(\begin{smallmatrix}-I_m&&0\\~\\~~0&&I_m\end{smallmatrix}\right)$.
\smallskip

Now, from Eq. \eqref{eq:Quadratic_HP_Delta} one that
\[
M_{\mathcal{HP}_{\color{blue}T}}
=\left(\begin{smallmatrix}-{\color{blue}T}&&~~I_m\\~\\~~I_m&&-{\color{blue}T}
\end{smallmatrix}\right),
\]
and thus the following is obtained.

\begin{La}\label{La:Quadratic_KYP_HP_T}
Let $F(s)$ be an $m\times m$-valued rational function, and let
$R_F$ be a \mbox{$(n+m)\times(n+m)$} corresponding realization.
\[
F(s)=C(sI_n-A)^{-1}B+D\quad\quad\quad
R_F=\left({\footnotesize\begin{array}{c|c}A&B\\ \hline C&D\end{array}}\right).
\]
Let also \mbox{$I_m\succ{\color{blue}T}\succcurlyeq 0$} and $H\in\mathbf{P}_n$ be parameters. 
\smallskip

Then the two following statements are equivalent.
\begin{itemize}
\item[(i)~~~]{}
\begin{equation}\label{eq:KYP_Quadratic_HP_T}
\left(\begin{smallmatrix}
A&&B\\~\\C&&D\\~\\I_n&&0\\~\\0&&I_m
\end{smallmatrix}\right)^*
\underbrace{
\left(\begin{smallmatrix}
~0&&~0&&-H&&~0\\~\\~0&&-{\color{blue}T}&&~0&&~I_m\\~\\-H&&~0&&~0&&~0\\~\\~0&&~I_m&&~0&&-{\color{blue}T}
\end{smallmatrix}\right)}_V
\left(\begin{smallmatrix}
A&&B\\~\\C&&D\\~\\I_n&&0\\~\\0&&I_m
\end{smallmatrix}\right)
\in\overline{\mathbf P}_{n+m}\, .
\end{equation}
\item[(ii)~~]{}\quad
$
\left(\begin{smallmatrix}-H&&0\\~\\0&&I_m\end{smallmatrix}\right)R_F+{R_F}^*
\left(\begin{smallmatrix}-H&&0\\~\\0&&I_m\end{smallmatrix}\right)\succcurlyeq
\left(\begin{smallmatrix}C~&&D\\~\\0_{m\times n}&&I_m\end{smallmatrix}\right)^*
\left(\begin{smallmatrix}{\color{blue}T}&&0\\~\\0&&{\color{blue}T}\end{smallmatrix}\right)
\left(\begin{smallmatrix}C~&&D\\~\\0_{m\times n}&&I_m\end{smallmatrix}\right)
$
\end{itemize}
and imply that 
\begin{itemize}
\item[(iii)~]{}
$F(s)$ is a $\mathcal{HP}_{\color{blue}T}$ function, i.e. it satisfies
\[
\left(\begin{smallmatrix}F(s)\\~\\I_m\end{smallmatrix}\right)^*
\left(\begin{smallmatrix}-{\color{blue}T}&&~~I_m\\~\\~~I_m&&-{\color{blue}T}
\end{smallmatrix}\right)
\left(\begin{smallmatrix}F(s)\\~\\I_m\end{smallmatrix}\right)
\in\overline{\mathbf P}_m\quad\forall s\in\C_R~.
\]
\end{itemize}
If the realization is minimal then {\rm (iii)} imply {\rm (i)} and {\rm (ii)}.
\end{La}

Indeed, item (i) is just a quadratic form of item (ii) which in turn is Eq.
\eqref{eq:HP_Delta_KYP}.\\
Item (i) and (iii) are a special case of Lemma \ref{La:Quadratic_KYP}. 
\smallskip

Thus item (i) in Theorem \ref{Tm:Kyp_Hyper_Pos_W} is established.


\subsection{Proof of items (ii), (iii) of Theorem \ref{Tm:Kyp_Hyper_Pos_W}}
\label{Subsec:Balanced}
We start by recalling the classical notion of {\em balanced realization.}
\smallskip

Let $F(s)$ be a $p\times m$-valued rational function, with no pole at infinity, of McMillan degree $n$,
and let \mbox{$R_F=\left({\footnotesize\begin{array}{c|c}A&B\\ \hline C&D\end{array}}\right)$} be a
corresponding minimal realization. Recall that if $A$ is Hurwitz-stable (i.e.  spectrum in $\C_L$)
then there exists a pair of $\mathbf{P}_n$ matrices $H_{\rm cont}$ and $H_{\rm obs}$ satisfying
\[
H_{\rm cont}A^*+AH_{\rm cont}=-BB^*\quad\quad{\rm and}\quad\quad H_{\rm obs}A+A^*H_{\rm obs}=-C^*C.
\]
Then $H_{\rm cont}$ and $H_{\rm obs}$ are called the Controllability and Observability Gramians,
respectively.
\smallskip

Let $V\in\C^{n\times n}$ be a non-singular. Applying coordinate transformation
means taking \mbox{$\hat{R}_F:=\left(\begin{smallmatrix}V&&0\\~\\0&&I_p\end{smallmatrix}\right)^{-1}
R_F
\left(\begin{smallmatrix}V&&0\\~\\0&&I_p\end{smallmatrix}\right)$} yielding another minimal
realization of the same $F(s)$. In this case, the Gramian equations turns to be
\[
\begin{matrix}
\underbrace{(V^*H_{\rm cont}^{-1}V)^{-1}}_{\hat{H}_{\rm cont}}
\underbrace{(V^{-1}AV)^*}_{{\hat{A}}^*}
+\underbrace{V^{-1}AV}_{\hat{A}}
\underbrace{(V^*H_{\rm cont}^{-1}V)^{-1}}_{\hat{H}_{\rm cont}}
&=&-\underbrace{V^{-1}B}_{\hat{B}}\underbrace{(V^{-1}B)^*}_{{\hat{B}}^*}
\\
\overbrace{
V^*H_{\rm obs}V}^{\hat{H}_{\rm obs}}\overbrace{V^{-1}AV}^{\hat{A}}+\overbrace{(V^{-1}AV)^*}^{{\hat{A}}^*}
\overbrace{V^*H_{\rm obs}V}^{\hat{H}_{\rm obs}}&=&-\overbrace{(CV)^*}^{{\hat{C}^*}}
\overbrace{CV}^{\hat{C}}.
\end{matrix}
\]
One can always choose $V$ so that
\[
\hat{H}_{\rm cont}=\hat{H}_{\rm obs}~,
\]
and then the resulting
realization is said to be {\em balanced}.
For a specific example, see the system in Eq.
\eqref{eq:Example_Balanced_Realization}, and for details e.g. \cite[Section 3.9]{Zhou1996}. For
completeness, we recall that if ${\rm spec}(A)$ is not confined to single open half planes, a
balancing transformation may not exist. For details see \cite{KenneyHewer1987}.
\smallskip

Relying on Theorem \ref{Tm:Kyp_Hyper_Pos_W} and Propositions \ref{Pn:SP_and_HP},
\ref{Pn:Sets_Of_Realizations}, in the sequel, we shall technically manipulate realization arrays of
Hyper-Positive functions as if they were matrices. However, {\em in a general framework}, the
motivation to do that, is questionable. This is illustrated next.

\begin{Ex}\label{Ex:Caution_Inversion}
{\rm
{\bf a.}~
Consider a scalar $\mathcal{P}$ function $f(s)$, along with its balanced realization $R_f$
\[
f(s)=\begin{smallmatrix}\frac{s}{s+1}\end{smallmatrix}\quad\quad
R_f
=\left({\footnotesize\begin{array}{r|r}-1&-1\\ \hline 1&1\end{array}}\right).
\]
The rational function $f(s)$ is well defined, but as a matrix, $R_f$ is singular.
\smallskip

{\bf b.}~
Consider a scalar Hurwitz stable (but not in $\mathcal{P}$) rational functions $f(s)$, along with
its minimal realization $R_f$,
\[
f(s)=
\begin{smallmatrix}
\frac{1}{(s+1)^2}
\end{smallmatrix}
\quad
\quad
\quad
R_f=\left({\footnotesize\begin{array}{rr|r}-1&1&0\\0&-1&1\\ \hline 1&0&0\end{array}}\right).
\]
Treating $R_f$ as a matrix, take \mbox{$R_{\hat{f}}:=(R_f)^{-1}$} to be a
(minimal) realization of another rational function $\hat{f}(s)$
\[
\hat{f}(s)=
\begin{smallmatrix}
\frac{1}{s^2}+\frac{2}{s}
\end{smallmatrix}
\quad
\quad
\quad
R_{\hat{f}}:={R_f}^{-1}=\left({\footnotesize\begin{array}{cc|c}0&0&1\\1&0&1\\
\hline 1&1&1\end{array}}\right).
\]
Now, $\hat{f}(s)$ is {\em unstable}.
}
$\T$
\end{Ex}

In contrast to the above, item (iii) of Theorem \ref{Tm:Kyp_Hyper_Pos_W} below, focuses on a framework
where inversion of realization arrays is meaningful.
\smallskip

The following is used in the sequel.

\begin{La}\label{La:R_Non-singular}
For $I_m\succ{\color{blue}T}\succ 0$, consider again Eq. \eqref{eq:HP_Delta_KYP},
\[
\left(\begin{smallmatrix}-H&&0\\~\\0&&I_m\end{smallmatrix}\right)
\underbrace{
\left(\begin{smallmatrix}A&&B\\~\\C&&D\end{smallmatrix}\right)
}_{R_F}
+
\underbrace{
\left(\begin{smallmatrix}A&&B\\~\\C&&D\end{smallmatrix}\right)^*
}_{{R_F}^*}
\left(\begin{smallmatrix}-H&&0\\~\\0&&I_n\end{smallmatrix}\right)
\succcurlyeq
\underbrace{
\left(\begin{smallmatrix}C~&&D\\~\\0_{m\times n}&&I_m\end{smallmatrix}\right)^*
\left(
\begin{smallmatrix}{\color{blue}T}&&0\\~\\
0&&{\color{blue}T}\end{smallmatrix}\right)
\left(\begin{smallmatrix}C~&&D\\~\\0_{m\times n}&&I_m\end{smallmatrix}\right)
}_{Q},
\]
for realizations of ${\mathcal{HP}}_{\color{blue}T}$ functions (or the simpler Eq. \eqref{eq:M_HP_beta}
for realizations of $\mathcal{HP}_{\color{blue}\beta}$ functions, with
$\begin{smallmatrix}{\color{blue}\beta}\end{smallmatrix}\in(0,~1)$).
\smallskip

The pair $A$, $C$ is observable, if and only if,
the pair $R_F$, $Q$ is observable.
\smallskip

In this case, the realization array $R_F$, can be viewed as a \mbox{$(n+m)\times(n+m)$} 
non-singular matrix.
Furthermore, this matrix has $n$ eigenvalues in $\C_L$ and $m$ eigenvalues in $\C_R~$.
\end{La}

{\bf Proof :}~
By the Popov-Belevich-Hautus test, see e.g. \cite[Theorem 2.14]{DullPaga2000}, \cite[Theorem 3.3]{Zhou1996},
the pair $A$, $C$ is unobservable, if and only if, there exists a pair \mbox{$0\not=v\in\C^n$} and
$\lambda\in\C$, so that
\begin{equation}\label{eq:A,C_Observable}
\begin{smallmatrix}n+m~\end{smallmatrix}
\lbrace{\left(\begin{smallmatrix}A-\lambda{I}_n\\~\\C\end{smallmatrix}\right)}
\begin{smallmatrix}v\end{smallmatrix}
=
0_{(n+m)\times 1}~.
\end{equation}
In  a similar way, the pair $R_F, Q$ is unobservable, if and only if, there exists a pair
\mbox{$0\not=u\in\C^{n+m}$} and $\mu\in\C$, so that
\[
\begin{smallmatrix}
2(n+m)~
\end{smallmatrix}
\lbrace{\left(\begin{smallmatrix}R_F-\mu{I}_{n+m}\\~\\Q\end{smallmatrix}\right)}
\begin{smallmatrix}
u
\end{smallmatrix}
=
0_{2(n+m)\times 1}~.
\]
For some $u_n\in\C^n$ and $u_m\in\C^m$, partition now
\mbox{$\begin{smallmatrix}u\end{smallmatrix}
=\left(\begin{smallmatrix}u_n\\~\\u_m\end{smallmatrix}\right)$} 
so that
\[
\left(\begin{smallmatrix}R_F-\mu{I}_{n+m}\\~\\Q\end{smallmatrix}\right)
\left(\begin{smallmatrix}u_n\\~\\u_m\end{smallmatrix}\right)
=
\left(
\begin{smallmatrix}
\begin{smallmatrix}
A-\mu{I}_n&&B\\~\\C&&D-\mu{I}_m
\end{smallmatrix}
\\~\\~\\
\left(\begin{smallmatrix}C~&&D\\~\\0_{m\times n}&&I_m\end{smallmatrix}\right)^*
\left(
\begin{smallmatrix}{\color{blue}T}&&0\\~\\0&&{\color{blue}T}\end{smallmatrix}\right)
\left(\begin{smallmatrix}C~&&D\\~\\0_{m\times n}&&I_m\end{smallmatrix}\right)
\end{smallmatrix}\right)
\left(\begin{smallmatrix}u_n\\~\\u_m\end{smallmatrix}\right)
=0_{2(n+m)\times 1}~.
\]
Note that the lower-right $m\times m$ block of $Q$ is, \mbox{$\left(\begin{smallmatrix}D\\~\\
I_m\end{smallmatrix}\right)^*\begin{smallmatrix}{\color{blue}T}\end{smallmatrix}\left(
\begin{smallmatrix}D\\~\\I_m\end{smallmatrix}\right)$}, so it is positive-definite. This implies
that $u_m$ (the lower part of the vector $u$) vanishes. It thus turns out that, for some
\mbox{$0\not=c\in\C$}, one can always take \mbox{$u=\begin{smallmatrix}c\end{smallmatrix}
\left(\begin{smallmatrix}v\\~\\0\end{smallmatrix}\right)$}, with $v$ from
Eq. \eqref{eq:A,C_Observable} (and then $\lambda=\mu$). Hence, this part of the claim is established.
\qed
\bigskip

{\bf Proof of item (ii) of Theorem \ref{Tm:Kyp_Hyper_Pos_W}}\\
Consider a Lyapunov equation of the form $\hat{H}L+L^*\hat{H}=Q$, with $\hat{H}\in\mathbf{H}_q$,
$Q\in\overline{\mathbf P}_q$ and for all \mbox{$n\geq j\geq k\geq 1$}, the eigenvalue of $L$ satisfy
\mbox{${\lambda}_j+{{\lambda}_k}^*\not=0$}. Recall now (see e.g.
\cite[Theorems 2.4.7 and 2.4.10]{HornJohnson2}, \cite[Lemma 3.19]{Zhou1996}) that if the pair $L, Q$ is
observable, the matrices $\hat{H}$ and $L$ share the same regular inertia.
\smallskip

In our case $q=n+m$, \mbox{$\hat{H}=\left(\begin{smallmatrix}-H&&0\\~\\0&&I_m\end{smallmatrix}\right)$,}
with $H\in\mathbf{P}_n$ and $L=R_F$, so each $R_F$ has $n$ eigenvalues in $\C_L$ and $m$ eigenvalues in
$\C_R$. In particular, $R_F$ is non-singular. So the claim is established.
\qed
\smallskip

We are now ready to the following.
\bigskip

{\bf Proof of item (iii) of Theorem \ref{Tm:Kyp_Hyper_Pos_W}}\\
By assumption the realization $R_F$ is minimal and hence the pair $A,~C$ is  observable. Thus,
non-singularity of $R_F$ was established in Lemma \ref{La:R_Non-singular}.
\smallskip

Now, to show that this set is closed under inversion, substitute Eq. \eqref{eq:KYP_Quadratic_HP_T}
in item (iii) of Lemma \ref{La:Structural_Properties} 
to obtain $q=n+m$, and the three following \mbox{$(n+m)\times(n+m)$} matrices: $E=R_F$,
\mbox{$X=Y=\left(\begin{smallmatrix}0_{n\times n}&~0\\~\\0&-{\color{blue}T}\end{smallmatrix}\right)$}
and \mbox{$V=\left(\begin{smallmatrix}-H&0~\\~\\~0&I_m\end{smallmatrix}\right)$,}
so the claim is established.
\qed
\smallskip

We next show that, the last fact can be actually made explicit.

\begin{Cy}
Let $F(s)$ be a $m\times m$ valued rational function of McMillan degree $n$ and let
\mbox{$R_F=\left( {\footnotesize\begin{array}{c|c}A&B\\ \hline C&D\end{array}}\right)$} be
a corresponding realization, i.e. $A$ and $D$ are $n\times n$ and $m\times m$ respectively.

{\bf (i)}~ If all three matrices $R_F$, $A$~ and $D$ are non-singular, then
\[
R_{\hat{F}}:=(R_F)^{-1}=\left({\footnotesize\begin{array}{c|c}(A-BD^{-1}C)^{-1}&A^{-1}B(CA^{-1}B-D)^{-1}
\\ \hline D^{-1}C(BD^{-1}C-A)^{-1}&(D-CA^{-1}B)^{-1}\end{array}}\right).
\]
{\bf (ii)}~ 
If $F\in\mathcal{HP}_{\color{blue}T}$, for some \mbox{$I_m\succ{\color{blue}T}\succ 0$,} then $R_F$, $A$
and $D$ are non-singular, and the resulting $R_{\hat{F}}$ is a realization of $\hat{F}(s)$, which belongs
to the same $F\in\mathcal{HP}_{\color{blue}T}$.
\end{Cy}

Indeed, the first part is essentially a re-writing of \cite[Eq. (0.7.3.1)]{HJ1}. The second part is
immediate from items (ii), (iii) of Theorem \ref{Tm:Kyp_Hyper_Pos_W} along with Lemma
\ref{La:R_Non-singular}.

\subsection{Two kinds of Inversion}
\label{Subsec:Two_Inversions}
 
First, we recall an important subset of $\mathcal{P}$ functions: Positive Real Odd functions (a.k.a. Lossless
Positive or Foster), denoted by $\mathcal{PO}$. For details see e.g. \cite[Theorem 2.7.4]{AnderVongpa1973},
\cite[Ch 8, items 36-50]{Belev1968}, \cite[Section 4.2]{CohenLew2007}, \cite{Lewk2021a} and
\cite[p. 36]{Wohl1969}.  Using Eq. \eqref{eq:Quad_Def_P}, one can describe $\mathcal{PO}$ functions as,
\begin{equation}\label{eq:Quad_Def_PO}
\left(\begin{smallmatrix}F(s)\\~\\I_m\end{smallmatrix}\right)^*
\left(\begin{smallmatrix}0&&I_m\\~\\I_m&&0\end{smallmatrix}\right)
\left(\begin{smallmatrix}F(s)\\~\\I_m\end{smallmatrix}\right)~
\left\{
\begin{smallmatrix}
~&\preccurlyeq 0,&&\forall s\in\C_L\\~\\
~&=0&&\forall s\in{i}\R
\\~\\
~&\succcurlyeq 0&&\forall s\in\C_R~.
\end{smallmatrix}
\right.
\end{equation}
From part {\bf a} of Theorem \ref{Tm:Set_Of_P_Functions} we know that the set $\mathcal{P}$ is closed under
inversion. It turns out that that its subset $\mathcal{PO}$ is closed under inversion as well: Note that in
Eq.  \eqref{eq:Quad_Def_PO} $F\preccurlyeq 0$ is equivalent to $F^{-1}\preccurlyeq 0$, whenever the inverse
exists. The same is true for ``$=$" and for ``$\succcurlyeq 0$".
\smallskip

In fact, since each $\mathcal{P}$ and $\mathcal{PO}$,
is closed under inversion, then so is the set \mbox{$\mathcal{P}\smallsetminus\mathcal{PO}$.}
\smallskip

We next illustrate the fact that item (iii) of Theorem \ref{Tm:Kyp_Hyper_Pos_W}
is not self-evident.

\begin{Ex}\label{Ex:SP_Inversion}
{\rm
We shall next show that for arbitrary ${\color{blue}T}$, \mbox{$I_m\succ{\color{blue}T}\succcurlyeq 0$},
the following holds
\begin{equation}\label{eq:Inclusions}
\mathcal{HP}_{\color{blue}T}\subset\mathcal{SP}\subset\{\mathcal{P}\smallsetminus\mathcal{PO}\},
\end{equation}
and each inclusion is strict. Then note that, as rational function, each of these three sets is closed
under inversion. In contrast, in the framework of realization arrays, only the two extreme sets are
closed under matrix-inversion, but not the set $\mathcal{SP}$. Here are the details.
\smallskip

Consider Eq. \eqref{eq:Inclusions}: The strict inclusion on the left-hand side follows from Proposition
\ref{Pn:SP_and_HP}. For the strict inclusion on the right-hand side, by part {\bf a.} of Definition
\ref{Dn:SP}, on $i\R$, $\mathcal{SP}$ functions can have neither poles nor zeros.
\smallskip

{\em Inversion of Functions:}\\
Here we note that for each of the three sets one has that when $F(s)$ belongs to it, then so is
$(F(s))^{-1}$:~ For $\mathcal{HP}_{\color{blue}T}$ functions see item (i) of Theorem
\ref{Tm:Set_Of_HP_Functions}.\\
For $\mathcal{SP}$, combine Definition \ref{Dn:SP} with part {\bf a} of Theorem
\ref{Tm:Set_Of_P_Functions}.\\
For the set on the right-hand side, see the discussion below Eq. \eqref{eq:Quad_Def_PO}.
\smallskip

{\em Inversion of Realization arrays:}\\
From items (ii) and (iii) of Theorem \ref{Tm:Kyp_Hyper_Pos_W} we know that if $R_F$ is a minimal
realization of $F\in\mathcal{HP}_{\color{blue}T}$ then as a matrix is non-singular and if we denote
\mbox{$R_{\hat{F}}:={R_F}^{-1}$} then $R_{\hat{F}}$ is a realization of $\hat{F}$, another function
within the same $\in\mathcal{HP}_{\color{blue}T}$.
\smallskip

To see that in terms of realization arrays, the set $\mathcal{P}$ is closed under inversion, whenever
$R_F$ is non-singular, multiply Eq. \eqref{eq:Original_KYP} by $({R_F}^{-1})^*$ and ${R_F}^{-1}$ from
the left and from the right, respectively. For $\mathcal{PO}$ functions, repeat the same process
recalling that in this case the right-hand side of Eq. \eqref{eq:Original_KYP} vanishes, see e.g.
\cite{ag1}. As before since both $\mathcal{P}$ and $\mathcal{PO}$ are
closed under inversion, then so is the set \mbox{$\{\mathcal{P}\smallsetminus\mathcal{PO}\}$.}
\smallskip

Finally, in the framework of realization arrays, the set $\mathcal{SP}$ is {\em not} closed under
inversion: Return to item {\bf b.} in Example \ref{Ex:Caution_Inversion} and take the scalar
$\mathcal{SP}$ function $f_{+}(s)=\frac{1}{s+1}$ and its balanced realization is
\mbox{$R_{f_+}=\left({\footnotesize\begin{array}{r|r}-1&1\\ \hline 1&0\end{array}}\right)$.} Now let
\mbox{$R_{\hat{f}_{+}}:={R_{f_{+}}}^{-1}=\left({\footnotesize\begin{array}{c|c}0&1\\
\hline 1&1\end{array}}\right)$,} which is a realization of the rational function
\mbox{$\hat{f}(s)=\frac{1}{s}+1$.} It turns out that this \mbox{$\hat{f}(s)$} is in
\mbox{$\mathcal{P}\smallsetminus\left\{\mathcal{SP}\bigcup\mathcal{PO}\right\}$.}
\smallskip

One concludes that in
Eq. \eqref{eq:Inclusions} only the two extreme sets are closed under both kinds of inversion.
}
$\T$
\end{Ex}
\smallskip

In the Introduction we pointed out that in the 1970's it was recognized that $\mathcal{P}$ functions can serve
as a model for linear time-invariant {\em passive} systems. In fact, this was a consequence of earlier
earlier development in the area of electric circuits. In \cite{Brune1} it was shown that the driving point
impedance of an arbitrary $R-L-C$ electrical circuit, is a rational $\mathcal{P}$ function. The fact that every
rational $\mathcal{P}$ function may be realized by an $R-L-C$ electrical circuit, was established in
\cite{BottDuff1949} (in the matrix-valued case one may in addition need transformers and gyrators). We next
illustrate the fact that by restricting the structure of the $R-L-C$ circuits, the resulting driving point
impedance is of rational $\mathcal{HP}_{\beta}$ functions.

\begin{figure}[H]
\centering
\begin{minipage}{0.50\linewidth}
{\rm
Our starting point is the electrical circuit to the right: Its driving point impedance is a rational function
$Z_{\rm in}(s)$ and with a corresponding realization $R_Z$. Then from the inverses ${R_Z}^{-1}$ and
$(Z_{\rm in}(s))^{-1}$ we obtain $\hat{Y}$ and $Y_{\rm in}$ respectively. Finally, from ${R_{Y_{\rm in}}}^{-1}$,
one obtains yet another driving point impedance, which in turn corresponds to the circuit in Figure
\ref{Fig:Another_Point_Impedance_Degree_One}.
}
\end{minipage}\quad\begin{minipage}{0.47\linewidth}
\begin{tikzpicture}[scale=1.3]
   \draw[color=black, thick]
      (0,0) to [short,o-] (3.6,0){} 
      (-0.1,0.6) node[]{\large{$\mathbf{Z}_{\rm\bf in}~~\mathbf{\rightarrow}$}}
     (0,1.2) to [short,o-] (0.1,1.2)
     (0.1,1.2)  to [R,l=$\mathbf{R_1}$,](2,1.2)
     (2,1.2)   to node[short]{} (3.6,1.2)
     (3.6,0) to [C,l=$\mathbf{C}$,*-*] (3.6,1.2)
     (2,0) to [R, l=$\mathbf{R_2}$, *-*] (2,1.2)
     ;
\end{tikzpicture}
\caption{
${\rm\bf ~Z}_{\rm\bf in}(s)=R_1+\frac{\frac{1}{C}}{s+\frac{1}{R_2C}}$}
\end{minipage}
\label{Fig:Impedance_Degree_One}
\end{figure}
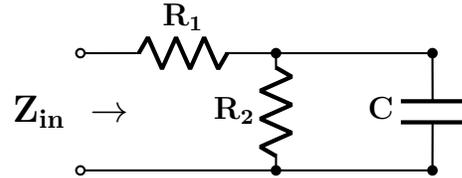

$
\begin{matrix}
{\rm\bf ~Z}_{\rm\bf in}(s)=R_1+\frac{\frac{1}{C}}{s+\frac{1}{R_2C}}
&&&
R_{\rm\bf Z}=\left({\footnotesize\begin{array}{c|c}\frac{-1}{R_2C}&
\frac{1}{\sqrt C}\\ \hline\frac{1}{\sqrt C}&R_1\end{array}}\right)
\\~\\
{\rm\bf ~Y}_{\rm\bf in}(s)=\left({\rm\bf~Z}_{\rm\bf in}(s)\right)^{-1}
=\frac{1}{R_1}\cdot
\frac{s+\frac{1}{CR_1}}
{s+\frac{1}{C}\left(\frac{1}{R_1}+\frac{1}{R_2}\right)}
&&&
R_{\rm\bf Y}=
\left({\footnotesize\begin{array}{c|c}
-\frac{1}{C}\left(\frac{1}{R_1}+\frac{1}{R_2}\right)&\frac{-1}{R_1\sqrt{C}}
\\
\hline
\frac{1}{R_1\sqrt{C}}&\frac{1}{R_1}\end{array}}\right)
\\~\\
{\rm\bf\color{blue}\hat{Y}}(s)=\frac{1}{R_1+R_2}+
\frac{C\frac{{R_2}^2}{(R_1+R_2)^2}}{s+C\left(\frac{1}{R_1}
+\frac{1}{R_2}\right)}
&&&
R_{\rm\bf\color{blue}\hat{Y}}:=\left(R_{\rm\bf Z}\right)^{-1}=
\left({\footnotesize\begin{array}{c|c}
-\frac{CR_1R_2}{R_1+R_2}&\frac{R_2\sqrt{C}}{R_1+R_2}\\
\hline
\frac{R_2\sqrt{C}}{R_1+R_2}&\frac{1}{R_1+R_2}\end{array}}\right)
\\~\\
{\rm\bf\color{red}\hat{Z}}(s)=R_1+\frac{sR_2}{s+CR_2}
&&&
R_{\rm\bf\color{red}\hat{Z}}:=\left(R_{\rm\bf Y}\right)^{-1}=
\left({\footnotesize\begin{array}{c|c}
-CR_2&-R_2\sqrt{C}\\ \hline R_2\sqrt{C}&R_1+R_2
\end{array}}\right).
\end{matrix}
$

\begin{figure}[H]
\centering
\begin{minipage}{0.48\linewidth}
{\rm
The function ${\rm\bf\color{red}\hat{Z}}(s)$ may be realized as the
driving point impedance of the circuit in Figure
\ref{Fig:Another_Point_Impedance_Degree_One} by substituting the value
\mbox{$L=\frac{1}{C}$}.
}
\end{minipage}
\quad
\begin{minipage}{0.47\linewidth}
\begin{tikzpicture}[scale=1.3]
   \draw[color=black, thick]
      (0,0) to [short,o-] (3.6,0){} 
      (-0.1,0.6) node[]{\large{$\mathbf{\color{red}{\hat Z}}~~\mathbf{\rightarrow}$}}
     (0,1.2) to [short,o-] (0.1,1.2)
     (0.1,1.2)  to [R,l=$\mathbf{R_1}$,](2,1.2)
     (2,1.2)   to node[short]{} (3.6,1.2)
     (3.6,0) to [L,l=$\mathbf{L}$,*-*] (3.6,1.2)
     (2,0) to [R, l=$\mathbf{R_2}$, *-*] (2,1.2)
     ;
\end{tikzpicture}
\end{minipage}
\caption{${\rm\bf\color{red}\hat Z}=R_1+
{\frac{sR_2}{s+\frac{R_2}{L}}}_{|_{L=\frac{1}{C}}}=R_1+\frac{sR_2}{s+CR_2}$}
\label{Fig:Another_Point_Impedance_Degree_One}
\end{figure}

\begin{Ex}\label{Ex:Generating_Functions_Inversions}
{\rm
A straightforward computation reveals that in Figure \ref{Fig:Impedance_Degree_One},
\mbox{${\rm\bf ~Z}_{\rm\bf in}(s)=R_1+\frac{\frac{1}{C}}{s+\frac{1}{R_2C}}$} is
a $\mathcal{HP}_{\color{blue}\beta}$ function with
$~
{\scriptstyle\color{blue}\beta}=\left\{\begin{smallmatrix}2\left(R_1+\frac{1}{R_1}\right)^{-1}&&&
R_1\in\left(0,~\sqrt{\scriptstyle(\frac{R_2}{2})^2+1}-{\scriptstyle\frac{R_2}{2}}\right]
\\~\\2\left(R_1+R_2+\frac{1}{R_1+R_2}\right)^{-1}&&&R_1>
\left({\scriptstyle\sqrt{(\frac{R_2}{2})^2+1}}-{\scriptstyle\frac{R_2}{2}}\right).
\end{smallmatrix}\right.
$
\smallskip

Furthermore, upon combining item (i) of Theorem \ref{Tm:Set_Of_HP_Functions} together with item
(iii) of Theorem \ref{Tm:Kyp_Hyper_Pos_W}, it follows that the three subsequent
functions, belong to the same $\mathcal{HP}_{\color{blue}\beta}$.
}
$\T$
\end{Ex}

Note that in item (ii) of Proposition \ref{Pn:SP_and_HP}, the right-hand side of Eq.
\eqref{eq:Original_KYP} is positive definite. In contrast, in Eq. \eqref{eq:HP_Delta_KYP} (or Eq.
\eqref{eq:M_HP_beta}) the right-hand side can be positive {\em semi}-definite. In fact, there is
no contradiction: First there is inequality in Eq. \eqref{eq:HP_Delta_KYP}. Second, for each case,
one can choose a different $H$. Third, Proposition \ref{Pn:SP_and_HP} focuses on
$\mathcal{HP}_{\color{blue}T}$ functions with \mbox{$I_m\succ{\color{blue}T}\succ 0$.} In principle,
characterization of $\mathcal{HP}_{\color{blue}T}$ functions through their realization, is well
defined even for {\em singular} ${\color{blue}T}$, see e.g.
Theorem \ref{Tm:Kyp_Hyper_Pos_W}. This is next illustrated.

\begin{Ex}\label{Ex:HP_T_with_T_singular}
{\rm
We here explore an example of $F\in\mathcal{HP}_{\color{blue}T}$, where ${\color{blue}T}$ is in
\mbox{$\{\overline{\mathbf P}_m\smallsetminus\mathbf{P}_m\}$.}\\
Consider the following $2\times 2$-valued function along with its realization
\begin{equation}\label{eq:Example_Balanced_Realization}
F(s)=\left(\begin{smallmatrix}\frac{{\gamma}^2(2s+3)}{(s+1)(s+2)}+1&&\frac{{\gamma}^2(2s+3)}{(s+1)(s+2)}
\\~\\\frac{{\gamma}^2(2s+3)}{(s+1)(s+2)}&&\frac{{\gamma}^2(2s+3)}{(s+1)(s+2)}\end{smallmatrix}\right)
\quad\quad\quad R_F={\footnotesize\left(\begin{array}{rr|rr}-1&0&\gamma&\gamma\\ 0&-2&\gamma&\gamma
\\ \hline\gamma&\gamma&1&0\\\gamma&\gamma&0&0\end{array}\right)}, 
\end{equation}
where $\begin{smallmatrix}\gamma\end{smallmatrix}\geq 0$ is a parameter.
\smallskip

{\bf a.}~
From Eq. \eqref{eq:M_HP_beta} it follows that whenever $F(s)$ is Hyper-Positive, it implies
that \mbox{$(D+D^*)\in\mathbf{P}_m$.} Since here 
\mbox{$(D+D^*)\in\{\overline{\mathbf P}_2\smallsetminus\mathbf{P}_2\}$,} for all
\mbox{$\begin{smallmatrix}\gamma\end{smallmatrix}\geq 0$}, this $F(s)$ is in fact in
\mbox{$\mathcal{P}\smallsetminus\mathcal{HP}_{\color{blue}\beta}$,} whenever 
\mbox{$\begin{smallmatrix}{\color{blue}\beta}\end{smallmatrix}>0$.} Strictly speaking for
\mbox{$\begin{smallmatrix}\gamma\end{smallmatrix}=0$} this is a non-minimal realization of zero degree
function, and for \mbox{$\begin{smallmatrix}\gamma\end{smallmatrix}>0$,} the realization is in fact
balanced (see Subsection \ref{Subsec:Balanced}).
\smallskip

Now, partitioning Eq. \eqref{eq:HP_Delta_KYP} as \mbox{Left$\succcurlyeq$Right} and substituting
$H=I_2$ yields 
\mbox{${\rm Left}=\left(\begin{smallmatrix}2&0&0&0\\0&4&0&0\\0&0&2&0\\0&0&0&0\end{smallmatrix}\right)$,}
and for \mbox{${\color{blue}T}=\left(\begin{smallmatrix}\frac{1}{2}&&0\\~\\0&&0\end{smallmatrix}\right)$,} 
\mbox{${\rm Right}=\left(\begin{smallmatrix}
\frac{1}{2}{\gamma}^2&\frac{1}{2}{\gamma}^2&\frac{1}{2}\gamma&0\\
\frac{1}{2}{\gamma}^2&\frac{1}{2}{\gamma}^2&\frac{1}{2}\gamma&0\\
\frac{1}{2}\gamma~&\frac{1}{2}\gamma~&1~&0\\0~~&0~~&0~&0\end{smallmatrix}\right)$.} 
In particular \mbox{${\rm Right}\in\overline{\mathbf P}_4$,} is of
\mbox{${\rm rank}=2$.}\\
Now, it is easy to verify that Eq. \eqref{eq:HP_Delta_KYP} is satisfied for all
\mbox{$\begin{smallmatrix}\gamma\end{smallmatrix}\in[0, \begin{smallmatrix}\frac{4}{3}
\end{smallmatrix}]$.}
\smallskip

{\bf b.}~ Denote \mbox{$\Phi(s):=\left(\begin{smallmatrix}1&&
\frac{{\gamma}^2(2s+3)}
{s^2+(2{\gamma}^2+3)s+3{\gamma}^2+2}
\\~\\0&&-1\end{smallmatrix}\right)$,} with
\mbox{$\begin{smallmatrix}\gamma\end{smallmatrix}\geq 0$,} parameter.
Using $F(s)$ from Eq. \eqref{eq:Example_Balanced_Realization}, we construct the following
$2\times 2$-valued function,
\[
F_1(s):={\Phi(s^*)}^*F(s)\Phi(s)=\left(\begin{smallmatrix}
\frac{s^2+(2{\gamma}^2+3)s+3{\gamma}^2+2}{(s+1)(s+2)}&&
0\\~\\0&&
\frac{{\gamma}^2(2s+3)}
{s^2+(2{\gamma}^2+3)s+3{\gamma}^2+2}
\end{smallmatrix}\right).
\]
Note now that in $F_1(s)$, the lower-right entry is a (scalar) $\mathcal{P}$ function with a zero at
infinity, so it cannot be Hyper-Positive. In contrast, the upper-left entry is in
$\mathcal{HP}_{\color{blue}\beta}$ with
\[
\begin{smallmatrix}{\color{blue}\beta}\end{smallmatrix}=
\left(
\begin{smallmatrix}\frac{1}{2}\end{smallmatrix}(
\begin{smallmatrix}\frac{3}{2}{\gamma}^2+1\end{smallmatrix}+
(
\begin{smallmatrix}\frac{3}{2}{\gamma}^2+1\end{smallmatrix})^{-1}\right)^{-1}
\quad\quad\forall~\begin{smallmatrix}\gamma\end{smallmatrix}\geq 0.
\]
Roughly speaking, when a function $F(s)$ is in $\mathcal{HP}_{\color{blue}T}$ and ${\color{blue}T}$ is
singular, this $F(s)$ behaves like a block-diagonal function with a
$\mathcal{P}\smallsetminus\mathcal{HP}$ part, along with a $\mathcal{HP}_{\color{blue}T}$ part.
\smallskip

The significance of having \mbox{${\color{blue}T}\in\{\overline{\mathbf P}_m\smallsetminus\mathbf{P}_m\}$}
is further discussed in Example \ref{Ex:Truncation} below.
}

$\T$
\end{Ex}

\section{$n, m$ Matrix-Convexity of Realizations of $\mathcal{HP}_{\color{blue}\beta}$ Functions}

The lion-share of this section is devoted to the following question: Under what conditions a matrix-convex 
combination of realizations of a given collection of Hyper-Positive functions, yields a realization of
another Hyper-Positive function.

\subsection{Preliminaries}
To illustrate a difficulty involved, we next demonstrate the fact that the family of realization arrays of
Hyper-Positive functions, is not closed under unitary similarity and thus, strictly speaking, can not be
matrix-convex.

\begin{Ex}\label{Ea:Unitary_Similarity}
{\rm 
Consider a scalar function $f(s)$, along with its balanced realization $R_f$
\[
f(s)=\begin{smallmatrix}\frac{s+2}{s+1}\end{smallmatrix}\quad\quad
R_f
=\left({\footnotesize\begin{array}{r|r}-1&1\\ \hline 1&1\end{array}}\right).
\]
Note that $f\in\mathcal{HP}_{\color{blue}\beta}$ with
\mbox{$\begin{smallmatrix}\beta\end{smallmatrix}=
\begin{smallmatrix}\frac{4}{5}\end{smallmatrix}$.}
\smallskip

Now, from the above $R_f$ we construct another function,
\[
g(s)=\begin{smallmatrix}\frac{s}{1-s}\end{smallmatrix}\quad\quad
R_g
:=\left(\begin{smallmatrix}0&1\\1&0\end{smallmatrix}\right)
R_f
\left(\begin{smallmatrix}0&1\\1&0\end{smallmatrix}\right)
=\left({\footnotesize\begin{array}{r|r}1&1\\ \hline 1&-1\end{array}}\right).
\]
By construction $R_f$ is unitarily similar to $R_g$, but $g(s)$ has a pole in
$\C_R$ (not Hurwitz stable), so it can not be a $\mathcal{P}$ function.
}
$\T$
\end{Ex}

Example \ref{Ea:Unitary_Similarity} implies that this set of realization arrays of all
$\mathcal{HP}_{\color{blue}T}$ cannot be $(n+m)$-matrix-convex. This suggests relaxing the requirement and 
looking for a property which is weaker than $(n+m)$-matrix-convexity, but yet more strict than a simple
convexity (which by Proposition \ref{Pn:Sets_Of_Realizations} is already guaranteed).
Thus, we resort to the following, see \cite[Definition 6.3]{Lewk2021a}, \cite[Definition 5.1]{Lewk2020c}
and \cite[Definition 6.7]{Lewk2024b}.

\begin{Dn}\label{Dn:n,mMatrixConvex}
{\rm
For a natural $k$, let
\[
\begin{matrix}
{\scriptstyle{\Upsilon}}_{n_1}~,~\ldots~,~{\scriptstyle{\Upsilon}_{n_k}}&
{\rm of~dimensions}&{\scriptstyle n_1\times\nu}~,~\ldots~,~
{\scriptstyle n_k\times\nu}
\\~\\
{\scriptstyle{\Upsilon}_{m_1}}~,~\ldots~,~{\scriptstyle{\Upsilon}_{m_k}}&
{\rm of~dimensions}&{\scriptstyle m_1\times\mu}~,~\ldots~,~
{\scriptstyle m_k\times\mu}
\end{matrix}
\]
be so that,
\begin{equation}\label{eq:n,mIsometry}
\sum\limits_{j=1}^k\underbrace{\left(\begin{smallmatrix}{\Upsilon}_{n_j}&&0\\~\\0&&{\Upsilon}_{m_j}
\end{smallmatrix}\right)^*}_{{\Upsilon}_j^*}\underbrace{\left(\begin{smallmatrix}{\Upsilon}_{n_j}&&0
\\~\\0&&{\Upsilon}_{m_j}\end{smallmatrix}\right)}_{{\Upsilon}_j}=\left(\begin{smallmatrix}I_{\nu}&&0
\\~\\0&&I_{\mu}\end{smallmatrix}\right).
\end{equation}
Let $\mathbf{R}$ be of \mbox{ $(n+m)\times(n+m)$} matrices. We shall say that $\mathbf{R}$
is~ $n,m$-{\em matrix-convex}~ if
\begin{equation}\label{eq:DefRgBR}
R_F:=\sum\limits_{j=1}^k
\left({\footnotesize\begin{array}{c|c}
{\begin{smallmatrix}{\Upsilon}_{n_j}\end{smallmatrix}}^*A_j{\begin{smallmatrix}{\Upsilon}_{n_j}\end{smallmatrix}}&
{\begin{smallmatrix}{\Upsilon}_{n_j}\end{smallmatrix}}^*B_j{\begin{smallmatrix}{\Upsilon}_{m_j}\end{smallmatrix}}\\
\hline
{\begin{smallmatrix}{\Upsilon}_{m_j}\end{smallmatrix}}^*C_j{\begin{smallmatrix}{\Upsilon}_{n_j}\end{smallmatrix}}&
{\begin{smallmatrix}{\Upsilon}_{m_j}\end{smallmatrix}}^*D_j{\begin{smallmatrix}{\Upsilon}_{m_j}\end{smallmatrix}}
\end{array}}\right),
\end{equation}
belongs to ${\mathbf R}$ for all :
\[
\begin{matrix}
{\scriptstyle{\Upsilon}_{n_1}}~,~\ldots~,~{\scriptstyle{\Upsilon}_{n_k}}~,
&&
{\scriptstyle\nu}\in[1,~\min(n_1,~\ldots~,~n_k)],
\\~\\
{\scriptstyle{\Upsilon}_{m_1}}~,~\ldots~,~{\scriptstyle{\Upsilon}_{m_k}}~,
&&{\scriptstyle\mu}\in[1,~\min(m_1,~\ldots~,~m_k)],
\end{matrix}
\]
and for all natural $k$. 
$\T$
}
\end{Dn}

Note that the notion of
\mbox{$n,m$-{\em matrix-convexity}} is indeed intermediate between (the more strict)
$(n+m)$-{\em matrix-convexity} (see part {\bf a.} of Definition \ref{Dn:n-MatrixConvex}), and the (weaker)
classical convexity.
\bigskip

We now pose the following question:~ For a natural parameter $k$, let 
\mbox{$R_{F_1}~,~\ldots~,~R_{F_k}$} be a collection of \mbox{ $(n+m)\times(n+m)$}
realization of functions $F_1(s)~,~\ldots~,~F_k(s)$ i.e.
\begin{equation}\label{eq:RGj}
R_{F_j}=\left({\footnotesize\begin{array}{l|r}A_j&B_j\\ \hline C_j&D_j\end{array}}\right)
\quad\quad\quad j=1,~\ldots~,~k.
\end{equation}
Assume that $F_1(s)~,~\ldots~,~F_k(s)$ are all in $\mathcal{HP}_{\color{blue}T}~$.
\smallskip

Let the ``matrix" $R_F$ be defined as in Eq.  \eqref{eq:DefRgBR}. Assume now that this $R_F$
is a realization of an \mbox{$m\times m$-valued} rational function $F(s)$. Under what conditions
this $F(s)$ belongs to the same $\mathcal{HP}_{\color{blue}T}$?
\smallskip

To proceed, we need to recall the notion of ``internally passive" realization:
First, in quadratic form, Eq. \eqref{eq:Original_KYP} (the KYP Lemma) was formulated as
having $H\in\mathbf{P}_n$ so that
\[
\left(\begin{smallmatrix}R_F\\~\\I_{n+m}\end{smallmatrix}\right)^*
\left(\begin{smallmatrix}
~0&&0&&-H&&0\\~\\~0&&0&&~0&&I_m\\~\\-H&&0&&~0&&0\\~\\~0&&I_m&&~0&&0\end{smallmatrix}\right)
\left(\begin{smallmatrix}R_F\\~\\I_{n+m}\end{smallmatrix}\right)=
\begin{smallmatrix}Q\end{smallmatrix}\in\overline{\mathbf P}_{n+m}~.
\]
Now, using the same $H\in\mathbf{P}_n$, we construct the $\mathbf{P}_{n+m}$ matrix 
\mbox{$L:=\left(\begin{smallmatrix}H^{-\frac{1}{2}}&&0\\~\\0&&I_m\end{smallmatrix}\right)$,}
and apply the following change of coordinates
\begin{equation}\label{eq:Change_Coordinates}
\left(\begin{smallmatrix}R_F\\~\\I_{n+m}\end{smallmatrix}\right)~\longrightarrow~
\left(\begin{smallmatrix}\hat{R}_F\\~\\I_{n+m}\end{smallmatrix}\right):=
\left(\begin{smallmatrix}L^{-\frac{1}{2}}&&0\\~\\0&&L^{-\frac{1}{2}}\end{smallmatrix}\right)
\left(\begin{smallmatrix}R_F\\~\\I_{n+m}\end{smallmatrix}\right)
\begin{smallmatrix}L\end{smallmatrix},
\end{equation}
results in,
\begin{equation}\label{eq:Interally_Passive_Kyp_P}
\left(\begin{smallmatrix}\hat{R}_F\\~\\I_{n+m}\end{smallmatrix}\right)^*
\left(\begin{smallmatrix}
~0&&0&&-I_n&&0\\~\\~0&&0&&~0&&I_m\\~\\-I_n&&0&&~0&&0\\~\\~0&&I_m&&~0&&0\end{smallmatrix}\right)
\left(\begin{smallmatrix}\hat{R}_F\\~\\I_{n+m}\end{smallmatrix}\right)^*
=\begin{smallmatrix}LQL\end{smallmatrix}\in\overline{\mathbf P}_{n+m}~.
\end{equation}

\begin{Rk}\label{Rk:Internally_Passive_P}
{\rm 
A rational function whose realization satisfies Eq. \eqref{eq:Interally_Passive_Kyp_P}
is called ``internally passive", see e.g.\begin{footnote}{Originally formulated in the framework
of Eq. \eqref{eq:Original_KYP}.}\end{footnote} \cite[Definition 3]{Will1976}.
}
$\T$
\end{Rk}

We next extend the use of this terminology from $\mathcal{P}$ to Hyper-Positive functions: Let $F(s)$ be a
$\mathcal{HP}_{\color{blue}T}$ function. Thus, its realization satisfies Eq. \eqref{eq:HP_Delta_KYP} (or
Eq.  \eqref{eq:M_HP_beta}, in the special case of $\mathcal{HP}_{\color{blue}\beta}$). Recall that the
quadratic form of the characterization-through-realization of $\mathcal{HP}_{\color{blue}\beta}$ functions,
was given in Eq. \eqref{eq:KYP_Quadratic_HP_T}, as
\[
\left(\begin{smallmatrix}R_F\\~\\I_{n+m}\end{smallmatrix}\right)^*
\left(\begin{smallmatrix}
0&&0&&-H&&0\\~\\0&&-{\color{blue}T}&&0&&I_m\\~\\-H&&0&&0&&0\\~\\0&&I_m&&0&&-{\color{blue}T}
\end{smallmatrix}\right)
\left(\begin{smallmatrix}R_F\\~\\I_{n+m}\end{smallmatrix}\right)
=\begin{smallmatrix}Q\end{smallmatrix}\in\overline{\mathbf P}_{n+m}~.
\]
Using the same change of coordinates as in Eq. \eqref{eq:Change_Coordinates}, yields
\[
\left(\begin{smallmatrix}\hat{R}_F\\~\\I_{n+m}\end{smallmatrix}\right)^*
\left(\begin{smallmatrix}
0&&0&&-I_n&&0\\~\\0&&-{\color{blue}T}&&0&&I_m\\~\\-I_n&&0&&0&&0\\~\\0&&I_m&&0&&-{\color{blue}T}
\end{smallmatrix}\right)
\left(\begin{smallmatrix}\hat{R}_F\\~\\I_{n+m}\end{smallmatrix}\right)
=\begin{smallmatrix}LQL\end{smallmatrix}\in\overline{\mathbf P}_{n+m}~.
\]
or in the special case where 
\mbox{${\color{blue}T}=\begin{smallmatrix}{\color{blue}\beta}\end{smallmatrix}I_m$,}
\begin{equation}\label{eq:Internally_Passive_Kyp_HP}
\left(\begin{smallmatrix}\hat{R}_F\\~\\I_{n+m}\end{smallmatrix}\right)^*
\left(\begin{smallmatrix}
0&&0&&-I_n&&0\\~\\0&&-{\color{blue}\beta}I_m&&0&&I_m\\~\\-I_n&&0&&0&&0
\\~\\0&&I_m&&0&&-{\color{blue}\beta}I_m\end{smallmatrix}\right)
\left(\begin{smallmatrix}\hat{R}_F\\~\\I_{n+m}\end{smallmatrix}\right)
=\begin{smallmatrix}LQL\end{smallmatrix}\in\overline{\mathbf P}_{n+m}~.
\end{equation}
And this will be referred to a an ``internally passive" realization of a
$\mathcal{HP}_{\color{blue}T}$ or
$\mathcal{HP}_{\color{blue}\beta}$ function.
\smallskip

We conclude this subsection with a simple observation from \cite{Lewk2024b}.

\begin{Cy}\label{Cy:Balanced_Internally_Passive}
Let $F(s)$ be a function in $\mathcal{HP}_{\color{blue}\beta}$
\mbox{($\begin{smallmatrix}{\color{blue}\beta}\end{smallmatrix}\in[0,~1)$),} and let $R_F$
be a corresponding state-space realization. Then whenever $R_F$ is balanced, it is internally
passive. Namely, in this case $R_F=\hat{R}_F$ and satisfies Eq. \eqref{eq:Internally_Passive_Kyp_HP}.
\end{Cy}

Recall that in \cite[Remark 6.6]{Lewk2024b} it was shown that in fact, a balanced realization is
a (very) special case of internally passive.
\smallskip

\subsection{Proof of Proposition \ref{Pn:Sets_Of_Realizations}} 
This will be done in steps.
\smallskip

{\bf Proof of item (i) of Proposition \ref{Pn:Sets_Of_Realizations}}
\smallskip

To show that this set is convex, substitute Eq. \eqref{eq:KYP_Quadratic_HP_T} in item (ii) of Lemma
\ref{La:Structural_Properties} to obtain \mbox{$q=n+m$,}  and the three following
\mbox{$(n+m)\times(n+m)$} matrices: $E=R_F$,
\mbox{$X=Y=\left(\begin{smallmatrix}0_{n\times n}&~0\\~\\0&-{\color{blue}T}\end{smallmatrix}\right)$}
and \mbox{$V=\left(\begin{smallmatrix}-H&0~\\~\\~0&I_m\end{smallmatrix}\right)$,} with
\mbox{$H\in\mathbf{P}_n$.} Since indeed \mbox{$-X\in\overline{\mathbf P}_{n+m}$,} the claim is established.
\qed
\smallskip

We now return to the problem posed below Eq. \eqref{eq:RGj}. To answer it, one needs to limit the scope in
two ways:\\
(i) The realizations $R_{F_1}$, $\ldots$, $R_{F_k}$ must be all {\em internally passive} and\\
(ii) The weight is essentially scalar, i.e. 
\mbox{${\color{blue}T}=\begin{smallmatrix}{\color{blue}\beta}I_m\end{smallmatrix}$,}
\mbox{$\begin{smallmatrix}{\color{blue}\beta}\end{smallmatrix}\in[0,~1)$,} so the set addressed is
$\mathcal{HP}_{\color{blue}\beta}$.
\smallskip

{\bf Proof of item (ii) of Proposition \ref{Pn:Sets_Of_Realizations}}
\smallskip

We here need to show that this set is {\em $n, m$-matrix-convex}. First, substitute Eq.
\eqref{eq:M_HP_beta} in item (vii) of Lemma \ref{La:Structural_Properties}, to obtain $q=n+m$,
and the three following \mbox{$(n+m)\times(n+m)$} matrices: $E=R_F$,
\mbox{$X=Y=\left(\begin{smallmatrix}0_{n\times n}&~0\\~\\0&-{\color{blue}\beta}I_m\end{smallmatrix}\right)$}
and \mbox{$V=\left(\begin{smallmatrix}-H&0~\\~\\~0&I_m\end{smallmatrix}\right)$,} with
\mbox{$H\in\mathbf{P}_n$.}
\smallskip

Now, the adaptation of the notion of {\em $(n+m)$-matrix-convexity} to {\em $n, m$-matrix-convexity}
is straightforward and thus omitted.
\qed
\smallskip

One can now formulate a technical observation, to be used in
Subsection \ref{subsec:Balanced Truncation and $m, n$-matrix-convexity}.

\begin{La}\label{La:Truncation}
Let $F(s)$ in $\mathcal{HP}_{\color{blue}\beta}$
\mbox{($\begin{smallmatrix}{\color{blue}\beta}\end{smallmatrix}\in[0,~1)$),} 
be a \mbox{$m\times m$-valued}
rational function of McMillan degree $n$. Denote by $R_F$ a corresponding \mbox{$(n+m)\times(n+m)$,} 
{\em internally passive}, realization, i.e. it satisfies Eq. \eqref{eq:Internally_Passive_Kyp_HP}.
\smallskip

Let the
following \mbox{$(n+m)\times(\nu+\mu)$,} with $\nu\in[1, n]$ and $\mu\in[1, m]$, block-diagonal isometry
\begin{equation}\label{eq:Upsilon}
\Upsilon=\left(\begin{smallmatrix}{\upsilon}_n&&0\\~\\0&&{\upsilon}_m\end{smallmatrix}\right)
\quad{\rm where}\quad\begin{smallmatrix}{\upsilon}_n\in\C^{n\times\nu}~~{{\upsilon}_n}^*{\upsilon}_n
=I_{\nu}\\~\\{\upsilon}_m\in\C^{n\times\mu}~~{{\upsilon}_m}^*{\upsilon}_m=I_{\mu}~,\end{smallmatrix}
\end{equation}
be arbitrary.\quad Then, the $(\nu+\mu)\times(\nu+\mu)$ matrix $R_{\hat{F}}$,
\[
R_{\hat{F}}:={\Upsilon}^*R_F{\Upsilon},
\]
may be viewed as an {\em internally passive} realization of $\hat{F}(s)$, a \mbox{$\mu\times\mu$-valued}
rational function, of McMillan degree $\nu$.\\
Furthermore, the resulting $\hat{F}(s)$ is another $\mathcal{HP}_{\color{blue}\beta}$ function with the
same $\begin{smallmatrix}{\color{blue}\beta}\end{smallmatrix}$.
\end{La}

Indeed, let $\Upsilon$ be a $(n+m)\times(\nu+\mu)$ isometry, as in
with $\Upsilon$ as in Eq. \eqref{eq:Upsilon} and consider the following transformation
\[
\left(\begin{smallmatrix}R_F\\~\\I_{n+m}\end{smallmatrix}\right)~\longrightarrow~
\left(\begin{smallmatrix}R_{\hat{F}}\\~\\I_{\nu+\mu}\end{smallmatrix}\right):=
\left(\begin{smallmatrix}{\Upsilon}^*&&0\\~\\0&&{\Upsilon}^*\end{smallmatrix}\right)
\left(\begin{smallmatrix}R_F\\~\\I_{n+m}\end{smallmatrix}\right)
\begin{smallmatrix}\Upsilon\end{smallmatrix}.
\]
Consequently, Eq. \eqref{eq:Internally_Passive_Kyp_HP} is substituted by
\[
\left(\begin{smallmatrix}R_{\hat{F}}\\~\\I_{\nu+\mu}\end{smallmatrix}\right)^*
\left(\begin{smallmatrix}
0&&0&&-I_{\nu}&&0\\~\\0&&-{\color{blue}\beta}I_{\mu}&&0&&I_{\mu}\\~\\-I_{\nu}&&0&&0&&0
\\~\\0&&I_{\mu}&&0&&-{\color{blue}\beta}I_{\mu}\end{smallmatrix}\right)
\left(\begin{smallmatrix}R_{\hat{F}}\\~\\I_{\nu+\mu}\end{smallmatrix}\right)
=\begin{smallmatrix}{\Upsilon}^*Q\Upsilon\end{smallmatrix}\in\overline{\mathbf P}_{\nu+\mu}~,
\]
which is indeed an {\em internally passive} realization of an $\mathcal{HP}_{\color{blue}\beta}$
function $\hat{F}(s)$.

\begin{Ex}\label{Ex:Truncation}
{\rm
Recall that in Example \ref{Ex:HP_T_with_T_singular}
we examined the following $2\times 2$-valued function along with its realization
\[
F(s)=\left(\begin{smallmatrix}\frac{{\gamma}^2(2s+3)}{(s+1)(s+2)}+1&&\frac{{\gamma}^2(2s+3)}{(s+1)(s+2)}
\\~\\\frac{{\gamma}^2(2s+3)}{(s+1)(s+2)}&&\frac{{\gamma}^2(2s+3)}{(s+1)(s+2)}\end{smallmatrix}\right)
\quad\quad\quad R_F={\footnotesize\left(\begin{array}{rr|rr}-1&0&\gamma&\gamma\\ 0&-2&\gamma&\gamma
\\ \hline\gamma&\gamma&1&0\\\gamma&\gamma&0&0\end{array}\right)}, 
\]
where $\begin{smallmatrix}\gamma\end{smallmatrix}\geq 0$ is a parameter. It was shown that, on the one hand
for all 
\mbox{$\begin{smallmatrix}\gamma\end{smallmatrix}\in[0, \begin{smallmatrix}\frac{4}{3}\end{smallmatrix}]$,}
$F\in\mathcal{HP}_{\color{blue}T}$ with \mbox{${\color{blue}T}
=\left(\begin{smallmatrix}\frac{1}{2}&&0\\~\\0&&0\end{smallmatrix}\right)$.} On the other hand, whenever 
\mbox{${\color{blue}T}=\begin{smallmatrix}{\color{blue}\beta}\end{smallmatrix}I_2$,} with
\mbox{$\begin{smallmatrix}{\color{blue}\beta}\end{smallmatrix}>0$,}
this $F(s)$ is in 
\mbox{$\mathcal{P}\smallsetminus\mathcal{HP}_{\color{blue}\beta}$,} 
for all $\begin{smallmatrix}\gamma\end{smallmatrix}\geq 0$.
\smallskip

Following Lemma \ref{La:Truncation}, with the above $R_F$ and the isometry 
\mbox{$\begin{smallmatrix}\Upsilon\end{smallmatrix}=
\left(\begin{smallmatrix}1&0&0\\0&1&0\\0&0&1\\0&0&0\end{smallmatrix}\right)$,}
define \mbox{$R_{\hat{f}}:= \begin{smallmatrix}\Upsilon\end{smallmatrix}^*
R_F\begin{smallmatrix}\Upsilon\end{smallmatrix}$,} so that one obtains a scalar function,
\[
\hat{f}(s)=
\begin{smallmatrix}\frac{{\gamma}^2(2s+3)}{(s+1)(s+2)}+1\end{smallmatrix}
\quad\quad\quad R_{\hat{f}}={\footnotesize\left(\begin{array}{rr|r}-1&0&\gamma\\ 0&-2&\gamma
\\ \hline\gamma&\gamma&1\end{array}\right)}. 
\]
Then, when \mbox{$\begin{smallmatrix}\gamma\end{smallmatrix}=0$,} $R_{\hat{f}}$ is a non-minimal realization of
a zero degree function \mbox{$\hat{f}(s)\equiv 1$.} Now, $\forall \begin{smallmatrix}\gamma\end{smallmatrix}>0$,
from \cite[Corollary 3.3.]{AlpayLew2021} or \cite[Proposition 3.1]{AlpayLew2024}, it follows that
$\hat{f}(s)$ is a $\mathcal{HP}_{\color{blue}\beta}$ function with,
\[
\begin{smallmatrix}\frac{1}{\color{blue}\beta}\end{smallmatrix}
=
\sup\limits_{s\in\C_R}\begin{smallmatrix}\frac{{\hat{f}(s)}^*\hat{f}(s)+1}{\hat{f}(s)+(\hat{f}(s))^*}
\end{smallmatrix}=
\begin{smallmatrix}\frac{{\hat{f}(s)}^*\hat{f}(s)+1}{\hat{f}(s)+(\hat{f}(s))^*}
\end{smallmatrix}_{|_{s=0}}=
\begin{smallmatrix}\frac{1}{2}\end{smallmatrix}(
\begin{smallmatrix}\frac{3}{2}{\gamma}^2+1\end{smallmatrix}+
\begin{smallmatrix}\frac{1}{\frac{3}{2}{\gamma}^2+1}\end{smallmatrix}).
\]
}
$\T$
\end{Ex}
\smallskip

We conclude this subsection by pointing out that the idea of Lemma \ref{La:Truncation} can be carried over
to several  $\mathcal{HP}_{\color{blue}\beta}$ functions (with various ${\scriptstyle\color{blue}\beta}$'s).
Namely, a simultaneous truncation of a matrix-convex combination of whole collection of realization of
functions, while guaranteeing that the resulting function is in a prescribed
$\mathcal{HP}_{\color{blue}\beta}~$.

\begin{Pn}\label{Pn:Matrix-Convex_Combination_of_Realizations}
Let \mbox{$F_1(s),~\ldots~,~F_k(s)$} be \mbox{$m_j\times m_j$-valued} ($j=1,~\ldots~,~k$)
$\mathcal{HP}_{\color{blue}{\beta}_j}$ 
\mbox{($\begin{smallmatrix}{\color{blue}\beta}_j\end{smallmatrix}\in[0,~1)$),} functions, of McMillan
degree $n_1,~\ldots~,~n_k$, respectively. Denote by $R_{F_j}$ corresponding
\mbox{$(n_j+m_j)\times(n_j+m_j)$,} {\em internally passive}, realizations.
\smallskip

For
\[
\nu\in[1, \min(n_1,~\ldots~,~n_k)]\quad
and\quad
\mu\in[1, \min(m_1,~\ldots~,~m_k)],
\]
let the matrices,
\begin{equation}\label{eq:Upsilon_j}
\begin{matrix}
{\upsilon}_{n_j}\in\C^{n_j\times\nu}
\\~\\
{\upsilon}_{m_j}\in\C^{m_j\times\mu}
\end{matrix}
~~\quad where \quad~~
\begin{matrix}
\sum\limits_{j=1}^k
{{\upsilon}_{n_j}}^*{\upsilon}_{n_j}=I_{\nu}
\\~\\
\sum\limits_{j=1}^k
{{\upsilon}_{m_j}}^*{\upsilon}_{m_j}=I_{\mu}
\end{matrix}
\quad\quad
j=1,~\ldots~,~k
\end{equation}
be arbitrary. Then,
\begin{equation}\label{eq:Def_Rg_BR_Again}
R_{\hat{F}}:=\left({\footnotesize\begin{array}{c|c}\sum\limits_{j=1}^k
{\begin{smallmatrix}{\upsilon}_{n_j}\end{smallmatrix}}^*A_j{\begin{smallmatrix}{\upsilon}_{n_j}\end{smallmatrix}}&
\sum\limits_{j=1}^k
{\begin{smallmatrix}{\upsilon}_{n_j}\end{smallmatrix}}^*B_j{\begin{smallmatrix}{\upsilon}_{m_j}\end{smallmatrix}}\\
\hline
\sum\limits_{j=1}^k
{\begin{smallmatrix}{\upsilon}_{m_j}\end{smallmatrix}}^*C_j{\begin{smallmatrix}{\upsilon}_{n_j}\end{smallmatrix}}&
\sum\limits_{j=1}^k
{\begin{smallmatrix}{\upsilon}_{m_j}\end{smallmatrix}}^*D_j{\begin{smallmatrix}{\upsilon}_{m_j}\end{smallmatrix}}
\end{array}}\right),
\end{equation}
is an {\em internally passive} realization of $\hat{F}(s)$, a \mbox{$\mu\times\mu$-valued} rational function,
of McMillan degree $\nu$. Furthermore, the resulting $\hat{F}(s)$ is a $\mathcal{HP}_{\color{blue}\hat{\beta}}$
function with
\[
\begin{smallmatrix}{\color{blue}\hat{\beta}}\end{smallmatrix}\geq\min(
\begin{smallmatrix}{\color{blue}{\beta}_1}\end{smallmatrix}~\ldots~
\begin{smallmatrix}{\color{blue}{\beta}_k}\end{smallmatrix}).
\]
\end{Pn}

This result follows from a straightforward adaption of the proof of item (vii) of Lemma
\ref{La:Structural_Properties} from $(n+m)$-matrx-convexity to $~n, m$-matrix-convexity.

\subsection{Two kinds of Matrix-convexity}

In Subsection \ref{Subsec:Two_Inversions} we showed that for all \mbox{$I_m\succ{\color{blue}T}\succcurlyeq 0$,}
the set $\mathcal{HP}_{\color{blue}T}$ is closed under inversion, in two senses, (i) as $m\times m$-valued
rational function and as (ii) $(n+m)\times(n+m)$ realization arrays. In this subsection, we illustrate how 
this idea can be carried over from inversion to matrix-convexity. To this end, we need to restrict the
discussion to \mbox{${\color{blue}T}={\scriptstyle\color{blue}\beta}I_m$}, see
Proposition \ref{Pn:Matrix-Convex_Combination_of_Realizations}
\smallskip

We next illustrate how, by applying matrix-convex operations, one can construct, from four given 
$\mathcal{HP}_{\color{blue}\beta}$ functions additional ones, with the same 
$\begin{smallmatrix}{\color{blue}\beta}\end{smallmatrix}$. This is done in both setups, of rational
functions and of realization arrays.

\begin{Ex}\label{Ex:Generating_Functions_B}
{\rm
We now return to the four inter-related $\mathcal{HP}_{\color{blue}\beta}$ functions, given just before
Figure \ref{Fig:Another_Point_Impedance_Degree_One}.
\smallskip

{\bf a.}~ First, for rational functions we employ item (v) of Theorem \ref{Tm:Set_Of_HP_Functions}. To
this end, let the scalars \mbox{${\scriptstyle{\Upsilon}_1}$,} $\ldots$, \mbox{${\scriptstyle{\Upsilon}_4}$} 
be so that~ \mbox{${\scriptstyle{\Upsilon}_1}^*{\scriptstyle{\Upsilon}_1}+~\ldots~+
{\scriptstyle{\Upsilon}_4}^*{\scriptstyle{\Upsilon}_4}=1$.} Then, for each quadruple
${\scriptstyle{\Upsilon}_1}$, $\ldots$, ${\scriptstyle{\Upsilon}_4}$, the combination
\begin{center}
$
\begin{smallmatrix}{{\Upsilon}_1}^*\end{smallmatrix}\mathbf{Z}_{\rm in}(s)
\begin{smallmatrix}{\Upsilon}_1\end{smallmatrix}
+\begin{smallmatrix}{{\Upsilon}_2}^*\end{smallmatrix}\mathbf{Y}_{\rm in}(s)
\begin{smallmatrix}{\Upsilon}_2\end{smallmatrix}
+\begin{smallmatrix}{{\Upsilon}_3}^*\end{smallmatrix}{\color{blue}\hat{\mathbf Y}}(s)
\begin{smallmatrix}{\Upsilon}_3\end{smallmatrix}
+\begin{smallmatrix}{{\Upsilon}_4}^*\end{smallmatrix}{\color{red}\hat{\mathbf Z}}(s)
\begin{smallmatrix}{\Upsilon}_4\end{smallmatrix},
$
\end{center}
is yet indeed another $2\times 2$-valued function, within the same family.
\bigskip

{\bf b.}~ Now, for realization arrays, we employ item (ii) of Proposition \ref{Pn:Sets_Of_Realizations}. To
this end, let \mbox{${\scriptstyle{\Upsilon}_1}$,} $\ldots$, \mbox{${\scriptstyle{\Upsilon}_4}$,} be 
$\R^{2\times 2}$ {\em diagonal}, so that~ \mbox{${\scriptstyle{\Upsilon}_1}^*{\scriptstyle{\Upsilon}_1}+~
\ldots~+{\scriptstyle{\Upsilon}_4}^*{\scriptstyle{\Upsilon}_4}=I_2$.} Then, for each quadruple
${\scriptstyle{\Upsilon}_1}$, $\ldots$, ${\scriptstyle{\Upsilon}_4}$,
\begin{center}
\mbox{$
\begin{smallmatrix}{\Upsilon}_1\end{smallmatrix}^*R_{\rm\bf Z}\begin{smallmatrix}{\Upsilon}_1\end{smallmatrix}
+\begin{smallmatrix}{\Upsilon}_2\end{smallmatrix}^*R_{\rm\bf Y}\begin{smallmatrix}{\Upsilon}_2\end{smallmatrix}
+\begin{smallmatrix}{\Upsilon}_3\end{smallmatrix}^*R_{\rm\bf\color{red}\hat Y}\begin{smallmatrix}{\Upsilon}_3
\end{smallmatrix}+\begin{smallmatrix}{\Upsilon}_4\end{smallmatrix}^*R_{\rm\bf\color{red}\bf Z}
\begin{smallmatrix}{\Upsilon}_4\end{smallmatrix}
$}
\end{center} 
is a \mbox{$(2+2)\times(2+2)$} realization (recall, not necessarily minimal) of yet another 
$\mathcal{HP}_{\color{blue}\beta}$ function, with the same 
$\begin{smallmatrix}{\color{blue}\beta}\end{smallmatrix}$.
}
$\T$
\end{Ex}


\section[Balanced-Truncation]{Balanced Truncation of a Convex-hull of Functions}
\label{Sec:Balanced_Truncation}
\setcounter{equation}{0}

In this section we show how can one exploit the above background to simultaneously
obtain a balanced-truncated-reduced-order-model of a whole convex-hull of systems.
This is done for possibly rectangular Hurwitz stable functions and subsequently in
Subsection \ref{Subsec:Balanced_Truncation_HP_beta}, specialized to 
$\mathcal{HP}_{\color{blue}\beta}$ functions.

\subsection{Classical Balanced Truncation Model-Order Reduction}

We start with some classical background. Let $F(s)$ be an $m\times m$-valued Hurwitz stable rational
function, and let $R_F$ be the \mbox{$(n+m)\times(n+m)$} corresponding balanced realization,
\[
F(s)=C(sI_n-A)^{-1}B+D\quad\quad\quad
R_F=\left({\footnotesize\begin{array}{c|c}A&B\\ \hline C&D\end{array}}\right).
\]
Since in balanced realization the Gramians may be chosen to be diagonal, it can always be
written\begin{footnote}{Recall that in Corollary \ref{Cy:Balanced_Internally_Passive} we pointed out that this
implies internal-passivity.}\end{footnote} as
\begin{equation}\label{eq:Gramian=Sigma}
\Sigma{A}^*+A\Sigma=-BB^*\quad\quad\quad\Sigma{A}+A^*\Sigma=C^*C\quad\quad\quad
\Sigma:=\left(\begin{smallmatrix}{\sigma}_1&~   & ~  \\~         &\ddots&~           \\
          &~     &{\sigma}_n\end{smallmatrix}\right),
\end{equation}
where ${\sigma}_j$ are equal to the Hankel singular values of $F(s)$, see e.g.
\cite[Proposition4.10]{DullPaga2000}, \cite[Section 3.9]{Zhou1996}. For a {\em parameter}
\mbox{$\nu\in[1, n]$}, one can order the Hankel singular values as,
\[
{\sigma}_1\geq~\cdots~\geq{\sigma}_{\nu}>{\sigma}_{\nu+1}\geq~\cdots~\geq{\sigma}_n~,
\]
and then partition the balanced Gramian $\Sigma$ as follows,
\[
\Sigma=\left(\begin{smallmatrix}{\Sigma}_{\nu}&&0\\~\\0&&\tilde{\Sigma}\end{smallmatrix}\right)
\quad\quad{\rm where}\quad\quad\quad{\Sigma}_{\nu}=\left(\begin{smallmatrix}
{\sigma}_1&~     & ~            \\
~         &\ddots&~             \\
~         &~     &{\sigma}_{\nu}
\end{smallmatrix}\right)
\quad
\quad
{\rm and}
\quad
\quad
\tilde{\Sigma}
=
\left(\begin{smallmatrix}
{\sigma}_{\nu+1}&~     &~                                \\
~               &\ddots&                                 \\
                &~     &{\sigma}_n\end{smallmatrix}\right).
\]
\smallskip

We next {\em conformably} partition $A$, $B$, $C$ so that Eq. \eqref{eq:Gramian=Sigma} takes the form
\begin{equation}\label{eq:Patitioned_Gramians}
\begin{matrix}
\overbrace{
\left(\begin{smallmatrix}A_{\nu}&&*\\~\\ *&&*\end{smallmatrix}\right)
}^A
\left(\begin{smallmatrix}{\Sigma}_{\nu}&&0\\~\\0&&\tilde{\Sigma}\end{smallmatrix}\right)
+
\left(\begin{smallmatrix}{\Sigma}_{\nu}&&0\\~\\0&&\tilde{\Sigma}\end{smallmatrix}\right)
\overbrace{
\left(\begin{smallmatrix}A_{\nu}&&*\\~\\ *&&*\end{smallmatrix}\right)^*
}^{A^*}
=
-
\overbrace{
\left(\begin{smallmatrix}B_{\nu}\\~\\ *\end{smallmatrix}\right)
}^B
\overbrace{
\left(\begin{smallmatrix}B_{\nu}\\~\\ *\end{smallmatrix}\right)^*
}^{B^*}
\\~\\
\left(\begin{smallmatrix}{\Sigma}_{\nu}&&0\\~\\0&&\tilde{\Sigma}\end{smallmatrix}\right)
\underbrace{
\left(\begin{smallmatrix}A_{\nu}&&*\\~\\ *&&*\end{smallmatrix}\right)
}_A
+
\underbrace{
\left(\begin{smallmatrix}A_{\nu}&&*\\~\\ *&&*\end{smallmatrix}\right)^*
}_{A^*}
\left(\begin{smallmatrix}{\Sigma}_{\nu}&&0\\~\\0&&\tilde{\Sigma}\end{smallmatrix}\right)
=
-
\underbrace{
\left(\begin{smallmatrix}C_{\nu}&&*\end{smallmatrix}\right)^*
}_{C^*}
\underbrace{
\left(\begin{smallmatrix}C_{\nu}&&*\end{smallmatrix}\right).
}_C
\end{matrix}
\end{equation}
Then, the original $F(s)$ of degree $n$, can be {\em approximated} by $\hat{F}(s)$ of degree $\nu$.
Using Eq. \eqref{eq:Patitioned_Gramians}, the \mbox{$(\nu+p)\times(\nu+m)$} balanced
realization of $\hat{F}(s)$
is given by,
\[
R_{\hat F}=\left({\footnotesize\begin{array}{c|c}A_{\nu}&B_{\nu}\\ \hline C_{\nu}&D\end{array}}\right).
\]
For details see e.g. \cite{Moore1981} (the original work) and the books, \cite[Chapter 7]{Antou2005},
\cite[Subsection 4.6.2] {DullPaga2000} and \cite[Chapter 7]{Zhou1996}.

\subsection{Balanced Truncation and $m, n$-matrix-convexity}
\label{subsec:Balanced Truncation and $m, n$-matrix-convexity}

We first cast the above, in the framework of Definition \ref{Dn:n,mMatrixConvex}. To this end, for some
$\nu\in[1,~n]$, consider the following pair of block-diagonal isometries,
\begin{center}
$
\begin{smallmatrix}\Upsilon_{\rm left}\end{smallmatrix}:=
\left(\begin{smallmatrix}I_{\nu}&0\\
0_{(n-\nu)\times\nu}&0\\0&I_p\end{smallmatrix}\right)
\quad\quad
\begin{smallmatrix}\Upsilon_{\rm right}\end{smallmatrix}:=
\left(\begin{smallmatrix}I_{\nu}&0\\
0_{(n-\nu)\times\nu}&0\\0&I_m\end{smallmatrix}\right),
$
\end{center}
i.e.
\begin{center}
$
\begin{smallmatrix}\Upsilon_{\rm left}\end{smallmatrix}^*
\begin{smallmatrix}\Upsilon_{\rm left}\end{smallmatrix}=
\left(\begin{smallmatrix}I_{\nu}&&0\\~\\0&&I_p\end{smallmatrix}\right)
\quad\quad
\begin{smallmatrix}\Upsilon_{\rm right}\end{smallmatrix}^*
\begin{smallmatrix}\Upsilon_{\rm right}\end{smallmatrix}=
\left(\begin{smallmatrix}I_{\nu}&&0\\~\\0&&I_m\end{smallmatrix}\right).
$
\end{center}
Employing Eq. \eqref{eq:Patitioned_Gramians}, we have the following.

\begin{La}\label{La:BalancedTruncation}
Let $R_F=\left({\footnotesize\begin{array} {c|c}A&B\\ \hline C&D\end{array}}\right)$
be a \mbox{$(n+p)\times(n+m)$} balanced realization of a $p\times m$-valued Hurwitz stable
rational function $F(s)$ partitioned as in Eq. \eqref{eq:Patitioned_Gramians}.
\smallskip

Then, $R_{\hat{F}}$, a \mbox{$(\nu+p)\times(\nu+m)$} balanced truncation reduced order model
approximation of $F(s)$, of degree $\nu$, is given by
\[
\underbrace{\left({\footnotesize\begin{array}{c|c}
A_{\nu}&B_{\nu}\\ \hline C_{\nu}&D\end{array}}\right)
}_{
R_{\hat F}}
=
\begin{smallmatrix}\Upsilon_{\rm left}\end{smallmatrix}^*
\underbrace{
\left({\footnotesize\begin{array} {c|c}A&B\\ \hline C&D\end{array}}\right)
}_{R_F}
\begin{smallmatrix}\Upsilon_{\rm right}\end{smallmatrix}.
\]
\end{La}

\begin{Rk}
{\rm
For $\begin{smallmatrix}{\color{blue}\beta}\end{smallmatrix}=0$, i.e. $\mathcal{P}$ functions, item
(ii) of Lemma \ref{La:BalancedTruncation} is classical, see e.g. \cite[Theorem 1]{ChenWen1995},
\cite{DesPal1984} and \cite[Theorem 6.1]{Ober1991}.
}
$\T$
\end{Rk}

With the above background, we can now show that technically,
the operations of taking (i) convex combination and (ii) balanced truncation, commute.

\begin{Pn}\label{Pn:Balanced_Trunction_Convex_Hull}
For some natural $k$, $n$, $m$, $p$, let $F_1(s)$, $\ldots$, $F_k(s)$ be $p\times m$-valued, Hurwitz stable
rational functions, all sharing the same set of $n$ Hankel singular values. Let also, $R_{F_1}$, $\ldots$,
$R_{F_k}$ be a collection of corresponding \mbox{$(n+p)\times(n+m)$} balanced realizations, which for
a prescribed $\nu$, $\nu\in[1,~n]$, are all partitioned as in Eq. \eqref{eq:Patitioned_Gramians}.
\smallskip

Following Lemma \ref{La:BalancedTruncation} let now $R_{\hat{F}_1}$, $\ldots$, $R_{\hat{F}_k}$ be
corresponding, \mbox{$(\nu+m)\times(\nu+m)$} balanced truncation reduced order approximations.
\smallskip

{\rm (i)}~
Then
simultaneously for all $ j=1,~\ldots~k$ 
\[
R_{\hat{F}_j}=\left(\begin{smallmatrix}I_{\nu}&0\\0_{(n-\nu)\times\nu}&0\\0&I_p
\end{smallmatrix}\right)R_{F_j}\left(\begin{smallmatrix}I_{\nu}&0\\
0_{(n-\nu)\times\nu}&0\\0&I_m\end{smallmatrix}\right),
\]
are realization of $\hat{F}_j$, $\ldots$, $\hat{F}_n$, $p\times m$-valued, balanced-truncation approximation
of degree $\nu$.
\smallskip

{\rm (ii)}~
Let $\begin{smallmatrix}{\theta}_j\end{smallmatrix}\geq 0$ be arbitrary so that 
$\sum\limits_{j=1}^k\begin{smallmatrix}{{\theta}_j}\end{smallmatrix}=1$.
Denote by $F_o(s)$ the rational function whose realization is given by~
\begin{equation}\label{eq:F0}
R_{F_0}:=\sum\limits_{j=1}^k\begin{smallmatrix}{{\theta}_j}\end{smallmatrix}R_{F_j}~.
\end{equation}
Let also $R_{\hat{F}_0}$, be a realization of $\hat{F}_0(s)$, the balanced truncation reduced order
approximation of $F_o(s)$ of~ degree $\nu$. Then, $R_{\hat{F}_0}$
is equal to~ $\sum\limits_{j=1}^k\begin{smallmatrix}{{\theta}_j}\end{smallmatrix}R_{\hat{F}_j}~$,
with $R_{\hat{F}_j}$ from item {\rm (i)}.
\end{Pn}

\begin{Rk}
{\rm
{\bf a.}~ 
Proposition \ref{Pn:Balanced_Trunction_Convex_Hull} may be viewed as a scheme for balanced truncation
model order reduction of an uncertain system described by a convex hull of a polytope of realizations,
see Eq. \eqref{eq:F0}. As far as we know, it is original. 
\smallskip

{\bf b.}~ 
Proposition \ref{Pn:Balanced_Trunction_Convex_Hull} is in the spirit of \cite[Corollary 5.1]{CohenLew1997b}.
There, the treatment is more involved, as the simplifying framework of Lemma \ref{La:BalancedTruncation}
was not available.
}
$\T$
\end{Rk}

\subsection[Balanced-Truncation $\mathcal{HP}_{\color{blue}\beta}$ Functions]
{Balanced Truncation of
a Convex-hull of $\mathcal{HP}_{\color{blue}\beta}$ Functions}
\label{Subsec:Balanced_Truncation_HP_beta}

We here restrict the above discussion from Hurwitz stable, possibly rectangular-valued functions,
to $\mathcal{HP}_{\color{blue}\beta}$ functions. to obtain the following.

\begin{Cy}\label{Cy:Balanced_Truncation_HP_beta}
In the framework of Lemma \ref{La:BalancedTruncation} 
let $F(s)$ and $\hat{F}(s)$ be $m\times m$-valued
rational functions. Let $F(s)$ be of McMillan degree $n$ and let $\hat{F}(s)$ be the corresponding
balanced truncation reduced order approximation, of McMillan degree $\nu$.
\smallskip

If the original $F(s)$ belongs to \mbox{${{\mathcal{HP}}_{\color{blue}\beta}}$}, for some
\mbox{$\begin{smallmatrix}{\color{blue}\beta}\end{smallmatrix}\in[0,~1)$,}
then $\hat{F}(s)$ belongs to the same \mbox{${{\mathcal{HP}}_{\color{blue}\beta}}$.}
\end{Cy}

\begin{Cy}\label{Cy:Balanced_Trunction_HP_beta}
In the framework of Proposition \ref{Pn:Balanced_Trunction_Convex_Hull}, let $F_1(s)$, $\ldots$, $F_k(s)$
and $\hat{F}_1(s)$, $\ldots$, $\hat{F}_k(s)$ be $m\times m$-valued rational functions, where $\hat{F}_1(s)$,
$\ldots$, $\hat{F}_k(s)$ are the (degree $\nu$) balanced truncated approximation of the original $F_1(s)$,
$\ldots$, $F_k(s)$ (of degree $n$), respectively.
\smallskip

Let also $F_0(s)$ and $\hat{F}_0(s)$ be a combination (in the sense of Eq. \eqref{eq:F0}), of $F_1(s)$,
$\ldots$, $F_k(s)$ and of $\hat{F}_1(s)$, $\ldots$, $\hat{F}_k(s)$, respectively.
\smallskip

If 
$F_1\in\mathcal{HP}_{{\color{blue}\beta}_1}$,
$\ldots$,
$F_k\in\mathcal{HP}_{{\color{blue}\beta}_k}$
for some 
${\scriptstyle{\color{blue}\beta}_1}$,
$\ldots$,
${\scriptstyle{\color{blue}\beta}_k}\in(0,~1)$, then
$\hat{F}_0\in\mathcal{HP}_{{\color{blue}\beta}_0}$,
where
\mbox{${\scriptstyle{\color{blue}\beta}_0}\geq\min(
{\scriptstyle{\color{blue}\beta}_1}~ \ldots~{\scriptstyle{\color{blue}\beta}_k})$.}
\end{Cy}

We here offer a simple illustration of an application of Corollary 
\ref{Cy:Balanced_Trunction_HP_beta}.

\begin{Ex}
{\rm 
Following Eq. \eqref{eq:F0}, consider $F_0(s)$ as an {\em unknown} element in the
convex hull of given realization arrays $R_{F_1}$, $\ldots$, $R_{F_k}$, where
\begin{equation}\label{eq:Example_Realization}
R_{F_j}=\left({\footnotesize\begin{array}{cc|c}
-\begin{smallmatrix}\frac{1}{5}{{\alpha}_j}^2\end{smallmatrix}&
-\begin{smallmatrix}\frac{2}{11}{\alpha}_j{\delta}_j\end{smallmatrix}
&
\begin{smallmatrix}2{\alpha}_j\end{smallmatrix}
\\
-\begin{smallmatrix}\frac{2}{11}{\alpha}_j{\delta}_j\end{smallmatrix}
&-\begin{smallmatrix}\frac{1}{2}{{\delta}_j}^2\end{smallmatrix}
&\begin{smallmatrix}{\delta}_j\end{smallmatrix}\\
\hline
\begin{smallmatrix}2{\alpha}_j\end{smallmatrix}
&
\begin{smallmatrix}{\delta}_j\end{smallmatrix}&\begin{smallmatrix}d_j\end{smallmatrix}
\end{array}}\right)
\quad\quad
\begin{matrix}
 j=1,~\ldots~,~k\\~\\
0\not={\alpha}_j, {\delta}_j\in\R,~d_j>0~~{\rm parameters}.
\end{matrix}
\end{equation}
Note that in this case, the left-hand side of Eq. \eqref{eq:M_HP_beta} can be given by,
\[
\begin{matrix}
\left(\begin{smallmatrix}-I_2&&0_{2\times 1}\\~\\0_{1\times 2}&&1\end{smallmatrix}\right)
R_{F_j}+{R_{F_j}}^*
\left(\begin{smallmatrix}-I_2&&0_{2\times 1}\\~\\0_{1\times 2}&&1\end{smallmatrix}\right)
&=&
\left(\begin{smallmatrix}\frac{2}{5}{{\alpha}_j}^2&&
\frac{4}{11}{\alpha}_j{\delta}_j
&&0
\\~\\
\frac{4}{11}{\alpha}_j{\delta}_j&&{{\delta}_j}^2&&0\\~\\0&&0&&2d_j\end{smallmatrix}\right)
\\~\\~&=&
\underbrace{
\left(\begin{smallmatrix}\sqrt{\frac{2}{5}}{\alpha}_j\\~\\ \frac{2\sqrt{10}}{11}{\delta}_j\\~\\0
\end{smallmatrix}\right)
\left(\begin{smallmatrix}\sqrt{\frac{2}{5}}{\alpha}_j\\~\\ \frac{2\sqrt{10}}{11}{\delta}_j\\~\\0
\end{smallmatrix}\right)^*
+
\left(\begin{smallmatrix}0&&0&&0\\~\\0&&(\frac{9{\delta}_j}{11})^2&&0\\~\\0&&0&&d_j
\end{smallmatrix}\right)
}_{\in\mathbf{P}_3},
\end{matrix}
\]
so it is indeed positive definite. Thus, there always exists
$\begin{smallmatrix}{\color{blue}{\beta}_j}\end{smallmatrix}>0$, satisfying. Eq.
\eqref{eq:M_HP_beta}. Namely 
$F_1\in\mathcal{HP}_{\color{blue}{\beta}_1}$, $\ldots$,
$F_k\in\mathcal{HP}_{\color{blue}{\beta}_k}$,
and for $F_0(s)$ in Eq. \eqref{eq:F0} one can take \mbox{$\begin{smallmatrix}{\color{blue}{\beta}_0}
\end{smallmatrix}=\min(\begin{smallmatrix}{\color{blue}{\beta}_1}\end{smallmatrix},~\ldots~,~
\begin{smallmatrix}{\color{blue}{\beta}_k}\end{smallmatrix})$.}
\smallskip

To (indirectly) obtain $R_{\hat{F}_0}$, the reduced order approximation of $R_{F_0}$, we resort to
item (ii) of Corollary \ref{Cy:Balanced_Trunction_HP_beta} and recall that
\mbox{$R_{\hat{F}_0}=\sum\limits_{j=1}^k\begin{smallmatrix}{\theta}_j\end{smallmatrix}R_{\hat{F}_j}$},
where $\begin{smallmatrix}{\theta}_1\end{smallmatrix}$, $\ldots$,
$\begin{smallmatrix}{\theta}_k\end{smallmatrix}$ are as in Eq. \eqref{eq:F0}. Moreover, note that the
realizations $R_{F_1}$, $\ldots$, $R_{F_k}$ in Eq. \eqref{eq:Example_Realization} are all balanced:
They all simultaneously satisfy Eq. \eqref{eq:Patitioned_Gramians} with $\nu=1$ so that
$\Sigma=\left(\begin{smallmatrix}10&&0
\\~\\0&&1\end{smallmatrix}\right)$,
namely they all share the same Hankel singular values, 10 and 1. Thus, one can simply take
\begin{equation}\label{eq:Example_approximations}
R_{\hat{f}_j}=\left({\footnotesize\begin{array}{c|c}
-\begin{smallmatrix}\frac{1}{5}{{\alpha}_j}^2\end{smallmatrix}&
\begin{smallmatrix}2{\alpha}_j\end{smallmatrix}
\\
\hline
\begin{smallmatrix}2{\alpha}_j\end{smallmatrix}
&\begin{smallmatrix}d_j\end{smallmatrix}
\end{array}}\right)
\quad\quad
\begin{matrix}
 j=1,~\ldots~,~k\\~\\
0\not={\alpha}_j\in\R,~d_j>0~~{\rm parameters}.
\end{matrix}
\end{equation}
as a realization of a truncated approximation
of $F_j(s)$,
i.e. $\hat{f}_j(s)=\frac{4{{\alpha}_j}^2}{s+\frac{1}{5}{{\alpha}_j}^2}+d_j$.
\smallskip

Finally, as already mentioned, using Eq. \eqref{eq:Example_approximations}, one has that,
\mbox{$R_{\hat{f}_0}=\sum\limits_{j=1}^k\begin{smallmatrix}{\theta}_j\end{smallmatrix}
R_{\hat{f}_j}$,} is a balanced realization of
\mbox{$\hat{f}_0(s)\in\mathcal{HP}_{\color{blue}{\beta}_0}$,}
an approximating function of degree one.
}
$\T$
\end{Ex}

\end{document}